\documentclass[a4paper,reqno,11pt]{amsart}
\usepackage{mathabx}
\usepackage{lmodern}
\usepackage[utf8]{inputenc}
\usepackage{amsfonts}
\usepackage{amsmath}
\usepackage{amssymb}
\usepackage{float}
\usepackage{tikz-cd}
\usepackage{graphicx}
\usepackage{bbm}
\usepackage{amscd}
\usepackage[sort&compress,numbers]{natbib}
\usepackage[left=2cm,right=2cm,bottom=3cm,top=3cm]{geometry}
\usepackage{etoolbox}
\usepackage{filecontents}
\usepackage{enumerate}
\usepackage[hidelinks]{hyperref}
\usepackage{tikz}

\usetikzlibrary{matrix}
\newcommand{\ga}{\gamma}

\newcommand{\eit}{\end{itemize}}

\newcommand{\mc}[1]{\mathcal{#1}}
\newcommand{\mbb}[1]{\mathbb{#1}}
\newcommand{\mr}[1]{\mathrm{#1}}

\newcommand{\mtt}[1]{\mathtt{#1}}

\makeatletter
\patchcmd{\BR@backref}{\newblock}{\newblock(page~}{}{}
\patchcmd{\BR@backref}{\par}{)\par}{}{}
\makeatother
\theoremstyle{plain}
\newtheorem{theorem}{Theorem}[section]
\newtheorem*{theorem*}{Theorem}
\newtheorem{proposition}[theorem]{Proposition}

\newtheorem{corollary}[theorem]{Corollary}

\newtheorem{example}[theorem]{Example}

\newtheorem{lemma}[theorem]{Lemma}

\theoremstyle{definition}
\newtheorem{definition}[theorem]{Definition}

\theoremstyle{remark}

\newtheorem{remark}[theorem]{Remark}
\newtheorem{claim}{Claim}[theorem]
\newtheorem{pclaim}{Claim}[theorem]

\newcommand{\defi}{\begin{definition}}
\newcommand{\fdefi}{\end{definition}}
\newcommand{\eje}{\begin{example}}
\newcommand{\feje}{\end{example}}
\newcommand{\ejes}{\begin{ejemplos}}
\newcommand{\fejes}{\end{ejemplos}}
\newcommand{\lema}{\begin{lemma}}
\newcommand{\flema}{\end{lemma}}
\newcommand{\teor}{\begin{theorem}}
\newcommand{\fteor}{\end{theorem}}
\newcommand{\nota}{\begin{remark}}
\newcommand{\fnota}{ \end{remark}}
\newcommand{\clam}{\begin{claim}}
\newcommand{\fclam}{\end{claim}}
\newcommand{\pclam}{\begin{pclaim}}
\newcommand{\fpclam}{\end{pclaim}}
\newcommand{\clams}{\begin{claim*}}
\newcommand{\fclams}{\end{claim*}}
\newcommand{\prop}{\begin{proposition}}
\newcommand{\fprop}{\end{proposition}}
\newcommand{\cor}{\begin{corollary}}
\newcommand{\fcor}{\end{corollary}}

\newcommand{\lclam}{\begin{lclaim}}
\newcommand{\flclam}{\end{lclaim}}

\newcommand{\ben}{\begin{enumerate}}
\newcommand{\een}{\end{enumerate}}
\newcommand{\bit}{\begin{itemize}}

\DeclareMathOperator{\flim}{\mathrm{FLim}}

\DeclareMathOperator{\Emb}{\mathrm{Emb}}

\DeclareMathOperator{\Aut}{\mathrm{Aut}}

\DeclareMathOperator{\osc}{\mathrm{Osc}}

\newcommand{\C}{{\mathbb C}}

\newcommand{\prue}{\begin{proof}}
\newcommand{\fprue}{\end{proof}}

\newcommand{\de}{\delta}

\DeclareMathOperator{\im}{\text{Im}}
\newcommand{\nrm}[1]{\|#1\|}
\newcommand{\con}{\subseteq}
\newcommand{\vep}{\varepsilon}
\newcommand{\buit}{\emptyset}

\begin{document}
\def\cprime{$'$}

\title{ The Ramsey property for Operator spaces and noncommutative Choquet
simplices}
\author[D. Barto\v{s}ov\'{a}]{Dana Barto\v{s}ov\'{a}}
\address{Department of Mathematics,
University of Florida,
PO Box 118105,
Gainesville, FL 32611.}
\email{dbartoso@andrew.cmu.edu}
\author[J. L{\'{o}}pez-Abad]{Jordi L{\'{o}}pez-Abad}
\address{Departamento de Matem\'{a}ticas Fundamentales, Facultad de
Ciencias, UNED, 28040 Madrid, Spain}
\email{abad@mat.uned.es}
\author[M. Lupini]{Martino Lupini}
\address{School of Mathematics and Statistics\\
Victoria University of Wellington\\
PO Box 600\\
Wellington 6140\\
New Zealand}
\email{martino.lupini@vuw.ac.nz}
\urladdr{http://www.lupini.org/}
\author[B. Mbombo]{B. Mbombo}
\address{Collège Lionel-Groulx, 100 rue Duquet, Sainte-Therese, QC J7E3G6, Canada}
\email{BR.MDempowo@glc.qc.ca}
\address{Departamento de Matem\'atica, Instituto de Matem\'atica e Estat\'istica,
Universidade de S\~ao Paulo, 
rua do Mat\~ao, 1010, 
05508-090 S\~ao Paulo, SP, Brazil}

\thanks{\textcolor{black}{D.B. was supported by the grant FAPESP 2013/14458-9. J.L.-A. \ was
partially supported by the
Ministerio de Econom\'\i a y Competitividad  grant MTM2016-76808 (Spain), the Ministerio de Ciencia e Innovación grant PID2019-107701GB-I00 (Spain) and the Fapesp grants
2013/24827-1 and  2016/25574-8 (Brazil). M.L.\ was partially supported by the NSF Grant
DMS-1600186 and by the Marsden Fund Fast-Start Grant VUW1816 from the Royal
Society of New Zealand. B.M. was supported by Funda\c{c}\~{a}o de
Amparo \`{a} Pesquisa do Estado de S\~{a}o Paulo (FAPESP) postdoctoral
grant, processo 12/20084-1. This work was initiated during a visit of
J.L.-A.\ to the Universidade de Sao P\~{a}ulo in 2014, and continued during
visits of D.B.\ and J.L.-A. to the Fields Institute in the Fall 2014, a
visit of M.L.\ to the Instituto de Ciencias Matem\'{a}ticas in the Spring
2015, and a visit of all the authors at the Banff International Research
Station in occasion of the Workshop on Homogeneous Structures in the Fall
2015. The hospitality of all these institutions is gratefully acknowledged.}}

\begin{abstract}
The noncommutative Gurarij space $\mathbb{\mathbb{\mathbb{NG}}}$, initially
defined by Oikhberg, is a canonical object in the theory of operator spaces.
As the Fra\"{\i}ss\'{e} limit of the class of finite-dimensional nuclear
operator spaces, it can be seen as the noncommutative analogue of the
classical Gurarij Banach space. In this paper, we prove that the
automorphism group of $\mathbb{\mathbb{NG}}$ is extremely amenable, i.e.\
any of its actions on compact spaces has a fixed point. The proof relies on
the Dual Ramsey Theorem, and a version of the Kechris--Pestov--Todorcevic
correspondence in the setting of operator spaces.

Recent work of Davidson and Kennedy, building on previous work of Arveson,
Effros, Farenick, Webster, and Winkler, among others, shows that nuclear
operator systems can be seen as the noncommutative analogue of Choquet
simplices. The analogue of the Poulsen simplex in this context is the matrix
state space $\mathbb{NP}$ of the Fra\"{\i}ss\'{e} limit $A(\mathbb{NP})$ of
the class of finite-dimensional nuclear operator systems. We show that the
canonical action of the automorphism group of $\mathbb{NP}$ on the compact
set $\mathbb{NP}_1$ of unital linear functionals on $A(\mathbb{NP})$ is
minimal and it factors onto any  minimal action, whence providing a
description of the universal minimal flow of \textrm{Aut}$\left( \mathbb{NP}%
\right) $.

\end{abstract}

\maketitle

\section{Introduction}

Dynamics studies, generally speaking, actions of groups on spaces. When the
group $G$ under consideration is \emph{topological}, it is natural to
restrict to actions that are \emph{continuous}. While a classification of
arbitrary continuous actions is hopeless, one can hope for a good structure
theory after restricting to continuous actions \emph{on compact spaces}, also
called \emph{flows} or $G$-flows. In this case, by Zorn's lemma one can
conclude that every flow admits a subflow that is furthermore \emph{minimal }%
(with respect to inclusion). Thus, in some sense flows that are minimal
(i.e.\ have no nontrivial subflows) can be regarded as building blocks of
more general flows. Within the class of minimal flows there is a unique one
that is largest or \emph{universal}, in the sense that it factors onto any
 minimal flow \cite{ellis_universal_1960,gutman_new_2013}. Such a
universal minimal $G$-flow $M(G)$ is thus a canonical object in
the study of the dynamics of $G$, as it encodes fundamental properties of
the class of \emph{all }$G$-flows. For instance, $M(G)$ being
reduced to a single point, in which case the group $G$ is called \emph{%
extremely amenable}, is equivalent to the assertion that every $G$-flow has
a fixed point. More generally, obtaining a concrete description of $M(
G)$ entails at least in principle a classification of all minimal $G$%
-flows, which are precisely the factors of $M(G)$. This is
especially interesting when $M(G)$ turns out to be
\textquotedblleft small\textquotedblright\ or, specifically, \emph{metrizable%
}, in which case any minimal $G$-flow is metrizable as well.

While $M(G)$ is never metrizable when $G$ is locally compact
and not compact, breakthroughs due to Pestov \cite{pestov_dynamics_2006,
pestov_ramsey_2002, pestov_isometry_2007} culminating in the work of
Kechris, Pestov, and Todorcevic \cite{kechris_fraisse_2005} provided several
examples of \textquotedblleft large\textquotedblright\ topological groups for
which $M(G)$ is metrizable, or even trivial. These groups arise
as automorphism groups of mathematical structures satisfying a strong
homogeneity property called \emph{ultrahomogeneity}. Examples of such
structures are the linear order of the rationals $\left( \mathbb{Q},<\right) 
$, the Urysohn metric space $\mathbb{U}$, the infinite random graph $%
\mathcal{R}$, and the countable atomless Boolean algebra $\mathcal{B}$. If $G
$ denotes the automorphism group of one of these examples, then $M(
G) $ is trivial in the case of the rationals and the Urysohn space,
and it is equal to the space of all linear orderings on $\mathcal{R}$ or to
the space of \emph{natural} linear orderings on $\mathcal{B}$, in the case
of the infinite Random graphs and the countable atomless Boolean algebra,
respectively.

In the case of the automorphism group $G$ of a ultrahomogeneous structure $X$%
, the Kechris--Pestov--Todorcevic (KPT) correspondence from \cite%
{kechris_fraisse_2005}, later extended by Van The \cite{van_the_more_2013},
provides a way to compute $M(G)$ by studying the \emph{age }of $%
X$, which is the collection of all the \textquotedblleft
small\textquotedblright\ (finitely-generated) substructures of $X$. For
instance, in the case of the linear order of the rationals, of the Urysohn
space, and the infinite random graph, the age is the class of all the finite
linear orders, all finite metric spaces, and all finite graphs,
respectively. Precisely, the KPT correspondence characterizes extreme
amenability of $G$ in terms of a strong combinatorial property of the age of 
$X$, called the \emph{Ramsey property}, the name being due to the fact that
in the case of finite linear orders it can be seen as a reformulation the
finite Ramsey theorem. This provides a combinatorial way to establish extreme
amenability of $G$ or to compute $M(G)$ after finding a
suitable extremely amenable \textquotedblleft large\textquotedblright\
subgroup. (A different approach, using the method of concentration of
measure, was pioneered by Gromov ad Milman \cite{gromov_topological_1983},
who employed it to prove that the group of unitary operators on the Hilbert
space endowed with the strong operator topology is extremely amenable.)

Until recently, virtually all examples of application of the KPT
correspondence consisted of discrete structures arising in algebra and
combinatorics. This has changed in recent years, where the scope of the KPT
correspondence has been extended to \textquotedblleft
continuous\textquotedblright\ structures from geometry and functional
analysis. One can place in this context the main results of \cite%
{bartosova_2019}, where the Gurarij space and the Poulsen simplex are
studied. The Gurarij space is a classical Banach space constructed by
Gurarij \cite{gurarij_spaces_1966} that can be characterized as the only approximately 
ultrahomogeneous separable Banach space whose age consists of all the
finite-dimensional Banach spaces \cite%
{kubis_proof_2013,ben_yaacov_fraisse_2015}. It is proved in \cite%
{bartosova_2019} that the group $\mathrm{Aut}\left( \mathbb{G}\right) $ of
automorphisms of $\mathbb{G}$ is extremely amenable. Besides the KPT
correspondence, the proof relies on an analysis of the structure of
isometric embedding between finite-dimensional Banach spaces of the form $%
\ell _{\infty }^{n}$. Such embeddings are described in terms of rigid
surjections between finite ordered sets. This makes it possible to apply the
Dual Ramsey Theorem \cite{graham_ramseys_1971} to, ultimately, establish the
(approximate) Ramsey property for the class of finite-dimensional Banach
spaces.

A similar technique is employed in \cite{bartosova_2019} to determined the
universal minimal flow of the group $\mathrm{Aut}\left( \mathbb{P}\right) $
of affine homeomorphisms of the Poulsen simplex \cite{poulsen_simplex_1961}.
The Poulsen simplex $\mathbb{P}$ is the unique Choquet simplex with the
remarkable property that its boundary $\partial \mathbb{P}$ is dense within
the simplex $\mathbb{P}$ itself---in stark contrast with what happens for
the more common \emph{Bauer simplices}, which have closed boundary. One can
describe $\mathbb{P}$ in terms of homogeneous structures by means of the%
\emph{\ Kadison correspondence }\cite{alfsen_compact_1971}, which assigns to
a Choquet simplex $K$ the \emph{function system} $A\left( K\right) $ of
continuous affine scalar-valued functions on $K$. This establishes an
equivalence of categories between compact convex sets and\emph{\ }function
systems, where Choquet simplices correspond to \emph{nuclear }function
systems. The function system $A\left( \mathbb{P}\right) $ corresponding to
the Poulsen simplex is the only (nuclear) function system whose age is the
class of all the finite-dimensional function systems. Relying on this
correspondence, as well as the KPT correspondence and the Dual Ramsey
Theorem, it is proved in \cite{bartosova_2019} that $\mathbb{P}$ itself,
regarded as an $\mathrm{Aut}\left( \mathbb{P}\right) $-flow with the
canonical action of $\mathrm{Aut}\left( \mathbb{P}\right) $, is universal
(and minimal), whence it is the universal minimal $\mathrm{Aut}\left( 
\mathbb{P}\right) $-flow.

In this paper these results are extended to the \emph{noncommutative
analogues }of these objects, which can be constructed in the setting of
operator spaces and operator systems. An operator space $X$ is a complex
vector space endowed with a norm on $K\left( H\right) \otimes X$, where $%
B\left( H\right) $ is the algebra of operators on the separable
infinite-dimensional Hilbert space $H$, and $K\left( H\right) \subseteq
B\left( H\right) $ is the algebra of \emph{compact }operators. Concretely,
separable operator spaces can be thought of as closed subspaces $X$ of a
C*-algebra or, equivalently, of $B\left( H\right) $, endowed with the norm
induced by the inclusion $K\left( H\right) \otimes X\subseteq B\left(
H\right) \otimes B\left( H\right) \subseteq B\left( H\otimes H\right) $.
Every Banach space can be regarded as an operator space, and the operator
spaces that arise in this fashion are precisely those that can be realized
as subspaces of \emph{commutative }C*-algebras. The theory of operator
spaces can be thought of as a noncommutative generalization of the theory of
Banach spaces, and it has applications in the study of C*-algebras and
quantum information theory \cite{palazuelos_survey_2016}.

In much the same way, function systems admit \emph{operator systems} as
noncommutative analogues. An operator system is an operator space that can
be realized as a closed subspace $B\left( H\right) $ that is \emph{unital},
in the sense that it contains the identity operator---the unit. (Naturally,
in this context morphisms are also required to be unital, namely to map the
unit to the unit.) Choquet simplices in turn correspond to the operator
systems that are \emph{nuclear}, which is an approximation property akin to
amenability of groups or Banach algebras. In this context, the Kadison
correspondence between function systems and compact convex sets can be
generalized to a correspondence between compact \emph{matrix convex }sets
and operator systems, which we recall in Section \ref{Subs:operator_systems}%
. Operator systems can thus be thought of as noncommutative analogues of
compact convex sets, and noncommutative Choquet theory in this context has
been recently developed in \cite{davidson_noncommutative_2019} building on 
\cite%
{farenick_extremal_2000,farenick_pure_2004,webster_krein-milman_1999,effros_matrix_1997,effros_matrix_2009,effros_aspects_1978}%
. Operator systems also arise in the study of operator algebras, nonlocal
games, and free real algebraic geometry \cite%
{helton_free_2013,helton_tracial_2017}.

The \emph{noncommutative Gurarij space }$\mathbb{\mathbb{NG}}$ was
constructed by Oikhberg \cite{oikhberg_non-commutative_2006} and can be
characterized as the only (approximately) ultrahomogeneous separable \emph{%
nuclear} operator space whose age contains all the operator spaces that can
be realized as a subspace of a finite-dimensional C*-algebra \cite%
{lupini_uniqueness_2016}. In this paper, we prove that, as in the
commutative case, the group $\mathrm{\mathrm{Aut}}\left( \mathbb{\mathbb{NG}}%
\right) $ of automorphisms of $\mathbb{\mathbb{NG}}$ is extremely amenable.

The \emph{noncommutative Poulsen simplex }$\mathbb{NP}$ is a noncommutative
Choquet simplex (compact matrix convex set) whose corresponding operator
system $A\left( \mathbb{NP}\right) $ of matrix-valued continuous affine
functions is the unique separable nuclear operator system whose age contains
all the operator systems that can be realized as unital subspaces of a
finite-dimensional C*-algebra \cite{lupini_fraisse_2018}. One can also
characterize $A\left( \mathbb{NP}\right) $ as the unique separable nuclear
operator system that is \emph{universal }in the sense of Kirchberg and
Wassermann \cite{kirchberg_c*-algebras_1998,lupini_kirchberg_2018}. This
property can be though of as a noncommutative analogue of having dense
extreme boundary. It is also true that the \emph{matrix extreme }points of $%
\mathbb{NP}$ are dense in $\mathbb{NP}$, although it is unknown if this
property characterizes $\mathbb{NP}$ among the metrizable noncommutative
Choquet simplices. Due to the canonical role that $\mathbb{P}$ plays within
the class of Choquet simplices (as, for instance, it contains any metrizable
Choquet simplex as a proper face), it is natural to expect that $\mathbb{NP}$
will be an example of fundamental importance for noncommutative Choquet
theory \cite{davidson_noncommutative_2019}. We prove in this paper that the
compact space $\mathbb{NP}_{1}$ of $1$-dimensional points of $\mathbb{NP}$%
---which are precisely the unital linear functionals on $A\left( \mathbb{NP}%
\right) $---endowed with the canonical action of the group $\mathrm{\mathrm{%
Au}t}\left( \mathbb{NP}\right) $ of matrix-affine homeomorphisms of $\mathbb{%
NP}$, is the universal minimal $\mathrm{Aut}\left( \mathbb{NP}\right) $%
-flow. This is the natural noncommutative analogue of the corresponding
result for $\mathbb{P}$ from \cite{bartosova_2019}.

The paper is organized as follows. In Section 2 we review some fundamental
notions concerning operator spaces and operator systems. We also define the
notion of Fra\"{\i}ss\'{e} classes, Fra\"{\i}ss\'{e} limits, and the KPT
correspondence restricted to this specific context. In Section 3, we
introduce several nuclear operator spaces as Fra\"{\i}ss\'{e} limits of
classes of finite dimensional \emph{injective} operator spaces, and then in
Section \ref{Subs:ARP} we establish the (approximate) Ramsey property for
each of these classes, obtaining a proof of extreme amenability of $\mathrm{%
Aut}\left( \mathbb{NG}\right) $. In Section \ref{iojiorjiewjwe5}, we
consider the more general case of structures consisting of an operator space
with a distinguished morphism to another (fixed) operator space $R$.
Finally, in Section 4 we consider noncommutative Choquet simplices, operator
systems, operator systems with distinguished state, and we compute the
universal minimal flow of $\mathrm{\mathrm{Aut}}\left( \mathbb{NP}\right) $.

\subsubsection*{Acknowledgments}

We are grateful to Ita\"{\i}, Ben Yaacov, Clinton Conley, Valentin Ferenczi,
Alexander Kechris, Matt Kennedy, Julien Melleray, Lionel Nguyen Van Th\'{e},
Vladimir Pestov, Slawomir Solecki, Stevo Todorcevic, and Todor Tsankov for
several helpful conversations and remarks.

\section{Fra\"{\i}ss\'{e} classes and the Ramsey property of operator spaces
and systems}

\label{basic_notions}

\subsection{Operator spaces and operator systems}

\label{jio3i4jorwe3223}

We now recall some fundamental notions and results from the theory of
operator spaces. The monographs \cite%
{pisier_introduction_2003,effros_operator_2000,blecher_operator_2004}
provide good introductions to this subject. An \emph{operator space} $E$ is
a closed linear subspace of the space $B(H)$ of bounded linear operators on
some complex Hilbert space $H$. The inclusion $E\subset B(H)$ induces \emph{%
matrix norms }on each $M_{n}(E)$, $n\in \mathbb{N}$, the space of $n\times n$
matrices with entries in $E$. The norm of an element $\left[ x_{ij}\right] $
of $M_{n}(E)$ is defined as the operator norm of $\left[ x_{ij}\right] $
when regarded in the canonical way as an linear operator on the $n$-fold
Hilbertian direct sum of $H$ by itself. The \emph{$\infty $-sum} of two
operator spaces $E\subset B(H_{0})$ and $F\subset B(H_{1})$ is the space $%
E\oplus _{\infty }F\subset B(H_{0}\oplus H_{1})$ of operators of the form 
\begin{equation*}
\begin{bmatrix}
x & 0 \\ 
0 & y%
\end{bmatrix}%
\end{equation*}%
for $x\in E$ and $y\in F$. One can equivalently define operator spaces as
the closed subspaces of unital C*-algebras. A unital\emph{\ }C*-algebra is a
closed subalgebra of $B(H)$ containing the identity operator and closed
under taking adjoints. 
Unital C*-algebras can be abstractly characterized as the complex Banach
algebras with multiplicative identity and involution satisfying the
C*-identity $\left\Vert a^{\ast }a\right\Vert =\left\Vert a\right\Vert ^{2}$%
. Operator spaces also admit an abstract characterization, in terms of
Ruan's axioms for the matrix norms \cite[Theorem 13.4]%
{paulsen_completely_2002}. Precisely, a matrix normed complex vector space $%
X $ is an operator space if and only if the matrix norms satisfy the identity%
\begin{equation*}
\left\Vert \alpha _{1}^{\ast }x_{1}\beta _{1}+\cdots +\alpha _{\ell }^{\ast
}x_{\ell }\beta _{\ell }\right\Vert \leq \left\Vert \alpha _{1}^{\ast
}\alpha _{1}+\cdots +\alpha _{\ell }^{\ast }\alpha _{\ell }\right\Vert \max
\left\{ \left\Vert x_{1}\right\Vert ,\ldots ,\left\Vert x_{\ell }\right\Vert
\right\} \left\Vert \beta _{1}^{\ast }\beta _{1}+\cdots +\beta _{\ell
}^{\ast }\beta _{\ell }\right\Vert
\end{equation*}%
for $n_{1},\ldots ,n_{\ell },n\in \mathbb{N}$, $x_{i}\in M_{n_{i}}\left(
X\right) $, and $\alpha _{i},\beta _{i}\in M_{n_{i},n}\left( \mathbb{C}%
\right) $. In this identity, one consider the natural notion of matrix
multiplication between matrices over $X$ and scalar matrices. The norm of
scalar matrices is the operator norm, where matrices are regarded as
operators on finite-dimensional Hilbert spaces.

The abstract characterization of C*-algebras shows that, whenever $K$ is a
compact Hausdorff space, $C(K)$ with the pointwise operations and the
supremum norm is a unital C*-algebra. The unital C*-algebras of this form
are precisely the \emph{abelian} ones. Any complex Banach space $E$ has a
canonical operator space structure, obtained by representing $E$
isometrically as a subspace of $C(\mathrm{Ball}(E^{\ast }))$, where the unit
ball $\mathrm{Ball}(E^{\ast })$ of $E^*$ is endowed with the weak*-topology.
This operator space structure on $E$ is called its \emph{minimal
quantization }\cite{effros_operator_2000} and the corresponding operator
space is denoted by $\mathrm{MIN}(E)$. The matrix norms on $\mathrm{MIN}(E)$
are defined by $\left\Vert \left[ x_{ij}\right] \right\Vert =\sup_{\phi \in 
\mathrm{Ball}(E^{\ast })}\left\Vert \left[ \phi (x_{ij})\right] \right\Vert $
for $\left[ x_{ij}\right] \in M_{n}(E)$. The operator spaces that arise in
this fashion are called \emph{minimal operator spaces}. These are precisely
the operator spaces that can be represented inside an \emph{abelian }unital
C*-algebra. Arbitrary operator spaces can be thought of as the
noncommutative analog of Banach spaces.

If $\phi :E\rightarrow F$ is a linear map between operator spaces, then one
can consider its amplifications $\phi ^{(n)}:M_{n}(E)\rightarrow M_{n}(F)$
obtained by applying $\phi $ entrywise. The \emph{completely bounded norm }$%
\left\Vert \phi \right\Vert _{\mathrm{cb}}$ of $\phi $ is the supremum of $%
\left\Vert \phi ^{(n)}\right\Vert $ for $n\in \mathbb{N}$. A linear map $%
\phi $ is \emph{completely bounded} if $\left\Vert \phi \right\Vert _{%
\mathrm{cb}}$ is finite, and \emph{completely contractive} if $\left\Vert
\phi \right\Vert _{\mathrm{cb}}$ is at most $1$. The \textrm{cb}-distance
between two completely bounded linear maps $\phi ,\psi :E\rightarrow F$ is
defined by $d_{\mathrm{cb}}\left( \phi ,\psi \right) :=\left\Vert \phi -\psi
\right\Vert _{\mathrm{cb}}$. From now on, we regard the space of completely
bounded maps from $E$ to $F$, and all its subspaces, as a metric space
endowed with the \textrm{cb}-distance $d_{\mathrm{cb}}$.

We regard operator spaces as the objects of a category having completely
contractive linear maps as morphisms. An isomorphism in this category is a
surjective linear \emph{complete isometry}, which is an invertible
completely contractive linear map with completely contractive inverse. If $E$
is an operator space, then its automorphism group \textrm{Aut}$\left(
E\right) $ is the group of surjective linear complete isometries from $E$ to
itself. When $E$ is separable, this is a Polish group when endowed with the
topology of pointwise convergence. The \emph{dual operator space} of an
operator space $E$ is a canonical operator space structure on the space $%
E^{\ast }$ of (completely) bounded linear functionals on $E$, obtained by
identifying completely isometrically $M_{n}(E^{\ast })$ with the space of
completely bounded linear maps from $E$ to $M_{n}(\mathbb{C})$; see \cite[\S %
3.2]{effros_operator_2000}.

When $E,F$ are Banach spaces, and $\phi :E\rightarrow F$ is a linear map,
then $\phi $ is bounded if and only if it is completely bounded when $E$ and 
$F$ are endowed with their minimal operator space structure. Furthermore, in
this case one has the equality of norms $\left\Vert \phi :E\rightarrow
F\right\Vert =\left\Vert \phi :\mathrm{MIN}(E)\rightarrow \mathrm{MIN}%
(F)\right\Vert _{\mathrm{cb}}$. Thus the category of Banach spaces and
contractive linear maps can be seen as a full subcategory of the category of
operator spaces and completely contractive linear maps. In particular, the
group of surjective linear isometries of a Banach space $E$ can be
identified with the group of surjective linear complete isometries of $%
\mathrm{MIN}(E)$. We will identify a Banach space $E$ with the corresponding
minimal operator space $\mathrm{MIN}(E)$.

There is a natural class of geometric objects that correspond to operator
spaces, generalizing the correspondence between Banach spaces and compact
absolutely convex sets. A \emph{compact rectangular matrix convex set} in a topological vector space $V$  is a
sequence $(K_{n,m})$ of compact convex subsets of $M_{n,m}(V)$, the $n\times m$-matrices over $V$,   endowed with a notion of
rectangular convex combination. This is an expression $\alpha _{1}^{\ast
}p_{1}\beta _{1}+\cdots +\alpha _{\ell }^{\ast }p_{\ell }\beta _{\ell }$ for 
$p_{i}\in K_{n_{i},m_{i}}$, $\alpha _{i}\in M_{n_{i},n}(\C)$, and $\beta _{i}\in
M_{m_{i},m}(\C)$ satisfying $\left\Vert \alpha _{1}^{\ast }\alpha _{1}+\cdots
+\alpha _{\ell }^{\ast }\alpha _{\ell }\right\Vert \leq 1$ and $\left\Vert
\beta _{1}^{\ast }\beta _{1}+\cdots +\beta _{\ell }^{\ast }\beta _{\ell
}\right\Vert \leq 1$. The notion of an affine map and extreme points admit
natural rectangular matrix analogs, where usual convex combinations are
replaced with rectangular matrix convex combinations. When $E$ is an
operator space, let $\mathrm{C\mathrm{\mathrm{Bal}l}}(E^{\ast })$ be the
sequence $(K_{n,m})$, where  each $K_{n,m}$ is the unit ball of $M_{n,m}(E^{\ast })$. It is proved in 
\cite{fuller_boundary_2018} that any compact rectangular matrix convex set
arises in this way. Furthermore the correspondence $E\mapsto \mathrm{C%
\mathrm{Ball}}(E^{\ast })$ is a contravariant equivalence of categories from
the category of operator spaces and completely contractive maps to the
category of compact rectangular matrix convex sets and continuous
rectangular affine maps.

An \emph{operator system} is a closed linear subspace $X$ of the algebra $B(H)$ for
some Hilbert space $H$ that is \emph{unital} and \emph{self-adjoint}, i.e.\
contains the identity operator and is closed under taking adjoints. In
particular, the space $M_{n}(\mathbb{C})$ has a natural operator system
structure, obtained by identifying $M_{n}(\mathbb{C})$ with $B(\ell
_{2}^{n}) $. An operator system $X$ inherits from the inclusion $X\subset
B(H)$ an involution $x\mapsto x^{\ast }$, which corresponds to taking
adjoints, and a distinguished element $1$ (the $\emph{unit}$), which
corresponds to the identity operator. Furthermore, for every $n\in \mathbb{N}
$, $M_{n}(X)$ has a canonical norm and a notion of \emph{positivity},
obtained by setting $\left[ x_{ij}\right] \geq 0$ if and only if $\left[
x_{ij}\right] $ is positive when regarded as an operator on the $n$-fold
Hilbertian sum of $H$ by itself. The self-adjoint part $X_{\mathrm{sa}}$ of $%
X$ is the unital subspace of $X$ containing those $x\in X$ such that $%
x=x^{\ast }$. A linear map $\phi :X\rightarrow Y$ between operator systems
is \emph{unital} if it maps the unit of $X$ to the unit of $Y$, \emph{%
positive }if it maps positive elements to positive elements, and \emph{%
completely positive} if every amplification $\phi ^{(n)}$ is positive. We
abbreviate \textquotedblleft unital completely positive linear
map\textquotedblright\ as \textquotedblleft \emph{ucp map}\textquotedblright
. It is well known that a unital linear map $\phi $ between operator systems
is completely positive if and only if it is completely contractive. A unital
complete isometry $\phi :X\rightarrow Y$ is called a \emph{complete order
embedding}. A surjective complete order embedding is a \emph{complete order
isomorphism}. One can abstractly characterize the pairs $(X,1)$, where $X$ is
an operator space and $1\in X$, that are operator systems, in the sense that
there exists a complete isometry $\phi :X\rightarrow B(H)$ mapping $1$ to
the identity operator and $X$ onto a closed self-adjoint subspace of $B(H)$ 
\cite{blecher_metric_2011}. An earlier abstract characterization of operator
systems in terms of the matrix positive cones is due to Choi and Effros \cite%
{choi_injectivity_1977}.

\subsection{Fra\"{\i}ss\'{e} classes and Fra\"{\i}ss\'{e} limits}

We recall in this section Fra\"{\i}ss\'{e} classes and Fra\"{\i}ss\'{e}
limits in the setting of operator spaces and operator systems. These can be seen as particular instances of Fra\"{\i}ss\'{e} classes and Fra\"{\i}ss%
\'{e} limits of metric structures in the sense of \cite%
{ben_yaacov_model_2008,melleray_extremely_2014}. In order to make this paper
more self-contained, we will introduce all these notions in this particular
case.

Let $\mathtt{Osp}$ be the class of operator spaces. Given $X,Y\in \mathtt{Osp%
}$ and $\delta \geq 0$, let $\Emb_{\delta }^{\mathtt{Osp}}(X,Y)$ be the
space of $\delta $-embeddings form $X$ into $Y$, that is, injective complete
contractions $\phi :X\rightarrow Y$ such that $\Vert \phi ^{-1}\Vert _{%
\mathrm{cb}}\leq 1+\delta $. In this terminology, the complete isometries
are precisely the $0$-embeddings, which we will simply call embeddings. Let $%
\mathrm{Aut}^{\mathtt{Osp}}(X)$ be the group of surjective embeddings from $X
$ to itself. One can deduce from the small perturbation lemma in operator
space theory \cite[Lemma 2.13.2]{pisier_introduction_2003} that $\Emb%
_{\delta }^{\mathtt{Osp}}(X,Y)$ is a compact metric space whenever $X,Y$ are
finite-dimensional objects of $\mathtt{Osp}$. When $\delta =0$, we will write $%
\Emb^{\mathtt{Osp}}(X,Y)$ instead of $\Emb_{0}^{\mathtt{Osp}}(X,Y)$. The
members of $\Emb^{\mathtt{Osp}}(X,Y)$ are called $\mathtt{Osp}$\emph{%
-embeddings} from $X$ into $Y$. We also have that, when $X$ is separable, $%
\Aut^{\mathtt{Osp}}(X)$ is a Polish group when endowed with the topology of
pointwise convergence. Whenever there is no possibility of confusion we will
use $\Emb_{\delta }(X,Y)$ and $\Aut(X)$ instead of $\Emb_{\delta }^{\mathtt{%
Osp}}(X,Y)$ and $\Aut^{\mathtt{Osp}}(X)$, respectively.

Given an operator space $R$, by an \emph{$R$-operator space} we mean a pair $%
\boldsymbol{X}=(X,s_{X})$, where $X$ is an operator space and $%
s_{X}:X\rightarrow R$ is a complete contraction. Let $\mathtt{Osp}^{R}$ be
the collection of $R$-operator spaces. Given $\boldsymbol{X}=(X,s_{X})$ and $%
\boldsymbol{Y}=(Y,s_{Y})$ in $\mathtt{Osp}^{R}$, and $\delta \geq 0$, let $%
\Emb_{\delta }^{\mathtt{Osp}^{R}}(\boldsymbol{X},\boldsymbol{Y})$ be the
space of $\delta $-embeddings $\phi :X\rightarrow Y$ such that $\Vert
s_{Y}\circ \phi -s_{X}\Vert _{\mathrm{cb}}\leq \delta $. This is a metric
space when we consider the metric $d_{\mathrm{cb}}(\phi ,\psi ):=\Vert \phi
-\psi \Vert _{\text{\textrm{cb}}}$. Again, we will call a $0$-embedding,
simply, an embedding. We let $\Aut^{\mathtt{Osp}^{R}}(\boldsymbol{X})$ be
the group of surjective linear complete isometries $\phi $ from $X$ to
itself such that $s_{Y}\circ \phi =s_{X}$. Note that when $R$ is the trivial
operator space $\{0\}$, $R$-operator spaces can be identified with operator
spaces.

Similarly, we let $\mathtt{Osy}$ be the class of operator systems. Given $%
X,Y\in \mathtt{Osy}$, and $\delta \geq 0$, let $\Emb_{\delta }^{\mathtt{Osy}%
}(X,Y)$ be the collection of all injective \emph{unital} complete
contractions $\phi :X\rightarrow Y$ such that $\left\Vert \phi
^{-1}\right\Vert _{\text{\textrm{cb}}}\leq 1+\delta $. For a fixed operator
system $R$, let $\mathtt{Osy}^{R}$ be the class of \emph{$R$-operator systems%
}, that is pairs $\boldsymbol{X}=(X,s_{X})$ where $X$ is a operator system
and $s_{X}:X\rightarrow R$ is an unital complete contraction. We define $\Emb%
_{\delta }^{\mathtt{Osy}^{R}}(\boldsymbol{X},\boldsymbol{Y})$ to be the
collection of all injective unital complete contractions $\phi :X\rightarrow
Y$ such that $\left\Vert \phi ^{-1}\right\Vert _{\text{\textrm{cb}}}\leq
1+\delta $ and $\left\Vert s_{X}-s_{Y}\circ \phi \right\Vert _{\mathrm{cb}%
}\leq \delta $, endowed with the metric $d_{\text{\textrm{cb}}}\left( \phi
,\psi \right) :=\Vert \phi -\psi \Vert _{\mathrm{cb}}$. We also define $\Aut%
^{\mathtt{Osy}}(X)$ to be the group of unital surjective complete isometries
from $X$ to itself, and $\Aut^{\mathtt{Osy}^{R}}(\boldsymbol{X})$ to be the
group of unital surjective complete isometries $\phi $ from $X$ to itself
such that $s_{X}\circ \phi =s_{X}$. Also, when $\boldsymbol{X}=(X,s_{X})$, $%
\boldsymbol{Y}=(Y,s_{Y})$ are $R$-operator spaces or $R$-operator systems,
we write $\boldsymbol{X}\subseteq \boldsymbol{Y}$ to denote that $X\subseteq
Y$ and $s_{Y}\upharpoonright _{X}=s_{X}$. 

In the next definitions, $\mathtt{C}$ is either $\mathtt{Osp}$, $\mathtt{Osy}
$, $\mathtt{Osp}^{R}$ for a fixed operator space $R$, or $\mathtt{Osy}^{R}$
for a fixed operator system $R$.

\begin{definition}[Gromov-Hausdorff pseudometric]
The \emph{Gromov-Hausdorff pseudometric} $d_{\mathtt{C}}$ is defined by setting, for finite-dimensional $X,Y\in \mathtt{C}$, 
$d_{\mathtt{C}}(X,Y)$ to be the infimum of all $\delta >0$
such that there exist $f\in \Emb_{\delta }^{\mathtt{C}}(X,Y)$ and $g\in \Emb%
_{\delta }^{\mathtt{C}}(Y,X)$ such that $\Vert g\circ f-\mathrm{Id}_{X}\Vert
_{\mathrm{cb}}<\delta $ and $\Vert f\circ g-\mathrm{Id}_{Y}\Vert _{\mathrm{cb%
}}<\delta $.
\end{definition}

It is worth to point out that when $\mathtt{C}$ is the class of operator
spaces, then it is easily seen that for every finite dimensional operator
systems $X$ and $Y$ one has that $d_{\mathtt{C}}(X,Y)\leq d_{\mathrm{BM}%
}(X,Y)\leq 3d_{\mathtt{C}}(X,Y)$, where 
\begin{equation*}
d_{\mathrm{BM}}(X,Y):=\log \left( \inf \{{\Vert T\Vert _{\mathrm{cb}}\Vert
T^{-1}\Vert _{\mathrm{cb}}}\,:\,{T:X\rightarrow Y\text{ is a complete
bounded isomorphism}}\}\right) 
\end{equation*}%
is the well known \emph{Banach-Mazur} pseudometric. It follows that the
class of operator spaces of dimension $n$ has diameter $\leq n$, although
this class is not compact for $n\geq 3$. The reader can find more
information in \cite[Chapter 21]{pisier_introduction_2003}.

In the following, we let $\varpi :\mathbb{R}^{+}\rightarrow \mathbb{R}^{+}$
be a strictly increasing function, continuous at $0$, and vanishing at $0$, such that  $\varpi(\de)\ge \de$.

\begin{definition}[Stable Fraïssé class]
\label{stable_amalg}  Let $\mathcal{A}\subseteq \mathtt{C}$.

\begin{enumerate}[(a)]
\item $\mathcal{A}$ satisfies the \emph{stable amalgamation property (SAP)} with modulus $\varpi$ when
for every $X,Y,Z\in \mathcal{A}$, $\delta \geq 0$, $\varepsilon >0$, $\phi
\in \mathrm{Emb}_{\delta }^{\mathtt{C}}(X,Y)$, and $\psi \in \mathrm{Emb}%
_{\delta }^{\mathtt{C}}(X,Z)$, there exist $V\in \mathcal{A}$, $i\in \mathrm{%
Emb}^{\mathtt{C}}(Y,V)$, and $j\in \mathrm{Emb}^{\mathtt{C}}(Z,V)$ such that 
$\Vert i\circ \phi -j\circ \psi \Vert _{\mathrm{cb}}\leq \varpi (\delta
)+\varepsilon $.

\item $\mc A$ is a {\em stable amalgamation class} with modulus  $\varpi$ when $\mc A$ has the (SAP) and  the  \emph{joint embedding property}, that is,  for every 
$X,Y\in \mathcal{A}$ there exists $Z\in \mathcal{A}$ such that $\mathrm{Emb}%
^{\mathrm{C}}(X,Z)$ and $\mathrm{Emb}^{\mathrm{C}}(Y,Z)$ are nonempty.

\item $\mc A$ is a {\em stable Fraïssé class} with modulus $\varpi$ when it is a stable amalgamation class and $\mathcal{A}$ is separable with respect to the Gromov-Hausdorff
pseudometric $d_{\text{\texttt{C}}}$.

\end{enumerate}
\end{definition}

It is easy to see that if $\mathcal{A}$ has an element that can be embedded
into any other member of $\mathcal{A}$ then the stable amalgamation property
for $\mathcal{A}$ implies the joint embedding property for $\mathcal{A}$.

\begin{definition}[Stable Homogeneity]
Let $\mathcal{A}\subseteq \mathtt{C}$. We say that $M\in \mathtt{C}$
satisfies the stable homogeneity property with respect to $\mathcal{A}$ with
modulus $\varpi $ if:

\begin{enumerate}[(a)]

\item $\Emb^{\mathtt{C}}(X,M)$ is nonempty for every $X\in \mathcal{A}$.

\item For every $X\in \mathcal{A}$, $\delta \geq 0$, $\varepsilon >0$, and $%
f,g\in \Emb_{\delta }^{\mathtt{C}}(X,M)$ there is $\alpha \in \mathrm{Aut}^{%
\mathtt{C}}(M)$ such that $\Vert \alpha \circ g-f\Vert _{\text{\textrm{cb}}%
}\leq \varpi (\delta )+\varepsilon $.
\end{enumerate}

When $\mathcal{A}$ is the collection of all finite-dimensional $X\subseteq M$
in $\mathtt{C}$, we say that $M$ is \emph{stably homogeneous} with modulus $%
\varpi $.
\end{definition}

\begin{definition}
Given a stable Fra\"{\i}ss\'{e} class $\mathcal{A}\subseteq \mathtt{C}$, we
write $[\mathcal{A}]$ to denote the class of $E\in \mathtt{C}$ such that
every finite-dimensional $X\subseteq E$ is a limit, with respect to the
Gromov-Hausdorff distance, of a sequence of \emph{subspaces} of elements in $%
\mathcal{A}$. Let $\langle \mathcal{A}\rangle $ be the collection of all
finite-dimensional elements of $[\mathcal{A}]$.
\end{definition}

Notice that $E\in [\mathcal{A}]$ if and only if for every
finite-dimensional $X\subseteq E$ and every $\delta >0$ there is some $Y\in 
\mathcal{A}$ such that $\Emb_{\delta }^{\mathtt{C}}(X,Y)\neq \emptyset $.

\begin{definition}[Fra\"{\i}ss\'{e} limit]
Let $\mathcal{A}\subseteq \mathtt{C}$. The stable \emph{Fra\"{\i}ss\'{e}
limit} of $\mathcal{A}$ (with modulus $\varpi $), denoted by $\flim\mathcal{A%
}$, if it exists, is the unique separable object in $[\mathcal{A}]$ that is $%
\mathcal{A}$-stably homogeneous (with modulus $\varpi $).
\end{definition}

A usual back-and-forth argument shows the following; see for instance \cite[%
Subsection 2.6]{lupini_fraisse_2018}, \cite[Theorem 2.26]{ferenczi_amalgamation_2017}.

\begin{proposition}
\label{Proposition:exists}Suppose that $\mathcal{A}\subseteq \mathtt{C}$ is
a stable amalgamation class, and $M\in \mathtt{C}$ is separable. Then
the Fra\"{\i}ss\'{e} limit $\flim\mathcal{A}$ exists. Furthermore the
following assertions are equivalent:

\begin{enumerate}[(1)]

\item $M=\flim\mathcal{A}$.

\item $M$ is stably homogeneous with modulus $\varpi $, and the class $%
\mathrm{Age}^\mathtt{C}(M)$ of all finite-dimensional $X\in \mathtt{C}$ such
that $X\subseteq M$ is equal to $\langle\mathcal{A}\rangle$.
\end{enumerate}
\end{proposition}

Notice that whenever $\flim\langle \mathcal{A}\rangle $ exists, $\flim%
\mathcal{A}$ also exists and it must be equal to $\flim\langle \mathcal{A}%
\rangle $. Stable Fra\"{\i}ss\'{e} classes and stable Fra\"{\i}ss\'{e}
limits are in particular Fra\"{\i}ss\'{e} classes and Fra\"{\i}ss\'{e}
limits as \emph{metric structures} in the sense of \cite%
{ben_yaacov_fraisse_2015}. One can realize $\flim\mathcal{A}$ as the limit
of an inductive sequence of elements of $\mathcal{A}$ with \texttt{C}%
-embeddings as connective maps, and it can be proved that every separable
structure in $[\mathcal{A}]$ admits a  $\mathtt{C}$-embedding into $\flim%
\mathcal{A}$.

\textcolor{black}{The nomenclature ``homogeneous'' is related to the concept of ``disposition'' in Banach space theory,  that  for example was used by V. I. Gurarij in \cite{gurarij_spaces_1966}   to define his space (of universal “placement”).  At a midpoint of both notions,  we say that $M\in \mtt C$ is {\em stably of  $\mc A$-disposition} (with modulus $\varpi$) when  $\Emb^\mtt C(X,M)\neq \buit$ for every $X\in \mc A$ and when for every $\de\ge 0$ and $\vep>0$, every $X,Y\in \mc A$ every $f\in \Emb_\de^\mtt C(X, M)$ and $\iota\in \Emb_\de^\mtt C(X,Y)$ there is some $g\in \Emb^\mtt C(Y,M)$ such that $\nrm{g\circ \iota -f}_\mr{cb}\le \varpi(\de)+\vep$. It it proved in \cite[Proposition 2.12]{lupini_fraisse_2018}, implication (6)$
\Rightarrow $(1),  that a separable $M\in [\mc A]$ is the stable Fraïssé limit of $\mc A$ if and only if $M$ is  {stably of  $\mc A$-disposition}.}

Several structures in functional analysis arise as the Fra\"{\i}ss\'{e}
limit of a suitable class $\mathcal{A}$. For example the class of
finite-dimensional operator Hilbert spaces is a Fra\"{\i}ss\'{e} class, and
its corresponding limit is the separable operator Hilbert space $\mathrm{OH}$
introduced and studied in \cite{pisier_operator_1996}. Another natural
example of a family with the stable amalgamation property is the collection
of finite-dimensional Banach spaces $\{\ell _{p}^{n}\}_{n\geq 0}$ for every $%
p\in (1,+\infty )$. In the case $p=2$ one can use the polar decomposition
for bounded operators on a Hilbert space to deduce that every $\delta $%
-embedding between Hilbert spaces is close to an embedding. The other cases
are treated in \cite{ferenczi_amalgamation_2017}. In this case one uses a
result by Schechtman in \cite{Schechtman_1979} stating that for every such $%
p\neq 2$ there exists a function $\varpi _{p}:\mathbb{R}_{+}\rightarrow 
\mathbb{R}_{+}$ continuous at $0$ and vanishing at $0$, with the property
that if $\phi :\ell _{p}^{k}\rightarrow \ell _{p}^{m}$ is a $\delta $%
-embedding for some $\delta >0$, then there exist $n\in \mathbb{N}$, $I\in %
\Emb(\ell _{p}^{m},\ell _{p}^{n})$, and $J\in \Emb(\ell _{p}^{k},\ell
_{p}^{n})$ such that $\Vert J-I\circ \phi \Vert \leq \varpi _{p}(\delta )$.
The corresponding Fra\"{\i}ss\'{e} limit $\flim\{\ell _{p}^{n}\}_{n}$ of $%
\{\ell _{p}^{n}\}_{n}$ is the Lebesgue space $L_{p}[0,1]$. When $p$ is an
even integer other than $2$, the space $L_{p}[0,1]$ is not stably
homogeneous or, equivalently, the class $\langle \{\ell
_{p}^{n}\}_{n}\rangle $ does not have the stable amalgamation property. In
fact, in this case, $L_{p}[0,1]$ is not even \emph{approximately
ultrahomogeneous }as shown in \cite{lusky_consequences_1978}; see also \cite%
{ferenczi_amalgamation_2017}. An operator space $M$ is approximately
ultrahomogeneous when for every finite-dimensional $X\subseteq M$, every
complete isometry $\phi :X\rightarrow M$, and every $\varepsilon >0$ there
is a surjective linear complete isometry $\alpha :M\rightarrow M$ such that $%
\Vert \alpha \upharpoonright _{X}-\phi \Vert _{\mathrm{cb}}\leq \varepsilon $%
. Obviously, stably homogeneous spaces are approximately ultrahomogeneous.
Lusky proved in \cite{lusky_consequences_1978} that the space $L_{p}[0,1]$
is approximately ultrahomogeneous when $p\in (1,+\infty )$ is not an even
integer. This has been recently improved on \cite[Theorem 4.1]%
{ferenczi_amalgamation_2017} where it is shown that the spaces $L_{p}[0,1]$
for those $p$'s are \textquotedblleft quasi\textquotedblright\ stably
homogeneous in the sense that there is a modulus of stability depending on
dimensions $\widetilde{\varpi }:\mathbb{N}\times [0,\infty )
\rightarrow [0,\infty )$ such that for every $\delta \geq 0$, $%
\varepsilon >0$, every finite dimensional subspace $X\subseteq L_{p}[0,1]$,
and $f,g\in \mathrm{Emb}_{\delta }(X,L_{p}[0,1])$ there exists $\alpha \in 
\mathrm{Aut}(L_{p}[0,1])$ such that $\Vert \alpha \circ g-f\Vert _{\mathrm{cb%
}}\leq \varpi (\dim X,\delta )+\varepsilon $.

The Ramsey property of the following classes, proved to be Fra\"{\i}ss\'{e}
in \cite{lupini_fraisse_2018}, are the main subject of the present paper. Recall that an operator space $X$ is \emph{injective} if for every operator
spaces $Y\subseteq Z$, every complete contraction from $Y$ to $X$ can be
extended to a complete contraction from $Z$ to $X$. One defines injective
operator systems similarly, by replacing complete contractions with unital
complete contractions.

\begin{theorem}
\label{injective_classes_are_fraisse}Let $\mathtt{C}$ be either $\mathtt{Osp}
$ or $\mathtt{Osy}$. Suppose that $\mathcal{I}$ is a countable class of
finite-dimensional injective elements of $\mathtt{C}$ such that for $X,Y\in 
\mathcal{I}$, the $\infty $-sum $X\oplus _{\infty}Y$ embeds into an element
of $\mathcal{I}$. Then $\mathcal{I}$ is a stable amalgamation class and $\langle \mathcal{I}\rangle $ is a
stable Fra\"{\i}ss\'{e} class, both with stability modulus $\varpi (\delta
)=\delta $ if $\mathtt{C}$ is $\mathtt{Osp}$, and $\varpi \left( \delta
\right) =2\delta $ if $\mathtt{C}$ is $\mathtt{Osy}$.
\end{theorem}

For a class $\mc A$ of operator spaces, and an operator space $R$, we let $\mc A^R$ be the class of $R$-operator spaces of the form $(X,s)$ where $X\in \mc A$. There is a close relation between the stable amalgamation property of $\mc A$ and the one of $\mc A^R$. Similar considerations and notation apply in the case of operator systems.

\begin{proposition}
\label{oij32irj23jr32} Let $\mathtt{C}$ be either $\mathtt{Osp}$ or $\mathtt{%
Osy}$. Suppose that $\mathcal{A}$ is a class of finite-dimensional injective
elements of $\mathtt{C}$ that has the stable amalgamation property with
modulus $\varpi (\delta )$. Suppose that $R\in \mathtt{C}$ is separable and
such that every element in $\mathrm{Age}^{\mathtt{C}}(R)$ embeds into an
element of $\mathcal{A}$.

If either: (a) for $X,Y\in \mathcal{A}$, the $\infty $-sum $X\oplus _{\infty
}Y$ belongs to $\mathcal{A}$, or (b) $R$ is injective, and for $X,Y\in 
\mathcal{A}$, the $\infty $-sum $X\oplus _{\infty }Y$ has a $\mathtt{C}$%
-embedding into an element of $\mathcal{A}$, then the class $\mathcal{A}^{R}$
satisfies the stable amalgamation property with modulus $\varpi (\delta
)+\delta $.

Moreover, {suppose that each element of $\mathcal{A}$ is injective, and $R$ is $\mathcal{A}-$\emph{nuclear}, that is, the
identity map of $R$ is the pointwise limit of $\mathtt{C}$-morphisms that
factor through elements of $\mathcal{A}$; see \cite[Definition 3.4]%
{lupini_fraisse_2018}. Then the Fra\"{i}ss\'{e} limit $\flim\mathcal{A}^{R}$ is $(%
\flim\mathcal{A},\Omega )$} for an appropriate complete contraction $\Omega: \flim\mathcal{A}\to R$, that when $\mtt C=\mtt{Osy}$ is in addition unital (see Proposition \ref{Proposition:fraisse-functional} {\it(3)} and Proposition \ref{Proposition:fraisse-state} {\it(3)} for more details).
\end{proposition}

\begin{proof}
We suppose that (a) holds. The proof when (b) holds is similar. Let $\delta
\geq 0$, $\varepsilon >0$, $\boldsymbol{X}=(X,s_{X})$, $\boldsymbol{Y}%
=(Y,s_{Y})$ and $\boldsymbol{Z}=(Z,s_{Z})$ all in $\mathcal{A}^{R}$, and let 
$\phi \in \Emb_{\delta }^{\mathtt{C}^{R}}(\mathbf{X},\mathbf{Z})$ and $\psi
\in \Emb_{\delta }^{\mathtt{C}^{R}}(\mathbf{X},\mathbf{Y})$. By definition, $%
\phi \in \Emb_{\delta }^{\mathtt{C}}(X,Z)$ and $\psi \in \Emb_{\delta }^{%
\mathtt{C}}(X,Y)$, so it follows from the (SAP) of the class $\mathcal{A}$
with modulus $\varpi $ that there is some $V\in \mathcal{A}$, $i\in \Emb^{%
\mathtt{C}}(Y,V)$ and $j\in \Emb^{\mathtt{C}}(Z,V)$ such that $\Vert i\circ
\phi -j\circ \psi \Vert _{\mathrm{cb}}\leq \varpi (\delta )+\varepsilon $.
Let $R_{0}\in \mathcal{A}$, and let $\theta \in \Emb^{\mathtt{C}}(S,R_{0})$
where $S\in \mathrm{Age}^{\mathtt{C}}(R)$ is generated by $\im s_{Y}+\im %
s_{Z}$. Set, $W:=V\oplus _{\infty }R_{0}\in \mathcal{A}$. We define $\mathbf{%
W}:=(W,\pi _{2})\in \mathcal{A}^{R}$, where $\pi _{2}:V\oplus _{\infty
}R_{0}\rightarrow R_{0}\subseteq R$ is the canonical second-coordinate
projection. Let also $I:Y\rightarrow W$ and $J:Z\rightarrow W$ be defined by 
$I(y):=(i(y),\theta (s_{Y}(y)))$ and $J(z):=(j(z),\theta (s_{Z}(z)))$. It
follows that $I\in \Emb^{\mathtt{C}^{R}}(\mathbf{Y},\mathbf{W})$ and that $%
J\in \Emb^{\mathtt{C}^{R}}(\mathbf{Z},\mathbf{W})$. From definitions, we
have that $\Vert I\circ \phi -J\circ \psi \Vert _{\mathrm{cb}}=\max \{\Vert
i\circ \phi -j\circ \psi \Vert _{\mathrm{cb}},\Vert \theta \circ s_{Y}\circ
\phi -\theta \circ s_{Z}\circ \psi \Vert _{\mathrm{cb}}\}\leq \max \{\varpi
(\delta )+\varepsilon ,\Vert s_{Y}\circ \phi -s_{X}\Vert _{\mathrm{cb}%
}+\Vert s_{Z}\circ \psi -s_{X}\Vert _{\mathrm{cb}}\}\leq \varpi (\delta
)+\delta +\varepsilon $, as we are assuming that $\varpi $ satisfies $\varpi
(\delta )\geq \delta $.

Suppose now that each element of $\mathcal{A}$ is injective, and that $R$ is $\mathcal{A}$-nuclear. {We consider the case of
operator spaces. The case of operator systems is entirely similar. Write $\flim\mathcal{A}^{R}=(\mathbb{M},\Omega )$. We now show that }$%
\mathbb{M}=${$\flim\mathcal{A}$}. To this purpose, 
as we mentioned before, it
suffices to prove that the following property of approximate local {\em disposition}  holds: for every $\varepsilon >0$, $%
\delta \geq 0$, $F\in \mathcal{A}$, tuple $\overline{a}=\left( a_{0},\ldots
,a_{n}\right) $ in $F$, injective completely contractive map $f:\mathrm{%
\mathrm{span}}\left( \overline{a}\right) \rightarrow \mathbb{M}$ such that $%
\left\Vert f^{-1}\right\Vert _{\mathrm{cb}}\leq \delta $, there exists a
completely isometric linear map $g:F\rightarrow \mathbb{M}$ such that $%
\max_{i}\left\Vert g\left( a_{i}\right) -f\left( a_{i}\right) \right\Vert
\leq \varpi \left( \delta \right) +\delta +\varepsilon $.

Consider the completely contractive map $\Omega \circ f:\mathrm{\mathrm{span}%
}\left( \overline{a}\right) \rightarrow R$. We are assuming that $R$ is $%
\mathcal{A}$-nuclear.
 By the equivalence of (1) and (3) in \cite[%
Proposition 3.5]{lupini_fraisse_2018}, there exists a completely contractive
map $s:F\rightarrow R$ such that $\max_{i}\left\Vert s\left( a_{i}\right)
-\left( \Omega \circ f\right) \left( a_{i}\right) \right\Vert \leq
\varepsilon $. By the homogeneity property of {$\flim\mathcal{A}^{R}=(%
\mathbb{M},\Omega )$, there exists a }$\mtt{C}^{R}$-embedding $g:\left( F,s\right)
\rightarrow \left( \mathbb{M},\omega \right) $ such that $\left\Vert g\left(
a_{i}\right) -f\left( a_{i}\right) \right\Vert \leq \varpi \left( \delta
\right) +\delta +\varepsilon $. This concludes the proof.
\end{proof}

\subsection{The approximate Ramsey property}

We now introduce and characterize various version of the Ramsey property,
which are strengthening of the amalgamation property and that, as we will
see, can be used to obtain information about the automorphism group of Fra%
\"{\i}ss\'{e} limits. We still adopt the notation from above. 

\begin{definition}[Approximate and stable Ramsey Property]
\label{Definition:ARP} Let  $\mathtt{C}$ be one of the classes $\mathtt{Osp}$, $\mathtt{Osp}^{R}$, $%
\mathtt{Osy}$ or $\mathtt{Osy}^{R}$, and let  $\mathcal{A}\subseteq 
\mathtt{C}$. 

\begin{enumerate}[(a)]

\item $\mathcal{A}$ satisfies the \emph{approximate Ramsey property (ARP)}
if for every $X,Y\in \mathcal{A}$, $\varepsilon >0$ there exists $Z\in 
\mathcal{A}$ such that any \emph{continuous coloring} of $\mathrm{Emb}^{%
\mathtt{C}}(X,Z)$ \emph{$\varepsilon $-stabilizes} on $\gamma \circ \mathrm{%
Emb}^{\mathtt{C}}(X,Y)$ for some $\gamma \in \mathrm{Emb}^{\mathtt{C}}(Y,Z)$%
; that is, for every 1-Lipschitz mapping $c:\Emb^{\mathtt{C}%
}(X,Z)\rightarrow [0,1]$ there is some $\gamma \in \Emb^{\mathtt{C}%
}(Y,Z)$ such that $\osc(c\upharpoonright (\gamma \circ \Emb^{\mathtt{C}%
}(X,Y)))=\max_{\eta _{0},\eta _{1}\in \Emb^{\mathtt{C}}(X,Y)}|c(\gamma \circ
\eta _{0})-c(\gamma \circ \eta _{1})|\leq \varepsilon $.

\item $\mathcal{A}$ satisfies the \emph{stable Ramsey property (SRP)} with
stability modulus $\varpi $ if for every $X,Y\in \mathcal{A}$, $\varepsilon
>0$, $\delta \geq 0$ there exists $Z\in \mathcal{A}$ such that every
continuous coloring of $\mathrm{\ \mathrm{Emb}}_{\delta }^{\mathtt{C}}(X,Z)$ 
$(\varepsilon +\varpi (\delta ))$-stabilizes on $\gamma \circ \mathrm{Emb}%
_{\delta }^{\mathtt{C}}(X,Y)$ for some $\gamma \in \mathrm{Emb}^{\mathtt{C}%
}(Y,Z)$.

\item $\mathcal{A}$ satisfies the \emph{discrete approximate Ramsey property}%
, or discrete (ARP), if for every $X,Y\in \mathcal{A}$, $r\in \mathbb{N}$
and $\varepsilon >0$ there exists $Z\in \mathcal{A}$ such that every \emph{$r
$-coloring} of $\mathrm{Emb}^{\mathtt{C}}(X,Z)$ has an \emph{$\varepsilon $%
-monochromatic} set of the form $\gamma \circ \mathrm{Emb}^{\mathtt{C}}(X,Y)$
for some $\gamma \in \mathrm{Emb}^{\mathtt{C}}(Y,Z)$; that is, for every
coloring $c:\Emb^{\mathtt{C}}(X,Z)\rightarrow r=\{0,1,\dots ,r-1\}$ there is
some $\gamma \in \Emb^{\mathtt{C}}(Y,Z)$ and some $j\in r$ such that $\gamma
\circ \Emb^{\mathtt{C}}(X,Y)\subseteq (c^{-1}(j))_{\varepsilon }$. The \emph{%
discrete }stable Ramsey property, or discrete (SRP), is defined similarly.

\item The \emph{compact }(ARP) and the \emph{compact }(SRP) of $\mathcal{A}$
are defined as the (ARP) and the (SRP), respectively, by replacing
continuous colorings with compact colorings, i.e. 1-Lipschitz mappings into
compact metric spaces.
\end{enumerate}
\end{definition}

It is not difficult to see that the (SRP) with modulus $\varpi $ of a class
implies the stable amalgamation property of the class with modulus $\varpi $%
. Also, it is worth to point out that the (ARP) as in Definition \ref%
{Definition:ARP} is equivalent to the one considered in \cite[Definition 3.3]%
{melleray_extremely_2014} when $R$-operator spaces or systems are regarded
as structures in the logic for metric structures \cite{ben_yaacov_model_2008}
as in \cite[Appendix B]{goldbring_kirchbergs_2015} or \cite[\S 8.1]%
{lupini_fraisse_2018}. The following proposition provides reformulations of
the (ARP) in terms of discrete or compact colorings, and it is a
generalization of \cite[Proposition 2.7]{bartosova_2019} where the case of
Banach spaces is treated.

\begin{proposition} \label{oijweirjewiorwerew} Let  $\mathtt{C}$ be one of the classes $\mathtt{Osp}$, $\mathtt{Osp}^{R}$, $%
\mathtt{Osy}$ or $\mathtt{Osy}^{R}$. 
\label{ARP=DARP} The following are equivalent for a class $\mathcal{A}%
\subseteq \mathtt{C}$:

\begin{enumerate}[(1)]

\item $\mathcal{A}$ satisfies the (ARP).

\item $\mathcal{A}$ satisfies the \emph{discrete} (ARP).

\item $\mathcal{A}$ satisfies the \emph{compact} (ARP).
\end{enumerate}
\end{proposition}

\begin{proof}
The compact (ARP) obviously implies the (ARP). Suppose that $\mathcal{A}$
satisfies the (ARP), and let us prove that $\mathcal{A}$ satisfies the
discrete (ARP). This is done by induction on $r\in \mathbb{N}$. The case $r=1
$ is trivial. Suppose that we have shown that $\mathcal{A}$ satisfies the
discrete (ARP) for $r$-colorings. Consider $X,Y\in \mathcal{A}$ and $%
\varepsilon >0$. Then by the inductive hypothesis, there is $Z_{0}\in 
\mathcal{A}$ such that every $r$-coloring of $\mathrm{\mathrm{Emb}}^{\mathtt{%
C}}\left( X,Z_{0}\right) $ $\varepsilon $-stabilizes on $\gamma \circ 
\mathrm{Emb}^{\mathtt{C}}\left( X,Y\right) $ for some $\gamma \in \mathrm{Emb%
}^{\mathtt{C}}\left( Y,Z_{0}\right) $. Since by the assumption $\mathcal{A}$
satisfies the (ARP), there is $Z\in \mathcal{A}$ such that every continuous
coloring of $\mathrm{Emb}^{\mathtt{C}}\left( X,Z\right) $ $\varepsilon /2$%
-stabilizes on $\gamma \circ \mathrm{Emb}^{\mathtt{C}}\left( X,Z_{0}\right) $
for some $\gamma \in \mathrm{Emb}^{\mathtt{C}}\left( Z_{0},Z\right) $. We
claim that $Z$ witnesses that $\mathcal{A}$ satisfies the discrete (ARP) for 
$\left( r+1\right) $-colorings. Indeed, suppose that $c$ is an $\left(
r+1\right) $-coloring of $\mathrm{Emb}^{\mathtt{C}}\left( X,Z\right) $.
Define $f:\mathrm{Emb}^{\mathtt{C}}\left( X,Z\right) \rightarrow \left[ 0,1%
\right] $ by $f\left( \phi \right) :=\frac{1}{2}d_{\mathrm{cb}}\left( \phi
,c^{-1}\left( r\right) \right) $. This is a continuous coloring, so by the
choice of $Z$ there exists $\gamma \in \mathrm{Emb}^{\mathtt{C}}(Z_{0},Z)$
such that $f$ $\varepsilon /2$-stabilizes on $\gamma \circ \mathrm{Emb}^{%
\mathtt{C}}(X,Z_{0})$. Now, if there is some $\phi \in \mathrm{Emb}^{\mathtt{%
C}}(X,Z_{0})$ such that $c(\gamma \circ \phi )=r$, then $\gamma \circ 
\mathrm{Emb}^{\mathtt{C}}(X,Z_{0})\subseteq (c^{-1}(r))_{\varepsilon }$, so
choosing an arbitrary $\bar{\gamma}\in \mathrm{Emb}^{\mathtt{C}}(Y,Z_{0})$
we obtain that $c$ $\varepsilon $-stabilizes on $\gamma \circ \bar{\gamma}%
\circ \mathrm{Emb}^{\mathtt{C}}(X,Y)$. Otherwise, $(\gamma \circ \mathrm{Emb}%
^{\mathtt{C}}(X,Z_{0}))\cap c^{-1}(r)=\emptyset $, so defining $\bar{c}(\phi
):=c\left( \gamma \circ \phi \right) $ for $\phi \in \mathrm{Emb}^{\mathtt{C}%
}(X,Z_{0})$ gives an $r$-coloring of $\mathrm{Emb}^{\mathtt{C}}(X,Z_{0})$.
By the choice of $Z_{0}$ there exists $\bar{\gamma}\in \mathrm{Emb}^{\mathtt{%
C}}(Y,Z_{0})$ such that $\bar{c}$ $\varepsilon $-stabilizes on $\bar{\gamma}%
\circ \mathrm{Emb}^{\mathtt{C}}(X,Y)$. Therefore $c$ $\varepsilon $%
-stabilizes on $\gamma \circ \bar{\gamma}\circ \mathrm{Emb}^{\mathtt{C}%
}(X,Y) $. This concludes the proof that the (ARP) implies the discrete (ARP).

Finally, the discrete (ARP) implies the compact (ARP). Indeed, given $%
\varepsilon >0$ and a compact metric space $K$, one can find a finite $%
\varepsilon $-dense subset $D\subseteq K$. Thus if $Z\in \mathcal{A}$
witnesses the discrete (ARP) for $X,Y$, $\varepsilon $ and $D$, then given a 
$1$-Lipschitz $f:\mathrm{Emb}^{\mathtt{C}}\rightarrow K$ we can define a
coloring $c:\mathrm{Emb}^{\mathtt{C}}(X,Z)\rightarrow D\subseteq K$ such
that $d_{K}(c(\phi ),f(\phi ))\leq \varepsilon $ for every $\phi \in \mathrm{%
Emb}^{\mathtt{C}}(X,Z)$. In this way, if $c$ $\varepsilon $-stabilizes on $%
\gamma \circ \mathrm{Emb}^{\mathtt{C}}(X,Y)$, then $f$ $3\varepsilon $%
-stabilizes on $\gamma \circ \mathrm{Emb}^{\mathtt{C}}(X,Y)$.
\end{proof}

The following is a useful property of classes with the (SAP). 
It can be easily proved by induction, using the fact that the spaces $%
\mathrm{Emb}_{\delta }(X,Y)$ for finite-dimensional $X,Y\in \mathcal{A}$ are
compact (See also \cite[Claim 2.13.2]{bartosova_2019}, where a similar
result is proved for Banach spaces).

\begin{proposition}
\label{SAP_is_good} Suppose that $\mathcal{A}$ satisfies the stable
amalgamation property with modulus $\varpi $. For every $\mathcal{F}%
\subseteq \mathcal{A}$ finite and every $\varepsilon ,\delta >0$ there exist 
$V\in \mathcal{A}$ and $I_{X}\in \Emb^{\mathtt{C}}(X,V)$ for $X\in \mathcal{F%
}$ such that for every (not necessarily distinct) $X,Y,Z\in \mathcal{F}$ and
every $\gamma \in \Emb_{\delta }^{\mathtt{C}}(X,Y)$ and $\eta \in \Emb%
_{\delta }^{\mathtt{C}}(X,Z)$ there exists $J\in \Emb^{\mathtt{C}}(Z,V)$
such that $\Vert I_{Y}\circ \gamma -J\circ \eta \Vert _{\mathrm{cb}}\leq
\varpi (\delta )+\varepsilon $.

In particular, for every $X, Y\in \mathcal{F}$ one has that $I_Y\circ \Emb%
_\delta^\mathtt{C}(X,Y) \subseteq (\Emb^\mathtt{C}(X,V))_{\varpi(\delta)+%
\varepsilon}$.

\end{proposition}

The next proposition generalizes \cite[Proposition 2.13]{bartosova_2019}.

\begin{proposition}
\label{SRP=cSRP} Suppose that $\mathcal{A}\subseteq \mathtt{C}$ satisfies
the stable amalgamation property with modulus $\varpi $. Then the following
assertions are equivalent:

\begin{enumerate}[(1)]

\item $\mathcal{A}$ satisfies the (ARP).

\item $\mathcal{A}$ satisfies the (SRP) with modulus $\varpi $.

\item $\mathcal{A}$ satisfies the discrete (SRP) with modulus $\varpi $.

\item $\mathcal{A}$ satisfies the compact (SRP) with modulus $\varpi $.
\end{enumerate}

Suppose that $\mc A\con \mc B\con \langle \mathcal{A } \rangle$ also satisfies the stable
amalgamation property with modulus $\varpi$. Then \textit{(1)--(4)} are
equivalent to any of the following conditions.

\begin{enumerate}[(1)]\addtocounter{enumi}{4}

\item $\mc B $ satisfies the (ARP). 

\item For every $X,Y\in \mathcal{A}$, every $\varepsilon>0$ and every $r\in 
\mathbb{N}$ there is $Z\in \mc B$ such that every $r$%
-coloring of $\Emb^\mathtt{C}(X,Z)$ has an $\varepsilon$-monochromatic set
of the form $\gamma \circ \Emb^\mathtt{C}(X,Y)$ for some $\gamma \in \Emb^%
\mathtt{C}(Y,Z)$.
\end{enumerate}
\end{proposition}

\begin{proof}
A simple modification of the proof of the Proposition \ref{ARP=DARP} gives
that the compact (SRP) with modulus $\varpi $ implies the (SRP) with modulus 
$\varpi $, which in turn implies the discrete (SRP) with modulus $\varpi $.
Trivially, the discrete (SRP) with modulus $\varpi $ implies the discrete
(ARP), and this one is equivalent to the (ARP) by Proposition \ref{ARP=DARP}%
. We will now show that the (ARP) implies the compact (SRP) with modulus $%
\varpi $. Suppose that $\mathcal{A}$ satisfies the (ARP). Fix $X,Y\in 
\mathcal{A}$, $\delta ,\varepsilon >0$ and a compact metric space $K$. We
use Proposition \ref{SAP_is_good} to find $Y_{0}\in \mathcal{A}$ such that
for every $\phi ,\psi \in \mathrm{Emb}_{\delta }^{\mathtt{C}}(X,Y)$ there
are $i,j\in \mathrm{Emb}^{\mathtt{C}}(Y,Y_{0})$ such that $\Vert i\circ \phi
-j\circ \psi \Vert _{\mathrm{cb}}\leq \varpi (\delta )+\varepsilon $. We
will consider the space $\mathrm{Lip}(\mathrm{Emb}_{\delta }^{\mathtt{C}%
}(X,Y),K)$ of $1$-Lipschitz maps from $\mathrm{Emb}_{\delta }^{\mathtt{C}%
}(X,Y)$ to $K$ as a compact metric space, endowed with the metric $d\left(
f,g\right) =\sup \left\{ d_{K}\left( f\left( \phi \right) ,g\left( \phi
\right) \right) :\phi \in \mathrm{Emb}_{\delta }^{\mathtt{C}}(X,Y)\right\} $%
. By Proposition \ref{ARP=DARP}, $\mathcal{A}$ satisfies the compact (ARP),
so we apply it to $Y,Y_{0}\in \mathcal{A}$ and the compact space $\mathrm{Lip%
}(\mathrm{Emb}_{\delta }^{\mathtt{C}}(X,Y),K)$, and obtain some $Z\in 
\mathcal{A}$ such that every $1$-Lipschitz coloring $c:\mathrm{Emb}^{\mathtt{%
C}}(Y,Z)\rightarrow \mathrm{Lip}(\mathrm{Emb}_{\delta }^{\mathtt{C}}(X,Y),K)$
$\varepsilon $-stabilizes on $\gamma \circ \mathrm{Emb}^{\mathtt{C}}(Y,Y_{0})
$ for some $\gamma \in \mathrm{Emb}^{\mathtt{C}}(Y_{0},Z)$. We claim that $Z$
works, so let $c:\mathrm{Emb}_{\delta }^{\mathtt{C}}(X,Z)\rightarrow K$ be $1
$-Lipschitz. We can then define the 1-Lipschitz mapping $\widehat{c}:\mathrm{%
Emb}^{\mathtt{C}}(Y,Z)\rightarrow \mathrm{Lip}(\mathrm{Emb}_{\delta }^{%
\mathtt{C}}(X,Y),K)$ by setting, for $\gamma \in \mathrm{Emb}^{\mathtt{C}%
}(Y,Z)$, $\widehat{c}(\gamma ):\mathrm{Emb}_{\delta }^{\mathtt{C}%
}(X,Y)\rightarrow K$, $\phi \mapsto c\left( \gamma \circ \phi \right) $. By
the choice of $Z$, there exists $\bar{\gamma}\in \mathrm{Emb}^{\mathtt{C}%
}(Y_{0},Z)$ be such that $\widehat{c}$ $\varepsilon $-stabilizes on $\bar{%
\gamma}\circ \mathrm{Emb}^{\mathtt{C}}(Y,Y_{0})$. Choose an arbitrary $%
\varrho \in \mathrm{Emb}^{\mathtt{C}}(Y,Y_{0})$. We claim that $c$ $\left(
3\varepsilon +\varpi \left( \delta \right) \right) $-stabilizes on $\bar{%
\gamma}\circ \varrho \circ \mathrm{Emb}_{\delta }^{\mathtt{C}}(X,Y)$. Let $%
\phi ,\psi \in \mathrm{Emb}_{\delta }^{\mathtt{C}}(X,Y)$. By the choice of $%
Y_{0}$ there are $i,j\in \mathrm{Emb}^{\mathtt{C}}(Y,Y_{0})$ such that $d_{%
\mathrm{cb}}\left( i\circ \phi ,j\circ \psi \right) \leq \varepsilon +\varpi
(\delta )$. Since $d(\widehat{c}(\bar{\gamma}\circ \varrho ),\widehat{c}(%
\bar{\gamma}\circ i))\leq \varepsilon $ and $d(\widehat{c}(\bar{\gamma}\circ
\varrho ),\widehat{c}(\bar{\gamma}\circ j))\leq \varepsilon $, it follows
that $d_{K}(c(\bar{\gamma}\circ \varrho \circ \phi ),c(\bar{\gamma}\circ
i\circ \phi ))\leq \varepsilon $, $d_{K}(c(\bar{\gamma}\circ \varrho \circ
\psi ),c(\bar{\gamma}\circ j\circ \psi ))\leq \varepsilon $. Furthermore,
from $d_{\mathrm{cb}}\left( i\circ \phi ,j\circ \psi \right) \leq
\varepsilon +\varpi (\delta )$ and the fact that $c$ is $1$-Lipschitz we
deduce that $d_{K}(c(\bar{\gamma}\circ \varrho \circ \phi ),c(\bar{\gamma}%
\circ \varrho \circ \psi ))\leq 3\varepsilon +\varpi (\delta )$.

Suppose now that $\mc A\con \mc B\con \langle \mathcal{A}\rangle $ also satisfies the stable
amalgamation property with the same modulus $\varpi $. Obviously \textit{(1)}
implies \textit{(6)}. Let us prove that \textit{(6)} implies \textit{(5)}: We
prove that, assuming that {\it(6)} holds, the class $  \mathcal{B} $ satisfies
the discrete (ARP). Fix $X,Y\in   \mathcal{B}$.

\begin{claim}
\label{ijji4r4455} Let $\varepsilon,\delta>0$ and let $X_0\in \mathcal{A}$
and $t_X\in \Emb_\delta^\mathtt{C}(X,X_0)$. There are $Y_0\in \mathcal{A}$
and $t_Y\in \Emb_{\delta}^\mathtt{C}(Y,Y_0)$ such that for every $\eta\in %
\Emb^\mathtt{C}(X,Y)$ there is some $\gamma\in \Emb^\mathtt{C}(X_0,Y_0)$
such that $\|t_Y \circ \eta - \gamma \circ t_X\|_\mathrm{cb}\le
\varpi(\delta)+\varepsilon$.
\end{claim}

\begin{proof}[Proof of Claim:]
Let $0<\delta _{0}<\delta $ be such that $\varpi (\delta _{0})\leq \varpi
(\delta )/3$ and set $\varepsilon _{0}=\varepsilon /3$. Let $Z_{0}\in 
\mathcal{A}$ and $t_{0}\in \Emb_{\delta _{0}}^{\mathtt{C}}(Y,Z_{0})$. Apply
Proposition \ref{SAP_is_good} in the case of the class $\mathcal{B}$ to $\mathcal{F}:=\{X,X_{0},Z_{0}\}$, $\delta _{0}$ and $%
\varepsilon _{0}$ to obtain $Z_{1}\in  \mathcal{B}$ and $%
J_{0}\in \Emb^{\mathtt{C}}(Z_{0},Z_{1})$. Fix $Z_{2}\in \mathcal{A}$ and $%
t_{1}\in \Emb_{\delta _{0}}^{\mathtt{C}}(Z_{1},Z_{2})$. Apply again
Proposition \ref{SAP_is_good}, now in the case of the class $\mathcal{A}$,
for $\mathcal{F}:=\{X_{0},Z_{0},Z_{2}\}$, $\delta _{0}$ and $\varepsilon _{0}
$ to find the $Y_{0}\in \mathcal{A}$, and $J_{1}\in \Emb^{\mathtt{C}%
}(Z_{2},Y_{0})$. Let $I\in \Emb^{\mathtt{C}}(Z_{0},Y_{0})$ be such that 
\begin{equation}
\Vert J_{1}\circ t_{1}\circ J_{0}-I\Vert _{\mathrm{cb}}\leq {\varpi (\delta
_{0})+\varepsilon _{0}}.  \label{oij43iot3465656330}
\end{equation}%
We claim that $Y_{0}$ and $t_{Y}:=I\circ t_{0}\in \Emb_{\delta _{0}}^{%
\mathtt{C}}(Y,Y_{0})\subseteq \Emb_{\delta }^{\mathtt{C}}(Y,Y_{0})$ work.
Suppose that $\eta \in \Emb^{\mathtt{C}}(X,Y)$. By the choice of $Z_{1}$
there is some $\widehat{\eta }\in \Emb^{\mathtt{C}}(X_{0},Z_{1})$ such that 
\begin{equation}
\Vert J_{0}\circ t_{0}\circ \eta -\widehat{\eta }\circ t_{X}\Vert _{\mathrm{%
cb}}\leq \varpi (\delta _{0})+\varepsilon _{0}.  \label{oij43iot346565633}
\end{equation}%
Now by the choice of $Y_{0}$, there is some $\gamma \in \Emb^{\mathtt{C}%
}(X_{0},Y_{0})$ such that 
\begin{equation}
\Vert J_{1}\circ t_{1}\circ \widehat{\eta }-\gamma \Vert _{\mathrm{cb}}\leq
\varpi (\delta _{0})+\varepsilon _{0}.  \label{oij43iot3465656331}
\end{equation}%
It follows from the inequalities in \eqref{oij43iot3465656330}, \eqref{oij43iot346565633} and %
\eqref{oij43iot3465656331}, 
\begin{align*}
\Vert I\circ t_{0}\circ \eta -J_{1}\circ t_{1}\circ J_{0}\circ t_{0}\circ
\eta \Vert _{\mathrm{cb}}\leq & (\varpi (\delta _{0})+\varepsilon _{0})\Vert
t_{0}\circ \eta \Vert _{\mathrm{cb}}\leq (\varpi (\delta _{0})+\varepsilon
_{0}), \\
\Vert J_{1}\circ t_{1}\circ J_{0}\circ t_{0}\circ \eta -J_{1}\circ
t_{1}\circ \widehat{\eta }\circ t_{X}\Vert _{\mathrm{cb}}\leq & \Vert
J_{1}\circ t_{1}\Vert _{\mathrm{cb}}(\varpi (\delta _{0})+\varepsilon
_{0})\leq (\varpi (\delta _{0})+\varepsilon _{0}), \\
\Vert J_{1}\circ t_{1}\circ \widehat{\eta }\circ t_{X}-\gamma \circ
t_{X}\Vert _{\mathrm{cb}}\leq & (\varpi (\delta _{0})+\varepsilon _{0})\Vert
t_{X}\Vert _{\mathrm{cb}}\leq (\varpi (\delta _{0})+\varepsilon _{0}).
\end{align*}%
Consequently, 
\begin{equation*}
\Vert t_{Y}\circ \eta -\gamma \circ t_{X}\Vert _{\mathrm{cb}}=\Vert I\circ
t_{0}\circ \eta -\gamma \circ t_{X}\Vert _{\mathrm{cb}}\leq 3\varpi (\delta
_{0})+3\varepsilon _{0}\leq \varpi (\delta )+\varepsilon .\qedhere
\end{equation*}%
\end{proof}
Fix now $\varepsilon >0$, and let $\delta >0$ be such that $\varpi (\delta
)\leq \varepsilon /15$ and set $\varepsilon _{0}:=\varepsilon /5$. Fix also $%
X_{0}\in \mathcal{A}$ and $t_{X}\in \Emb_{\delta }^{\mathtt{C}}(X,X_{0})$.
We use the Claim \ref{ijji4r4455} to find $Y_{0}\in \mathcal{A}$ and $%
t_{Y}\in \Emb_{\delta }^{\mathtt{C}}(Y,Y_{0})$. We apply \textit{(6)} to $%
X_{0},Y_{0}$, $\varepsilon _{0}$ and $r$ to find the corresponding $Z_{0}\in
 \mathcal{B} $. Since $ \mathcal{B} $ satisfies
the stable amalgamation property with modulus $\varpi $, we can apply
Proposition \ref{SAP_is_good} to $\mathcal{F}:=\{X_{0},Y_{0},X,Y,Z_{0}\}$, $%
\delta $ and $\varepsilon _{0}$ and find $Z\in  \mathcal{B}$
and for each $W\in \mathcal{F}$ some $J_{W}\in \Emb^{\mathtt{C}}(W,Z)$ such
that for every (not necessarily distinct) $W_{0},W_{1},W_{2}\in \mathcal{F}$
and every $\gamma \in \Emb_{\delta }^{\mathtt{C}}(W_{0},W_{1})$ and $\eta
\in \Emb_{\delta }^{\mathtt{C}}(W_{0},W_{2})$ there is $J\in \Emb^{\mathtt{C}%
}(W_{2},Z)$ such that $\Vert I_{W_{1}}\circ \gamma -J\circ \eta \Vert _{%
\mathrm{cb}}\leq \varpi (\delta )+\varepsilon _{0}$. We claim that $Z$
works. For suppose that $c:\Emb^{\mathtt{C}}(X,Z)\rightarrow r$, and we
induce a coloring $\widehat{c}:\Emb^{\mathtt{C}}(X_{0},Z_{0})\rightarrow r$
as follows. By applying the defining property of $Z$ to the triple $%
(X,Z_{0},X)$, we obtain that $I_{Z_{0}}\circ \Emb_{\delta }^{\mathtt{C}%
}(X,Z_{0})\subseteq (\Emb^{\mathtt{C}}(X,Z))_{\varpi (\delta )+\varepsilon
_{0}}$, so, for each $\gamma \in \Emb^{\mathtt{C}}(X_{0},Z_{0})$ we can find 
$\widehat{\gamma }\in \Emb^{\mathtt{C}}(X,Z)$ such that $\Vert
I_{Z_{0}}\circ \gamma \circ t_{X}-\widehat{\gamma }\Vert _{\mathrm{cb}}\leq
\varpi (\delta )+\varepsilon _{0}$, and declare that $\widehat{c}(\gamma
):=c(\widehat{\gamma })$. Let $J\in \Emb^{\mathtt{C}}(Y_{0},Z_{0})$ and $j<r$
be such that $J\circ \Emb^{\mathtt{C}}(X_{0},Y_{0})\subseteq (\widehat{c}%
^{-1}(j))_{\varepsilon _{0}}$. Let $I\in \Emb^{\mathtt{C}}(Y,Z)$ be such
that 
\begin{equation}
\Vert I-I_{Z_{0}}\circ J\circ t_{Y}\Vert _{\mathrm{cb}}\leq \varpi (\delta
)+\varepsilon _{0}.  \label{lij34irfrewrtyt}
\end{equation}%
Let us see that $I\circ \Emb^{\mathtt{C}}(X,Y)\subseteq
(c^{-1}(j))_{\varepsilon }$. Fix $\eta \in \Emb^{\mathtt{C}}(X,Y)$. Let $%
\widehat{\eta }\in \Emb^{\mathtt{C}}(X_{0},Y_{0})$ be such that 
\begin{equation}
\Vert \widehat{\eta }\circ t_{X}-t_{Y}\circ \eta \Vert _{\mathrm{cb}}\leq
\varpi (\delta )+\varepsilon _{0}.  \label{lij34irfrewrtyt1}
\end{equation}%
Let $\gamma \in \Emb^{\mathtt{C}}(X_{0},Z_{0})$ be such that 
\begin{equation}
\text{$\Vert \gamma -J\circ \widehat{\eta }\Vert _{\mathrm{cb}}\leq
\varepsilon _{0}$ and $\widehat{c}(\gamma )=j$, }  \label{lij34irfrewrtyt2}
\end{equation}%
so, by the definition of $\widehat{c}$, let now $\widehat{\gamma }\in \Emb^{%
\mathtt{C}}(X,Z)$ be such that 
\begin{equation}
\text{$\Vert \widehat{\gamma }-I_{Z_{0}}\circ \gamma \circ t_{X}\Vert _{%
\mathrm{cb}}\leq \varpi (\delta )+\varepsilon _{0}$ and $c(\widehat{\gamma }%
)=j$.}  \label{lij34irfrewrtyt3}
\end{equation}%
Putting all this together, 
\begin{align*}
\nrm{I \circ \ga - I_{Z_0}\circ J \circ t_Y\circ \eta}_\mr{cb} \le & \varpi(\de)+\vep_0\\
\nrm{I_{Z_0}\circ J \circ t_Y\circ \eta- I_{Z_0}\circ J \circ \widehat{\eta}\circ t_X}_\mr{cb} \le &  \varpi(\de)+\vep_0\\
\nrm{I_{Z_0}\circ J \circ \widehat{\eta}\circ t_X - I_{Z_0} \circ \ga \circ t_X}_\mr{cb} \le & \vep_0.
\end{align*}  
From \eqref{lij34irfrewrtyt},  \eqref{lij34irfrewrtyt1}, \eqref{lij34irfrewrtyt2} and \eqref{lij34irfrewrtyt3} we obtain 
\begin{equation*} 
 \text{ $ \nrm{I \circ \eta - \widehat{\ga}}_\mr{cb}\le 3\varpi(\de) +4\vep_0 \le \vep$ and $c(\widehat{\ga})=j$.}
\end{equation*}  
\textit{(5)} implies \textit{(1)}: The proof of that the discrete (ARP) of $%
\mathcal{B}$ implies the discrete (ARP) of $\mathcal{A}$ is
similar to that of the implication \textit{6)} implies \textit{5)} so we
only sketch it and we leave further details to the reader. Fix $X,Y\in 
\mathcal{A}$, $\varepsilon >0$ and $r\in \mathbb{N}$. Let $\delta >0$ be
such that $\varpi (\delta )\leq \varepsilon /8$ and set $\varepsilon
_{0}:=\varepsilon /4$. Let $Z_{0}\in \mc B\con \langle \mathcal{A}\rangle $ be such
that every $r$-coloring of $\Emb^{\mathtt{C}}(X,Z_{0})$ has an $\varepsilon
_{0}$-monochromatic set of the form $J\circ \Emb^{\mathtt{C}}(X,Y)$ for some 
$J\in \Emb^{\mathtt{C}}(Y,Z_{0})$. Let $Z_{1}\in \mathcal{A}$ and $t\in \Emb%
_{\delta }^{\mathtt{C}}(Z_{0},Z_{1})$. We use the Proposition \ref%
{SAP_is_good} in the case of the class $\mathcal{A}$ for $\mathcal{F}%
=\{X,Y,Z_{1}\}$, $\varepsilon $ and $\delta $ to find the corresponding $Z$.
Then it follows that every $r$-coloring of $\Emb^{\mathtt{C}}(X,Z)$ has an $%
\varepsilon $-monochromatic set of the form $I\circ \Emb^{\mathtt{C}}(X,Y)$
for some $I\in \Emb^{\mathtt{C}}(Y,Z)$. %
\end{proof}

It is unclear to us whether the characterization of the (ARP) provided in %
\ref{SRP=cSRP} holds for an arbitrary class $\mathcal{A}\subseteq \mathtt{C}$.

Just like for the stable amalgamation property we presented in Proposition   \ref{oij32irj23jr32}, we have the following relationship between the approximate Ramsey properties of a class $\mc A$ of f.d. operator spaces or systems, and the corresponding class $\mc A^R$.
 \begin{proposition}\label{oij32irj23jr32weewe}
Let $\mathtt{C}$ be either $\mathtt{Osp}
$ or $\mathtt{Osy}$. Suppose that $\mathcal{A}$ is a  class of
finite-dimensional elements of $\mathtt{C}$ such that for $X,Y\in 
\mathcal{A}$, the $\infty $-sum $X\oplus _{\infty}Y$ belongs to $\mathcal{A}$. 
\begin{enumerate}[(1)]  
\item Suppose that   $R\in \mtt C$ is such  that every element in $\mr{Age}^\mtt C(R)$ embeds into an element of $\mc A$. 
\begin{enumerate}[({1}.1)]  
\item If $\mc A$ has the (ARP), then the class $\mc A^R$  also has the (ARP). 
\item If $\mc A$ has the (SRP) with modulus $\varpi(\de)$, then the class $\mc A^R$  also has the (SRP) with modulus $\varpi(\de)+\de$.   
 \end{enumerate}  
 \item Suppose that  $R\in \mc A$. 
\begin{enumerate}[({2}.1)]  
\item If $\mc A^R$ has the (ARP), then the class $\mc A$  also has the (ARP). 
\item If $\mc A^R$ has the (SRP) with modulus $\varpi(\de)$, then the class $\mc A$  also has the (SRP) with modulus $\varpi(\de)$.   
 \end{enumerate}  
\end{enumerate}  
\end{proposition}  
\begin{proof}
 {\it(1.1)} is proved similarly than {\it(1)} in Proposition \ref{oij32irj23jr32}, so we leave the details to the reader.  
 {\it(1.2)}: If $\mc A$ has the (SRP) with modulus $\varpi(\de)$, then it satisfies the stable amalgamation property with modulus $\varpi(\de)$. It follows from Proposition   \ref{oij32irj23jr32}  that $\mc A^R$ satisfies the stable amalgamation property with modulus $\varpi(\de)+\de$, and from this, {\it(1.1)}, and the equivalence between {\it(1)} and {\it(2)} in Proposition \ref{SRP=cSRP}, we obtain that $\mc A^R$ satisfies the (SRP) with modulus $\varpi(\de)+\de$.  {\it(2)} is proved similarly. 
\end{proof}  
\subsubsection{Extreme amenability and Ramsey properties}

Recall that a topological group $G$ is called \emph{extremely amenable} if
every continuous action of $G$ on a compact Hausdorff space has a fixed
point. We are going to present a characterization of extreme amenability for
automorphism groups of Fra\"{\i}ss\'{e} limits in one of the categories
introduced before in terms of a Ramsey property, which can be seen as an
instance of the Kechris-Pestov-Todorcevic correspondence \cite%
{kechris_fraisse_2005} in the case of metric structures in \cite%
{melleray_extremely_2014}. The key for this characterization is the
following intrinsic reformulation of extreme amenability from \cite[Theorem
2.1.11]{pestov_dynamics_2006}.

\begin{theorem}
\cite{pestov_dynamics_2006}\label{welmrjpewrwe} A topological group $G$ is
extremely amenable if and only if there is a directed collection of bounded
left-invariant continuous pseudometrics $(d_i)_{i\in I}$ determining the
topology of $G$ and such that every metric $G$-space $G /\{d_i=0\}$ is
finitely oscillation stable, that is for every finite $F\subseteq G/\{d_i=0\}
$, every bounded and uniformly continuous $\gamma: G/\{d_i=0\} \to \mathbb{R}
$ and every $\varepsilon>0$ there is $g\in G$ such that $\osc (\gamma
\upharpoonright (g \cdot F))= \max_{[h_0],[h_1] \in F} |\gamma([{g \cdot h_0}%
])- \gamma([{g \cdot h_1}])| \le \varepsilon.$
\end{theorem}

In the previous statement, $[h]$ denotes the equivalence class of $h$ in the
quotient $G/\{d_{i}=0\}$. As before, let $\mathtt{C}$ be one of the classes $%
\mathtt{Osp}$, $\mathtt{Osp}^{R}$, $\mathtt{Osy}$, or $\mathtt{Osy}^{R}$,
with the corresponding notion of $\mathtt{C}$-embedding. Fix $E\in \mathtt{C}
$. For a fixed finite dimensional $X\subseteq E$, we let $d_{X}$ be the
pseudometric on $\mathrm{Aut}^{\mathtt{C}}(E)$ defined by $d_{X}(\alpha
,\beta ):=\Vert (\alpha -\beta )\upharpoonright {X}\Vert _{\mathrm{cb}}$. In
general, $(d_{X})$ where $X$ varies within the finite-dimensional subspaces
of $E$ defines the topology on $\mathrm{Aut}^{\mathtt{C}}(E)$. Suppose that $%
E\in \mathtt{C}$ is stably homogeneous. Observe that by the stable
homogeneity property of $E$, the restriction map $\alpha \mapsto \alpha
\upharpoonright {X}$ is an isometry from $(\mathrm{Aut}^{\mathtt{C}%
}(E),d_{X})$ onto a dense subset of $\mathrm{Emb}^{\mathtt{C}}(X,E)$. In
particular, a continuous coloring of $(\mathrm{Aut}^{\mathtt{C}}(E),d_{X})$
induces a continuous coloring of $\mathrm{Emb}^{\mathtt{C}}(X,E)$. Also,
given a finite $F\subseteq \Emb^{\mathtt{C}}(X,E)$ we can find a finite
dimensional $Y\subseteq E$ such that $F\subseteq \Emb^{\mathtt{C}}(X,Y)$,
where we consider $\Emb^{\mathtt{C}}(X,Y)$ canonically included in $\Emb^{%
\mathtt{C}}(X,E)$ after composing with the corresponding inclusion of $Y$
into $E$. Reciprocally each set $\Emb^{\mathtt{C}}(X,Y)$ is compact, so it
has finite $\varepsilon $-dense subsets. Theorem \ref{welmrjpewrwe} can be
restated as follows: 

\begin{lemma}
\label{Lemma:oscillation-stable}The following statements are equivalent:

\begin{enumerate}[(a)]

\item $\mathrm{Aut}^\mathtt{C}( E) $ is extremely amenable.

\item For every finite-dimensional $X,Y\subseteq E$ and every $\varepsilon >0
$, every $2$-coloring of $\Emb^{\mathtt{C}}(X,E)$, has an $\varepsilon $%
-monochromatic set of the form $\alpha \circ \Emb^{\mathtt{C}}(X,Y)$ for
some $\alpha \in \mathrm{Aut}^{\mathtt{C}}(E)$, that is for every $c:\Emb^{%
\mathtt{C}}(X,E)\rightarrow r$ there is there exists $\alpha \in \mathrm{Aut}%
^{\mathtt{C}}(E)$ and $j\in \left\{ 0,\ldots ,r-1\right\} $ such that for
every $\gamma \in \Emb^{\mathtt{C}}(X,Y)$ there exists $\xi \in \Emb^{%
\mathtt{C}}(X,E)$ such that $c(\xi )=j$ and $\Vert \alpha \circ \gamma -\xi
\Vert _{\mathrm{cb}}\leq \varepsilon $.

\end{enumerate}
\end{lemma}

We refer the reader to \cite[Chapter 1]{pestov_dynamics_2006} for more
information on oscillation stability. We suppose that $\mathcal{A}\subseteq 
\mathtt{C}$ is such that $\langle \mathcal{A}\rangle $ is a stable Fra\"{\i}%
ss\'{e} class with modulus $\varpi $, whose Fra\"{\i}ss\'{e} limit is $E$.
We can now state the analogue in this context of the celebrated
Kechris-Pestov-Todorcevic correspondence from \cite{kechris_fraisse_2005}.
The possibility of using the KPT correspondence in the setting of metric
structures was first suggested by Melleray and Tsankov in \cite%
{melleray_extremely_2014}. 

\begin{proposition}[KPT correspondence for $R$-operator spaces and systems]
\label{Proposition:ARP} Suppose that $\mathcal{A}\subseteq \mathtt{C}$ is
such that $\mathcal{A}$ and $\langle \mathcal{A}\rangle $ satisfy the stable
amalgamation property with stability modulus $\varpi $, and let $E$ be the
corresponding Fra\"{\i}ss\'{e} limit. The following assertions are
equivalent:

\begin{enumerate}[(1)]

\item $\mathrm{Aut}^{\mathtt{C}}(E)$ is extremely amenable.

\item For every finite-dimensional $X,Y\subset E$ in $\mathtt{C}$ and $%
\delta \geq 0$, $\varepsilon >0$, every compact coloring of $\mathrm{Emb}%
_{\delta }^{\mathtt{C}}(X,E)$ has an $(\varepsilon +\varpi (\delta ))$%
-monochromatic set of the form $\alpha \circ \mathrm{Emb}_{\delta }^{\mathtt{%
C}}(X,Y)$ for some $\alpha \in \mathrm{Aut}^{\mathtt{C}}(E)$.

\item For every finite-dimensional $X,Y\subset E$ in $\mathtt{C}$ and $%
\varepsilon >0$, every finite coloring of $\mathrm{Emb}^{\mathtt{C}}(X,E)$
has an $\varepsilon $-monochromatic set of the form $\alpha \circ \Emb^{%
\mathtt{C}}(X,Y)$ for some $\alpha \in \mathrm{Aut}^{\mathtt{C}}(E)$.

\item For every $X,Y\subset E$ that belong to $\mathcal{A}$ and $\varepsilon
>0$, every finite coloring of $\mathrm{Emb}^{\mathtt{C}}(X,E)$ has an $%
\varepsilon $-monochromatic set of the form $\alpha \circ \Emb^{\mathtt{C}%
}(X,Y)$ for some $\alpha \in \mathrm{Aut}^{\mathtt{C}}(E)$.

\item $\langle \mathcal{A}\rangle $ satisfies the (ARP).

\item $\mathcal{A}$ satisfies the (ARP).

\item For every $X,Y\in \mathcal{A}$, every $r\in \mathbb{N}$ and every $%
\varepsilon >0$ there is $Z\in \langle \mathcal{A}\rangle $ such that every $%
r$-coloring of $\mathrm{Emb}^{\mathtt{C}}(X,Z)$ has an $\varepsilon $%
-monochromatic set of the form $\gamma \circ \mathrm{Emb}^{\mathtt{C}}(X,Y)$
for some $\gamma \in \mathrm{Emb}^{\mathtt{C}}(Y,Z)$.
\end{enumerate}
\end{proposition}

The key part is the equivalence between {\it(1)} and {\it(7)} since in this last
condition, $X,Y\subseteq E$ are not arbitrary finite-dimensional subspaces,
but are assumed to be in $\mathcal{A}$ while $Z$ is not necessarily in $%
\mathcal{A}$ but in the wider class $\langle \mathcal{A}\rangle$.

\begin{proof}
The proof uses standard arguments; see for example \cite[Proposition 3.4 and
Theorem 3.10]{melleray_extremely_2014}. We first show that properties 
\textit{(1)--(4)} are equivalent, and then we show that properties \textit{(4)}%
--\textit{(7)} are equivalent.

\noindent \textit{(1)} implies \textit{(2)}: Fix all data, in particular, let $%
c:\Emb_{\delta }^{\mathtt{C}}(X,E)\rightarrow K$ be a compact coloring, and
let $K=\bigcup_{j<r}U_{j}$ be a covering of $K$ by disjoint sets $U_{j}$ of
diameter less than $\varepsilon /3$. We consider the induced coloring $%
\widehat{c}:\Emb^{\mathtt{C}}(X,E)\rightarrow r$ that declares $\widehat{c}%
(\gamma )=j$ when $c(\gamma )\in U_{j}$. Since $E$ is stably homogeneous
with modulus $\varpi $, it follows from the Fra\"{\i}ss\'{e} correspondence
in Theorem \ref{Proposition:exists} that $\mathrm{Age}^{\mathtt{C}}(E)$ is a
stable Fra\"{\i}ss\'{e} class with modulus $\varpi $, so we can use {\it(1)} in 
Proposition \ref{SAP_is_good}  to find $Y_{0}\subseteq E$ containing and $%
I\in \Emb(Y,Y_{0})$ such that $I\circ \Emb_{\delta }^{\mathtt{C}%
}(X,Y)\subseteq (\Emb^{\mathtt{C}}(X,Y_{0}))_{\varpi (\delta )+\varepsilon
/6}$. By composing with the corresponding inclusion, we identify each $%
\delta $-embedding from $X$ into $Y_{0}$ as a $\delta $-embedding from $X$
into $E$. Since $\Emb_{\delta }^{\mathtt{C}}(X,Y_{0})$ is compact, we can
use Lemma \ref{Lemma:oscillation-stable} to find $\beta \in \mathrm{Aut}^{%
\mathtt{C}}(E)$ and $j<r$ such that $\beta \circ \Emb_{\delta }^{\mathtt{C}%
}(X,Y_{0})\subseteq (\widehat{c}^{-1}(j))_{\varepsilon /6}$. Let us see that 
$\alpha :=\beta \circ I$ satisfies the desired property: Given $\gamma \in %
\Emb_{\delta }^{\mathtt{C}}(X,Y)$, let $\eta \in \Emb^{\mathtt{C}}(X,Y_{0})$
such that $\Vert I\circ \gamma -\eta \Vert _{\mathrm{cb}}\leq \varpi (\delta
)+\varepsilon /6$. Let also $\xi \in \Emb^{\mathtt{C}}(X,E)$ such that $%
c(\xi )=j$ and $\Vert \beta \circ \eta -\xi \Vert _{\mathrm{cb}}\leq
\varepsilon /6$, and consequently, $\Vert \alpha \circ \gamma -\xi \Vert _{%
\mathrm{cb}}\leq \varpi (\delta )+\varepsilon /3$. Since $U_{j}\subseteq K$
has diameter less than $\varepsilon /3$, it follows that $\osc%
(c\upharpoonright (\alpha \circ \Emb_{\delta }^{\mathtt{C}}(X,Y)))\leq
\varpi (\delta )+\varepsilon $.

\noindent \textit{(2)} implies \textit{(3)}: It suffices to prove that the
following holds for every $\delta >0$: for every finite-dimensional $%
X,Y\subset E$ in $\mathtt{C}$, and $\varepsilon >0$, every finite coloring
of $\mathrm{Emb}_{\delta }^{\mathtt{C}}(X,E)$ has an $(\varepsilon +\varpi
(\delta ))$-monochromatic set of the form $\alpha \circ \Emb^{\mathtt{C}%
}(X,Y)$ for some $\alpha \in \mathrm{Aut}^{\mathtt{C}}(E)$. Fix a finite
coloring $c:\mathrm{Emb}_{\delta }^{\mathtt{C}}(X,E)\rightarrow r$, and let $%
K$ be the ball of $\ell _{\infty }^{k}$ centered at $0$ and of radius $2$,
and let $f:\mathrm{Emb}_{\delta }^{\mathtt{C}}(X,E)\rightarrow K$ be defined
by $f(\sigma ):=(d_{\mathrm{cb}}(\sigma ,c^{-1}(i)))_{i<r}$. This is a
compact coloring, so by the hypothesis, there is $\alpha \in \Aut^{\mathtt{C}%
}(E)$ such that the oscillation of $f$ on $\alpha \circ \Emb_{\delta }^{%
\mathtt{C}}(X,Y)$ is at most $\varepsilon +\varpi (\delta )$. Then $\alpha
\circ \Emb_{\delta }^{\mathtt{C}}(X,Y)$ is $(\varepsilon +\varpi (\delta ))$%
-monochromatic for $c$. Indeed, fix $\bar{\phi}\in \Emb_{\delta }^{\mathtt{C}%
}(X,Y)$, and let $i:=c(\alpha \circ \bar{\phi})$. Then the $i$-th coordinate
of $f(\alpha \circ \bar{\phi})$ with respect to the canonical basis of $\ell
_{r}^{\infty }$ is zero. Since the oscillation of $f$ on $\alpha \circ \Emb%
_{\delta }^{\mathtt{C}}(X,Y)$ is at most $\varepsilon +\varpi \left( \delta
\right) $, we have that for every $\phi \in \Emb_{\delta }^{\mathtt{C}}(X,Y)$%
, $f\left( \alpha \circ \phi \right) $ and $f\left( \alpha \circ \bar{\phi}%
\right) $ have distance at most $\varepsilon +\varpi \left( \delta \right) $%
. Since the $i$-th coordinate of $f\left( \alpha \circ \phi \right) $ is
equal to $d_{\mathrm{cb}}(\alpha \circ \phi ,c^{-1}(i))$, this implies that $%
\alpha \circ \phi ,c^{-1}(i))\leq \varepsilon +\varpi (\delta )$.

\noindent \textit{(3)} implies \textit{(4)} trivially.

\noindent \textit{(4)} implies \textit{(7)}: It is easy to see that \textit{(4)} and \textit{(7)}, respectively, are equivalent to

\begin{enumerate}[\textit{(4')}]

\item for every $X,Y\in \mathcal{A}$, every compact metric space $(K,d_K)$, $%
\varepsilon>0$ and 1-Lipschitz mapping $c: (\Emb^\mathtt{C}(X,E),d_\mathrm{cb%
})\to (K,d_K)$ there is some $\alpha\in \Aut^\mathtt{C}(E)$ such that $\osc%
(c\upharpoonright ( \alpha\circ \Emb^\mathtt{C}(X,Y)))\le \varepsilon$.

\item[\textit{(7')}] for every $X,Y\in \mathcal{A}$, every compact metric space $(K,d_K)$
and $\varepsilon>0$ there is some $Z\in \langle \mathcal{A}\rangle $ such
that for everyn 1-Lipschitz mapping $c: \Emb^\mathtt{C}(X,Z)\to K$ there is
some $\gamma\in \Emb^\mathtt{C}(Y,Z)$ such that $\osc(c\upharpoonright (
\gamma\circ \Emb^\mathtt{C}(X,Y)))\le \varepsilon$.
\end{enumerate}

This is similarly proved as the equivalence between the discrete (ARP) and
the compact (ARP) (Proposition \ref{ARP=DARP}). Suppose by contradiction
that \textit{(7')} does not hold. Therefore there exist $X, Y\subset E$ in $%
\mathcal{A} $, $\varepsilon _{0}>0$ and a compact metric $(K,d_K)$
witnessing that. Let $D_{X}\subseteq \Emb^{\mathtt{C}}(X,E)$ and $%
D_{Y}\subseteq \Emb^{\mathtt{C}}(Y,E)$ be countable dense subsets, and let $%
(Z_{n})_{n\in \mathbb{N}}$ be an increasing sequence of finite-dimensional $%
Z_{n}\subset E$ in $\langle \mathcal{A } \rangle$, such that $Y\subseteq
Z_{0}$ and such that for every $\phi \in D_{X}$ and $\psi \in D_{Y}$ there
exists $n\in \mathbb{N} $ such that $\phi\in \Emb^\mathtt{C}(X, Z_n)$ and $%
\psi\in \Emb^\mathtt{C}(Y,X_n)$. This implies that $\bigcup_{n}\Emb^{\mathtt{%
C}}(X,Z_{n})$ and $\bigcup_{n}\Emb^{\mathtt{C}}(Y,Z_{n})$ are dense in $\Emb%
^{\mathtt{C}}(X,E)$ and $\Emb^{\mathtt{C}}(Y,E)$, respectively. Choose for
each $n\in \mathbb{N}$ a \textquotedblleft bad 1-Lipschitz coloring $c_{n}:%
\Emb^{\mathtt{C}}(X,Z_{n})\rightarrow K $. Let $\mathcal{U}$ be a
non-principal ultrafilter on $\mathbb{N}$. We define ${c}:\Emb^{\mathtt{C}%
}(X,E)\rightarrow K $ by choosing for a given $\phi \in \Emb^{\mathtt{C}%
}(X,E)$ a sequence $(\phi _{n})_{n}$, each $\phi _{n}\in \Emb^{\mathtt{C}%
}(X,Z_{n})$, such that $\lim_{n}d_{\mathrm{cb}}(\phi ,\phi _{n})=0$, and
then by declaring ${c}(\phi):=\mathcal{U}-\lim {c}_{n}(\phi _{n})$. We claim
that $c$ is well defined. Indeed, if $(\psi _{n})_{n}$ is another sequence
such that $\lim_{n}d_{\mathrm{cb}}(\psi _{n},\phi )=0$, then $\lim_{n}d_{%
\mathrm{cb}}(\psi _{n},\phi _{n})=0$. Using the fact that each ${c}_{n}$ is $%
1$-Lipschitz, this implies that $\lim_{n}\left\vert{c}_{n}(\psi _{n})-{c}%
_{n}(\phi _{n})\right\vert_\infty =0$. Since $\mathcal{U}$ is nonprincipal,
we conclude that $\mathcal{U}-\lim {c}_{n}(\phi _{n})=\mathcal{U}-\lim {c}%
_{n}(\psi _{n})$. Let $\varepsilon >0$ be arbitrary. By the assumption 
\textit{(4')} there exists $\alpha \in \Aut^{\mathtt{C}}(E)$ such that $\osc%
(c\upharpoonright ( \alpha \circ \Emb^\mathtt{C}(X,Y) ) \le \varepsilon$.
Let $D\subseteq \Emb_{\delta }^{\mathtt{C}}(X,Y)$ be a finite $\varepsilon $%
-dense subset. Using the definition of $c$ and the fact that $c$ stabilizes
on $\alpha \circ \Emb_{\delta }^{\mathtt{C}}(X,Y))$, choose $n\in \mathbb{N}$
such that:

\begin{enumerate}[$\bullet$]

\item for every $\phi \in D$ there exists $\bar{\phi}\in \Emb^{\mathtt{C}%
}(X,Z_{n})$ such that $d_{\mathrm{cb}}(\alpha \circ \phi ,\bar{\phi})\leq
\varepsilon $, and such that for every $\phi ,\psi \in D$ one has that $%
\left\vert c_{n}(\bar{\phi})-c_{n}(\bar{\psi})\right\vert \leq \varepsilon$.

\item there is $\gamma \in \Emb^{\mathtt{C}}(Y,Z_{n})$ such that $d_{\mathrm{%
cb}}(\alpha \upharpoonright {Y},\gamma )\leq \varepsilon $.
\end{enumerate}

It follows that $d_{\mathrm{cb}}(\bar{\phi},\gamma \circ \phi )\leq
2\varepsilon $ for every $\phi \in D$. Hence, $\left\vert c_{n}(\gamma \circ
\phi )-c_{n}(\gamma \circ \psi )\right\vert \leq 5\varepsilon +\varpi
(\delta )$ for every $\phi ,\psi \in D$. Consequently, $c_{n}$ has
oscillation at most $7\varepsilon +\varpi (\delta )$ on $\gamma \circ \Emb%
_{\delta }^{\mathtt{C}}(X,Y)$. Since $\varepsilon $ is an arbitrary positive
real number, this contradicts the assumption that $c_{n}$ is a
\textquotedblleft bad\textquotedblright\ coloring witnessing the failure of 
\textit{(7')}.

\noindent The equivalence of \textit{(5)} --\textit{(7)} is done in
Proposition \ref{SRP=cSRP}. 

\noindent \textit{(7)} implies \textit{(4)}: Fix $X\subset Y\subset E$
that belong to $\mathcal{A}$, $\varepsilon >0$, and $r\in \mathbb{N}$. Let $%
Z\in \langle \mathcal{A}\rangle $ be witnessing that \textit{(7)} holds for
the given data. Since $E$ is universal for $\langle \mathcal{A}\rangle $, we
may assume that $Z$ is a substructure of $E$. Given an $r$-coloring $c$ of $%
\Emb^{\mathtt{C}}(X,E)$, we can take its restriction to $\Emb^{\mathtt{C}%
}(X,Z)$, and then find $\gamma \in \Emb^{\mathtt{C}}(Y,Z)$ such that $\gamma
\circ \Emb^{\mathtt{C}}(X,Y)$ is $\varepsilon $-monochromatic for $c$.
Finally, let $\alpha \in \Aut^{\mathtt{C}}(E)$ be such that $d_{\mathrm{cb}%
}(\alpha \upharpoonright {Y},\gamma )\leq \varepsilon $. It follows that $c$ 
$2\varepsilon $-stabilizes on $\alpha \circ \Emb^{\mathtt{C}}(X,Y)$.
\end{proof}

\subsubsection{The Dual Ramsey Theorem}

The proof of the (ARP) properties of several classes of $R$-operator spaces
and systems is based on the \emph{Dual Ramsey Theorem }(DRT) of R. L. Graham
and B. L. Rothschild \cite{graham_ramseys_1971}. Its statement is commonly
presented by using partitions, but for practical purposes, we recall it in
terms of rigid surjections between finite linear orderings. Given two linear
orderings $(R,<_{R})$ and $(S,<_{S})$, an onto map $f:R\rightarrow S$ is a 
\emph{rigid surjection} when $\min_R f^{-1}(s_{0})<\min_R f^{-1}(s_{1})$ for
every $s_{0},s_{1}\in S$ such that $s_{0}<_{S}s_{1}$. The class of rigid
surjections from $R$ onto $S$ is denoted by $\mathrm{Epi}(R,S)$.

\begin{theorem}[(DRT) \protect\cite{graham_ramseys_1971}]
\label{Theorem:DLT-rigid}For every finite linear orderings $R$ and $S$ and
every $r\in \mathbb{N}$ there exists an integer $n\in \mathbb{N}$ such that,
considering $n $ naturally ordered, every $r$-coloring of $\mathrm{Epi}(n,R)$
has a monochromatic set of the form $\mathrm{Epi}(S,R)\circ \gamma =\{{%
\sigma \circ \gamma }\,:\,{\sigma \in \mathrm{Epi}(S,R)}\}$ for some $\gamma
\in \mathrm{Epi}(n,S)$.
\end{theorem}

\section{The Ramsey property for operator spaces\label{Sec:spaces}}

The aim of this section is to prove the extreme amenability of the
automorphism groups of operator spaces which are Fra\"{\i}ss\'{e} limits of
certain classes of finite-dimensional injective operator spaces. 

\subsection{Fra\"{\i}ss\'{e} limits of exact operator spaces}

Recall that an operator space $E$ is \emph{injective} if it is injective in
the category of operator spaces; that is, whenever $X\subset Y$ are operator
spaces, any completely contractive map $\phi :X\rightarrow E$ can be
extended to a completely contractive map $\psi :Y\rightarrow E$. The
finite-dimensional injective operator spaces are precisely the ones of the
form $M_{q_{1},s_{1}}\oplus _{\infty }\cdots \oplus _{\infty }M_{q_{n},s_{n}}
$ for $n,q_{1},s_{1},\ldots ,q_{n},s_{n}\in \mathbb{N}$. Here, $M_{q,s}$ is
the space of $q\times s$ matrices with complex entries, regarded as a space
of operators on the $\left( q+s\right) $-dimensional Hilbert space of the
form%
\begin{equation*}
\begin{bmatrix}
0 & \ast \\ 
0 & 0%
\end{bmatrix}%
\text{,}
\end{equation*}%
where the diagonal blocks have size $q\times q$ and $s\times s$. The
operator spaces $M_{q,1}$ and $M_{1,q}$ are called the $q$-dimensional \emph{%
column operator Hilbert space} and the $q$-dimensional \emph{row operator
Hilbert space}, respectively. The space $M_{q,q}$ of $q\times q$ matrices
will be simply denoted by $M_{q}$, and the $n$-fold $\infty $-sum of $%
M_{q,s} $ by itself will be denoted by $\ell _{\infty }^{n}(M_{q,s})$. It is
known that the class of finite-dimensional injective operator spaces
coincides with the class of finite-dimensional ternary rings of operator;
see \cite{smith_finite_2000}. The finite-dimensional \emph{commutative }%
ternary rings of operators are precisely the spaces $\ell _{\infty }^{n}$
for $n\in \mathbb{N}$ \cite[Subsection 8.6.4]{blecher_operator_2004}, which
are precisely the finite-dimensional injective minimal operator spaces.

\begin{definition}[Injective classes]
\label{matrix_classes}We say that a family of finite-dimensional operator
spaces is an \emph{injective class} of operator spaces if it is one of the
following families

\begin{enumerate}[-]

\item $\mathbb{I}_{1}:=\{\ell _{\infty }^{n}\}_{n\in \mathbb{N} } $,

\item $\mathbb{I}_{q}:=\{\ell _{\infty }^{n}(M_{q})\}_{n\in \mathbb{N}}$,

\item $\mathbb{I}_{\mathrm{c}}:=\{\ell _{\infty }^{n}(M_{q,1})\}_{n,q\in 
\mathbb{N}} $,

\item $\mathbb{I}_{\mathrm{e}}:=\{M_{q}\}_{q\in \mathbb{N}}$, and

\item $\mathbb{I}_{\mathrm{inj}}=\{M_{q_{1},s_{1}}\oplus _{\infty }\cdots
\oplus _{\infty }M_{q_{n},s_{n}}\mathbb{\}}_{n,q_{1},s_{1},\ldots
,q_{n},s_{n}\in \mathbb{N}} $.
\end{enumerate}
\end{definition}

If $\mathbb{I}$ is any of the classes in Definition \ref{matrix_classes},
then $\mathbb{I}$ satisfies the assumptions of Theorem \ref%
{injective_classes_are_fraisse}. Thus,   $\mathbb{I}$ is a stable amalgamation class and $\langle 
\mathbb{I}\rangle$ is a stable Fra\"{\i}ss\'{e} class with modulus $\varpi
\left( \delta \right) =\delta$. The corresponding classes $[\mathbb{I}]$ can
be described as follows.

\begin{definition}[Spaces locally approximated by injective classes]
\label{Definition:exact} \mbox{}

\begin{enumerate}[(a)]

\item $[\mathbb{I}_{1}]$ is the class of \emph{minimal operator spaces},
which can be identified with\ the class of \emph{Banach spaces}.

\item $[\mathbb{I}_q]$ is the class of $q$-\emph{minimal operator spaces}
(see \cite{lehner_mn-espaces_1997}).

\item $[\mathbb{I}_\mathrm{c}]$ is the class of \emph{operator sequence
spaces} (see \cite{lambert_operatorfolgenraume_2002}).

\item $[\mathbb{I}_{\mathrm{e}}]=\left[ \mathbb{I}_{\mathrm{inj}}\right] $
is the class of \emph{exact operator spaces} (see \cite{pisier_exact_1995}, 
\cite[Theorem 17.1]{pisier_introduction_2003}).
\end{enumerate}
\end{definition}

Observe that Banach spaces (endowed with their minimal operator space
structure) coincide with $1$-minimal operator spaces. Their limits are the
following.

\begin{example}
\label{Definition:NG} \mbox{}

\begin{enumerate}[$\bullet$]

\item $\flim\mathbb{I}_1 $ is the \emph{Gurarij space} $\mathbb{G}$ \cite%
{gurarij_spaces_1966,lusky_gurarij_1976,kubis_proof_2013,ben_yaacov_fraisse_2015}%
.

\item $\flim \mathbb{I}_q$ is the \emph{$q$-minimal Gurarij space} $\mathbb{%
G }_{q}$ \cite[\S 6.5]{lupini_fraisse_2018}.

\item $\flim \mathbb{I}_\mathrm{c}$ is the \emph{Gurarij column space} $%
\mathbb{GC}$ \cite[\S 6.3]{lupini_fraisse_2018}.

\item $\flim \mathbb{I}_\mathrm{e}$ is the \emph{noncommutative Gurarij
space } $\mathbb{NG}$ \cite[\S 8.1]{lupini_fraisse_2018}.
\end{enumerate}
\end{example}

Observe that the Gurarij space $\mathbb{G}$ coincides with $\mathbb{G}_{1}$
in the notation of Example \ref{Definition:NG}. In fact, the original
definition of the Gurarij space considered by Gurarij \cite%
{gurarij_spaces_1966} and Lusky \cite%
{lusky_gurarij_1976,lusky_separable_1977,lusky_construction_1979} looks
somewhat different. The original characterization of $\mathbb{G}$ is as the
unique separable Banach space satisfying the following extension property:
for every finite-dimensional\ Banach spaces $E\subseteq F$, linear
contraction $\phi :E\rightarrow \mathbb{G}$, and $\varepsilon >0$, there
exists an extension $\hat{\phi}:F\rightarrow \mathbb{G}$ satisfying $||\hat{%
\phi}||{}<1+\varepsilon $. The fact that such a definition is equivalent to
the one given in Example \ref{Definition:NG} is proved in \cite%
{ben_yaacov_fraisse_2015}. Similarly, the original notion of noncommutative
Gurarij space considered by Oikhberg in \cite{oikhberg_non-commutative_2006}
and proved to be unique in \cite{lupini_uniqueness_2016} is the following: a
separable exact operator space $\mathbb{\mathbb{NG}}$ such that for every
finite-dimensional exact operator spaces $E\subseteq F$, linear complete
contraction $\phi :E\rightarrow \mathbb{NG}$, and $\varepsilon >0$, there
exists an extension $\hat{\phi}:F\rightarrow \mathbb{G}$ satisfying $||\hat{%
\phi}||_{\mathrm{cb}}{}<1+\varepsilon $. The fact that such a definition is
equivalent to the one given in Example \ref{Definition:NG} is established in 
\cite{lupini_uniqueness_2016}.

\subsection{The approximate Ramsey property and extreme amenability}

\label{Subs:ARP} Fix $q,s\in \mathbb{N}$. The goal of this part is to prove
that the classes $\mathbb{I}$ of finite-dimensional exact operator spaces
introduced above satisfy the stable Ramsey property with modulus $\varpi
\left( \delta \right) =\delta $. Since $\mathbb{I}$ and its generated class $%
\langle \mathbb{I}\rangle $ are both stable Fra\"{i}ss\'{e} classes, Proposition \ref{SRP=cSRP} tell us that it
suffices to prove the discrete approximate Ramsey property of each $\mathbb{I%
}$.

\begin{lemma}[ARP of $\{\ell_\infty^d(M_{q,s})\}_{d\in \mathbb{N}}$]
\label{Proposition:direct-ramsey-spaces}Fix $q,s\in \mathbb{N}$. For every $%
d,m,r\in \mathbb{N}$ and $\varepsilon >0$ there exists $n\in \mathbb{N}$
such that every $r$-coloring of $\mathrm{\mathrm{Emb}}( \ell ^{d}_{\infty }(
M_{q,s}) ,\ell^{n}_{\infty }( M_{q,s}) ) $ has an $\varepsilon$%
-monochromatic set of the form $\gamma \circ \Emb( \ell^{d}_{\infty }(
M_{q,s}) ,\ell ^{m}_{\infty }( M_{q,s}) ) $ for some $\gamma \in \mathrm{%
\mathrm{Emb}}( \ell^{m}_{\infty }( M_{q,s}) ,\ell^{n}_{\infty }( M_{q,s}) ) $%
.
\end{lemma}

Similarly as it was done for the proof of the (ARP) of the class of Banach
spaces $\{\ell_\infty^n\}_n$ in \cite{bartosova_2019}, rather than proving
Lemma \ref{Proposition:direct-ramsey-spaces} directly, we will establish its
natural \emph{dual statement}, which is Lemma \ref%
{Proposition:dual-ramsey-spaces} below. Given two operator spaces $X$ and $Y$%
, let $\mathrm{C\mathrm{Q}}(X,Y)$ be the set of \emph{completely contractive
quotient} mappings $\phi :X\rightarrow Y$, i.e.\ $\phi :X\rightarrow Y$ such
that each amplification $\phi ^{(n)}$ is a contractive quotient mapping \cite%
[\S 2.2]{effros_operator_2000}. Notice that this is equivalent to the
assertion that the dual map $\phi ^{\ast }:Y^{\ast }\rightarrow X^{\ast }$
is a completely isometric embedding. We denote by $T_{s,q}$ the operator
space dual of $M_{q,s}$. This can be identified as the space of $s\times q$
matrices with matrix norms given by the normalized Hilbert--Schmidt norm 
\cite[\S 1.2]{effros_operator_2000}. The duality between $M_{q,s}$ and $%
T_{s,q}$ is implemented by the paring $(\alpha ,\beta )\mapsto \mathrm{Tr}%
(\alpha \beta )$, there $\mathrm{Tr}$ denotes the normalized trace for $%
s\times s$ matrices. One can then canonically identify the operator space
dual of $\ell _{\infty }^{d}(M_{q,s})$ with the $1$-sum $\ell
_{1}^{d}(T_{s,q})$ of $d$ copies of $T_{s,q}$. From this it is easy to see
that Lemma \ref{Proposition:direct-ramsey-spaces} and Lemma \ref%
{Proposition:dual-ramsey-spaces} below are equivalent by duality.

\begin{lemma}
\label{Proposition:dual-ramsey-spaces}Fix $q,s\in \mathbb{N}$. For every $%
d,m,r\in \mathbb{N}$ and $\varepsilon >0$ there exists $n\in \mathbb{N}$
such that every $r$-coloring of $\mathrm{CQ}( \ell ^{n}_{1}( T_{s,q}) ,\ell
^{d}_{1}( T_{s,q}) ) $ has an $\varepsilon $-monochromatic set of the form $%
\mathrm{CQ}( \ell ^{m}_{1}( T_{s,q}) ,\ell^{d}_{1}( T_{s,q}) ) \circ \gamma $
for some $\gamma \in \mathrm{CQ}( \ell ^{n}_{1}( T_{s,q}) ,\ell ^{m}_{1}(
T_{s,q}) ) $.
\end{lemma}

In order to prove Lemma \ref{Proposition:dual-ramsey-spaces} we will need
the following fact about linear complete isometries, which is an immediate
consequence of \cite[Lemma 5.17]{lupini_uniqueness_2016}; see also \cite[%
Lemma 3.6]{eckhardt_perturbations_2010}.

\begin{lemma}
\label{Lemma:injective} Suppose that $q,s,q_{1},s_{1},\dots ,q_{n},s_{n}\in 
\mathbb{N}$, and $\phi _{i}:M_{q,s}\rightarrow M_{q_{i},s_{i}}$ are
completely contractive linear maps for $i=1,2,\ldots ,n$. Then the linear
map $\phi :M_{q,s}\rightarrow M_{q_{1},s_{1}}\oplus _{\infty }\cdots \oplus
_{\infty }M_{q_{n},s_{n}}$, $x\mapsto (\phi _{1}(x),\ldots ,\phi _{n}(x))$
is a complete isometry if and only if $\phi _{i}$ is a complete isometry for
some $i\leq n$.
\end{lemma}

The proof of Lemma \ref{Proposition:dual-ramsey-spaces} relies on the Dual
Ramsey Theorem \ref{Theorem:DLT-rigid} by Graham and Rothschild.

\begin{proof}[Proof of Lemma \protect\ref{Proposition:dual-ramsey-spaces}]
Fix $d,m,r\in \mathbb{N}$ and $\varepsilon >0$. We identify a linear map $%
\phi $ from $\ell _{1}^{n}(T_{s,q})$ to $\ell _{1}^{d}(T_{s,q})$ with a $%
d\times n$ matrix $\left[ \phi _{ij}\right] $ where $\phi
_{ij}:T_{s,q}\rightarrow T_{s,q}$ is a linear map. It follows from (the dual
of) Lemma \ref{Lemma:injective} that $\phi $ is a completely contractive
quotient mapping if and only if every row of $\left[ \phi _{ij}\right] $ has
an entry that is a surjective complete isometry of $T_{s,q}$, and every
column is a complete contraction from $T_{s,q}$ to $\ell _{1}^{d}(T_{s,q})$.
This implies that if a column has an entry that is a surjective complete
isometry of $T_{s,q}$, then all the other entries of the column are zero.
Let now $\mathcal{P}$ be a finite set of complete contractions from $T_{s,q}$
to $\ell _{1}^{d}(T_{s,q})$---which we regard as $d$-dimensional column
vectors with entries from $T_{s,q}$---with the following properties:

\begin{enumerate}[(i)]

\item the zero map belongs to $\mathcal{P}$;

\item for every $i\leq d$ the canonical embedding of $T_{s,q}$ into the $i$%
-th coordinate of $\ell _{1}^{d}(T_{s,q})$ belongs to $\mathcal{P}$;

\item for every nonzero complete contraction $\phi :T_{s,q}\rightarrow \ell
_{1}^{d}(T_{s,q})$ there exists a nonzero element $\phi _{0}$ of $\mathcal{P}
$ such that $\left\Vert \phi -\phi _{0}\right\Vert _{\mathrm{cb}%
}<\varepsilon $ and $\left\Vert \phi _{0}\right\Vert _{\mathrm{cb}}
<\left\Vert \phi \right\Vert _{\mathrm{cb}}$.
\end{enumerate}

Fix $\varepsilon _{0}>0$ small enough and a finite $\varepsilon _{0}$-dense
subset $U$ of the group of automorphisms of $T_{s,q}$. Let $\mathcal{Q}$ be
the (finite) set of linear complete isometries from $\ell _{1}^{d}(T_{s,q})$
to $\ell _{1}^{m}(T_{s,q})$ such that every row contains at most one nonzero
entry, every column exactly one nonzero entry, and every nonzero entry is an
automorphism of $T_{s,q}$ that belongs to $U$. Fix any linear order $<$ on $%
\mathcal{Q}$, and a linear order on $\mathcal{P}$ with the property that 
\begin{equation}
\text{$\phi <\phi ^{\prime }$ whenever $\left\Vert \phi \right\Vert _{%
\mathrm{cb} }<\left\Vert \phi ^{\prime }\right\Vert _{\mathrm{cb}}$.}
\end{equation}
Endow $\mathcal{Q}\times \mathcal{P}$ with the corresponding
antilexicographic order. An element of $\mathrm{Epi}(n,\mathcal{P})$ is an $n
$-tuple $\overline{v}=(v_{1},\ldots ,v_{n})$ of elements of $\mathcal{P}$.
We associate with such an $n$-tuple the element $\alpha _{\overline{v}}$ of $%
\mathrm{CQ}(\ell _{1}^{n}(T_{s,q}),\ell _{1}^{d}(T_{s,q}))$ whose
representative matrix has $v_{i}$ as $i$-th column for $i=1,2,\ldots ,n$.
Similarly an element of $\mathrm{Epi}(n,\mathcal{Q}\times \mathcal{P})$ is
an $n$-tuple $(\overline{B},\overline{w})=(B_{1},w_{1},\ldots ,B_{n},w_{n})$%
. We associate with such an $n$-tuple the element $\alpha _{(\overline{B},%
\overline{w})}$ of $\mathrm{CQ}(\ell _{1}^{n}(T_{s,q}),\ell
_{1}^{m}(T_{s,q}))$ with $B_{i}w_{i}$ as $i$-th column for $i=1,2,\ldots ,n$%
. Suppose now that $n\in \mathbb{N}$ is obtained from $\mathcal{P}$ and $%
\mathcal{Q}\times \mathcal{P}$ by applying the dual Ramsey theorem, Theorem %
\ref{Theorem:DLT-rigid}. We claim that such an $n$ satisfies the desired
conclusions. Suppose that $c$ is an $r$-coloring of $\mathrm{CQ}(\ell
_{1}^{n}(T_{s,q}),\ell _{1}^{d}(T_{s,q}))$. The map $\overline{v}\mapsto
\alpha _{\overline{v}}$ from $\mathrm{Epi}(n,\mathcal{P})$ to $\mathrm{CQ}%
(\ell _{1}^{n}(T_{s,q}),\ell _{1}^{d}(T_{s,q}))$ induces an $r$-coloring on $%
\mathrm{Epi}(n,\mathcal{P})$. By the choice of $n$ there exists an element $(%
\overline{B},\overline{w})$ of $\mathrm{Epi}(n,\mathcal{Q}\times \mathcal{P}%
) $ such that any rigid surjection from $n$ to $\mathcal{P}$ that factors
through $(\overline{B},\overline{w})$ has a fixed color $i\in r$. To
conclude the proof it remains to show the following.

\begin{claim}
The set of completely contractive quotient mappings from $\ell
_{1}^{n}(T_{s,q})$ to $\ell _{1}^{d}(T_{s,q})$ that factor through $\alpha
_{(\overline{B},\overline{w})} $ is $\varepsilon $-monochromatic.
\end{claim}

\begin{proof}[Proof of Claim:]
This will follow once we show that, given any $\rho \in \mathrm{CQ}(\ell
_{1}^{m}(T_{s,q}),\ell _{1}^{d}(T_{s,q}))$ there exists $\tau \in \mathrm{Epi%
}(\mathcal{Q}\times \mathcal{P},\mathcal{P})$ such that $\left\Vert \alpha
_{\tau (\overline{B},\overline{w})}-\rho \circ \alpha _{(\overline{B},%
\overline{w})}\right\Vert _{\mathrm{cb}}\leq \varepsilon $, where we denoted
by $\tau (\overline{B},\overline{w})$ the element $(\tau
(B_{1},w_{1}),\ldots ,\tau (B_{n},w_{n}))$ of $\mathrm{Epi}(n,\mathcal{P})$.
If $\rho $ has representative matrix $A$, then this is equivalent to the
assertion that, for every $i\leq n$, $\tau (B_{i},w_{i})$ has \textrm{cb}%
-distance at most $\varepsilon $ from $AB_{i}w_{i}$. We proceed to define
such a rigid surjection $\tau $ from $\mathcal{Q}\times \mathcal{P}$ to $%
\mathcal{P}$. By the structure of completely contractive quotient mappings
from $\ell _{1}^{m}(T_{s,q})$ to $\ell _{1}^{d}(T_{s,q})$ recalled above,
there exists $A^{\dagger }\in \mathcal{Q}$ such that $\left\Vert AA^{\dagger
}-\mathrm{Id}_{\ell _{1}^{d}(T_{s,q})}\right\Vert _{\mathrm{cb}}\leq
\varepsilon $, provided that $\varepsilon _{0}$ is small enough (depending
only from $\varepsilon $). Define now $\tau :\mathcal{Q}\times \mathcal{P}%
\rightarrow \mathcal{P}$ by letting, for $B\in \mathcal{Q}$ and $w\in 
\mathcal{P}$, $\tau (B,w)=0$ if $w=0$, $\tau (B,w)=w$ if $B=A^{\dagger }$,
and otherwise $\tau (B,w)\in \mathcal{P}$ such that $0<\left\Vert \tau
(B,w)\right\Vert _{\mathrm{cb}}<\left\Vert ABw\right\Vert _{\mathrm{cb}}$
and $\left\Vert \tau (B,w)-ABw\right\Vert _{\mathrm{cb}}<\varepsilon $. It
is clear from the definition that $\tau (B,w)$ has distance at most $%
\varepsilon $ from $ABw$. We need to verify that $\tau $ is indeed a rigid
surjection from $\mathcal{Q}\times \mathcal{P}$ to $\mathcal{P}$. Observe
that $\tau $ is onto, and the pairs $(B,0)$ are the only elements of $%
\mathcal{Q}\times \mathcal{P}$ that are mapped by $\tau $ to zero. It is
therefore enough to prove that, for every $w\in \mathcal{P}$, $(A^{\dagger
},w)$ is the minimum of the preimage of $w$ under $\tau $. Suppose that $%
(B^{\prime },w^{\prime })$ is an element of the preimage of $w$ under $\tau $%
. Then by definition of $\tau $ we have that%
\begin{equation*}
\left\Vert w\right\Vert _{\mathrm{cb}}<\left\Vert ABw^{\prime }\right\Vert _{%
\mathrm{cb}}\leq \left\Vert w^{\prime }\right\Vert _{\mathrm{cb}}\text{.}
\end{equation*}%
By our assumptions on the ordering of $\mathcal{P}$, it follows that $%
w<w^{\prime }$ and hence $(A^{\dagger },w)<(B^{\prime },w^{\prime })$. This
concludes the proof. \qedhere \mbox{ }\qedsymbol 
\end{proof}

\let\qed\relax
\end{proof}

Using the general facts about the approximate Ramsey property from
Proposition \ref{SRP=cSRP}, one can bootstrap the approximate Ramsey
property from the class considered in Lemma \ref%
{Proposition:dual-ramsey-spaces} to other classes of operator spaces. In
fact one can obtain the compact stable Ramsey property with modulus $\varpi
\left( \delta \right) =\delta $; see Definition \ref{Definition:ARP}.

\begin{theorem}
\label{Theorem:ARP-spaces}The following classes of finite-dimensional
operator spaces satisfy the compact (SRP) with modulus $\varpi (\delta
)=\delta $:

\begin{enumerate}[(1)]  

\item for every $q\in \mathbb{N}$, the injective class $\mathbb{I}_{q}$ and
the class $\langle \mathbb{I}_q\rangle$ of finite-dimensional $q$-minimal
operator spaces, and in particular the class of finite-dimensional Banach
spaces.

\item The injective class $\mathbb{I}_{\mathrm{c}}$ and the class $\langle 
\mathbb{I}_\mathrm{c}\rangle$ of finite-dimensional operator sequence spaces.

\item The classes $\mathbb{I}_{\mathrm{e}}$, $\mathbb{I}_{\mathrm{inj}}$ and
the class $\langle \mathbb{I}_\mathrm{e}\rangle=\langle \mathbb{I}_\mathrm{%
inj}\rangle$ of finite-dimensional exact operator spaces.

\end{enumerate}
\end{theorem}

\begin{proof}
In each of the cases, it is enough to verify that the given class satisfies
the discrete (ARP) in view of Proposition \ref{ARP=DARP} and the equivalence
of \textit{(1)} and \textit{(5)} in Proposition \ref{SRP=cSRP}.

\noindent\textit{(1)}: It follows from Lemma \ref%
{Proposition:direct-ramsey-spaces} for $q=r$ that $\mathbb{I}_{q}$ satisfies
the discrete (ARP). 

\noindent \textit{(2)}: Lemma \ref{Proposition:direct-ramsey-spaces} for $s=1$
is the discrete (ARP) of $\mathbb{I}_\mathrm{c}$. 

\noindent \textit{(3)}: We verify that the class of finite dimensional exact
operator spaces $\langle \mathbb{I}_{\mathrm{e}}\rangle$ satisfies the
(ARP). By the equivalence of \textit{(1)} and \textit{(6)} in Proposition \ref%
{SRP=cSRP}, it suffices to show that for every positive integer $p,q,r$ such
that $p\leq q$, and every $\varepsilon >0$ there is some finite-dimensional
exact operator space $Z$ such that every $r$-coloring of $\mathrm{Emb}%
(M_{p},Z)$ $\varepsilon $-stabilizes on $\gamma \circ \mathrm{Emb}%
(M_{p},M_{q})$ for some $\gamma \in \mathrm{Emb}(M_{q},Z)$. Now, $M_{p}$ and 
$M_{q}$ are $q$-minimal, so by \textit{(1)} there is such a $Z$ which is a
finite-dimensional $q$-minimal operator space. Since every $q$-minimal
operator space is exact, this concludes that $\langle \mathbb{I}_{\mathrm{e}%
}\rangle$ satisfies the (ARP). It follows from the equivalence \textit{(1)}
and \textit{(5)} in Proposition \ref{SRP=cSRP} that both $\mathbb{I}_\mathrm{e%
}$ and $\mathbb{I}_\mathrm{inj}$ also satisfies the (ARP). 
\end{proof}

Theorem \ref{Theorem:ea-spaces} below extends the work on the Gurarij space
from \cite[Theorem 2.5]{bartosova_2019}. It is a consequence of the KPT
correspondence as in Proposition \ref{Proposition:ARP} and corresponding
approximate Ramsey properties in Theorem \ref{Theorem:ARP-spaces}. It shows
that the limits of the Fra\"{\i}ss\'{e} classes mentioned in Definition \ref%
{Definition:NG} have extremely amenable automorphism groups.

\begin{theorem}
\label{Theorem:ea-spaces} The following operator spaces have extremely
amenable automorphism groups:

\begin{enumerate}[(1)]

\item each $q$-minimal Gurarij space $\mathbb{G}_{q}$, and in particular the
Gurarij space $\mathbb{G}$.

\item The column Gurarij space $\mathbb{CG}$.

\item The noncommutative Gurarij space $\mathbb{NG}$.
\end{enumerate}
\end{theorem}

\begin{remark}
One can also give a direct, quantitative proof of the ARP for
finite-dimensional\ Banach spaces using Lemma \ref%
{Proposition:dual-ramsey-spaces} when $s=q=1$ and the injective envelope
construction \cite[Subsection 4.2]{blecher_operator_2004}. Such a proof
yields an explicit bound of the Ramsey numbers for the class of
finite-dimensional\ Banach spaces in terms of the Ramsey numbers for the
Dual Ramsey Theorem. Furthermore the same argument applies with no change in
the case of real Banach spaces, yielding extreme amenability of the group of
surjective linear isometries of the real Gurarij space (see \cite[Appendix A]%
{bartosova_2019_1}).
\end{remark}

\subsection{Operator spaces with a distinguished functional\label%
{Sec:functional}}

\subsubsection{Fra\"{\i}ss\'{e} limits of $R$-operator spaces}

The natural geometric object associated with a Banach space $X$ is the unit
ball $\mathrm{Ball}(X^{\ast })$ of the dual of $X$, i.e.\ the $w^{\ast }$%
-compact convex symmetric set of contractive linear functionals on $X$. As
discussed in the introduction, the noncommutative analog of such a
correspondence involves the notion of matrix functionals. Given an operator
space $X$, a matrix functional on $X$ is a linear function from $X$ to $%
M_{q,r}(\C)$ for some $q,r\in \mathbb{N}$. The space $\mathrm{C\mathrm{Ball}}%
(X^{\ast })$ is the sequence $(K_{q,r})_{q,r\in \mathbb{N}}$, where $K_{q,r}$
is the $w^{\ast }$-compact convex set of completely contractive matrix
functionals from $X$ to $M_{q,r}(\C)$. The space $\mathrm{C\mathrm{Ball}}%
(X^{\ast })$ is endowed with a notion of rectangular matrix convex
combinations that makes it a compact rectangular matrix convex set \cite%
{fuller_boundary_2018}. Furthermore any compact rectangular convex set
arises in this way. It should be clear from this that matrix functionals
provide the right noncommutative analog of functionals on Banach spaces.

More generally, suppose that $R$ is a separable $\emph{nuclear}$ operator
space, that is, the identity map of $R$ is the pointwise limit of completely
contractive maps that factor through finite-dimensional injective operator
spaces. When $R$ is in addition a minimal operator space (i.e.\ a Banach
space), this is equivalent to the assertion that $R$ is a \emph{Lindenstrauss%
} Banach space \cite[Subsection 8.6.4]{blecher_operator_2004}. A classical
result of Wojtaszczyk \cite{wojtaszczyk_remarks_1972} asserts that the
separable Lindenstrauss spaces are precisely the separable Banach spaces
that are isometric to the range of a contractive projection on the Gurarij
space $\mathbb{G}$. The noncommutative analog of such a result asserts that
the separable nuclear operator spaces are precisely the separable operator
spaces that are completely isometric to the range of a completely
contractive projection on the noncommutative Gurarij space \cite%
{lupini_operator_2015}. A similar result holds for operator sequence spaces
in terms of the column Gurarij space \cite[Subsection 6.5]%
{lupini_fraisse_2018}. Notice that injective \emph{finite-dimensional}
operator spaces are always nuclear, but the converse does not hold. The
following result can be found in \cite[\S 6.5, \S 6.6]{lupini_fraisse_2018}.

\begin{proposition}
\label{nuclear_injective} Let $R$ be a separable operator space.

\begin{enumerate}[(1)]

\item Suppose that  $R$ is $q$-minimal. Then $R$ is nuclear if and only if it is $\mathbb{I}_{q}$-nuclear, i.e.\ the identity
on $R$ is the pointwise limit of completely contractive maps that factor
through some space in $\mathbb{I}_{q}$. In particular, nuclear $q$-minimal separable operator spaces belong to $[\mbb I_q]$. 

\item Suppose that $R$ is an operator sequence space. Then $R$ is nuclear if and only it is $\mathbb{I}_{\mathrm{c}}$-nuclear, i.e.\ the identity on $R$ is the pointwise limit of completely contractive maps that factor through some space in $\mathbb{I}_{\mathrm{c}}$.   In particular, nuclear  separable  sequence operator spaces belong to $[\mbb I_c]$. 
\end{enumerate}
\end{proposition}

\begin{definition}[$R$-functionals]
For an operator space $X$ and a separable nuclear operator space $R$, an $R$%
\emph{-functional} on $X$ is a completely bounded linear operator from $X$
to $R$. Let $\mathrm{CC}(X,R)$ be the space of completely contractive $R$%
-functionals on $X$, regarded as a Polish space with the topology of
pointwise convergence. Let $\mathrm{Aut}(X)\curvearrowright \mathrm{CC}(X,R)$
be the continuous action $(\alpha ,s)\mapsto s\circ \alpha ^{-1}$. Finally,
given $s\in \mathrm{CC}(X,R)$, let $\mathrm{Aut}(X,s)\subseteq \mathrm{Aut}%
(X)$ be the stabilizer of $s$ with respect to such an action.
\end{definition}

{Recall that given a family $\mathcal{A}$ of operator spaces and a separable
operator space $R$, let $\mathcal{A}^{R}$ be the collection of $R$-operator
spaces $\boldsymbol{X}=(X,s_{X})$ where $X\in \mathcal{A}$ and $%
s_{X}:X\rightarrow R$ is a complete contraction. }

\begin{proposition}
\label{Proposition:fraisse-functional}{Let $\mathbb{I}$ be an injective
class of operator spaces as in Definition \ref{matrix_classes}, and let $R$ be a separable operator space. Then}

\begin{enumerate}[(1)]
\item {$\mathbb{I}^{R}$ is a stable amalgamation class with stability
modulus $\varpi (\delta )=2\delta $ when $R\in \mathbb{I}$.}

\item $\langle \mathbb{I}\rangle ^{R}$ is a stable Fra\"{\i}ss\'{e} class
with stability modulus $\varpi (\delta )=2\delta $.

\item {Suppose that in addition $R$ is nuclear. Then the Fra\"{\i}ss\'{e} limit of $\langle 
\mathbb{I}\rangle ^{R}$ is the $R$-operator space $(\flim\mathbb{I},\Omega _{%
\flim\mathbb{I}}^{R})$.}
\end{enumerate}
\end{proposition}

\begin{proof}
This is a consequence of Proposition \ref{oij32irj23jr32}. Obviously both 
$\mathbb{I}^{R}$ and $\langle \mathbb{I}\rangle ^{R}$ have a minimal
element, so it suffices to show that these classes satisfy the (SAP) with
modulus $2\delta $. Every finite-dimensional subspace of $R$ embeds into an element of $%
\langle \mathbb{I}\rangle $. Consequently, \textit{(1)} and \textit{(2)} immediately
follow from Proposition \ref{oij32irj23jr32} \textit{(1)}. \textit{(3)} is
also consequence of Proposition \ref{oij32irj23jr32} and the fact that $\flim%
\mathbb{I}=\flim\langle \mathbb{I}\rangle $.
\end{proof}

The $R$-functional $\Omega _{\flim\mathbb{I}}^{R}$ as in Proposition \ref%
{Proposition:fraisse-functional} is called the \emph{generic} completely
contractive $R$-functional on $\flim\mathbb{I}$. The name is justified by
the fact that the $\mathrm{Aut}(\flim\mathbb{I})$-orbit of $\Omega _{\flim%
\mathbb{I}}^{R}$ is a dense $G_{\delta }$ subset of the space $\mathrm{CC}(%
\flim\mathbb{I},R)$ of completely contractive $R$-functionals on $\flim%
\mathbb{I}$.

\subsubsection{KPT correspondence and the approximate Ramsey property of $R$%
-operator spaces}

\label{Subs:extreme-functional} \label{Subs:ARP-functional}

We present the approximate Ramsey properties of several classes of $R$%
-operator spaces, and the corresponding extreme amenability of the
automorphism group of their Fra\"\i ss\'e limits.

\begin{theorem}
\label{Theorem:ARP-R-spaces} The following classes of finite-dimensional $R$%
-operator spaces satisfy the compact (SRP) with stability modulus $\varpi
(\delta )=2\delta $:

\begin{enumerate}[(1)]

\item for a $q$-minimal separable nuclear operator space $R$, the class $%
\mathbb{I}_q^R$, if $R$ is finite dimensional,  and the class $\left\langle \mathbb{I}_{q}\right\rangle ^{R}$
of finite-dimensional $q$-minimal $R$-operator spaces; in particular, for a
separable Lindenstrauss space $R$, the class of $R$-Banach spaces (which
recovers \cite[Theorem 2.41]{bartosova_2019}).

\item For a separable nuclear operator sequence space $R$, the class $%
\mathbb{I}_\mathrm{c}^R$, if $R$ is finite dimensional,  and the class $\left\langle \mathbb{I}_{\mathrm{c}%
}\right\rangle ^{R}$ of finite-dimensional $R$-operator sequence spaces.

\item For a separable nuclear operator space $R$, the class $\mathbb{I}_%
\mathrm{e}^R$, if $R$ is finite dimensional,  and the class $\left\langle \mathbb{I}_{\mathrm{e}%
}\right\rangle ^{R}$ of finite-dimensional exact $R$-operator spaces. 
\end{enumerate}
\end{theorem}

\begin{proof}
We have seen in Theorem \ref{Theorem:ARP-spaces} that the injective classes $\mbb I_q$, $\mbb I_c$, $\mbb  I_\mr{e}$ and $\mbb I_\mr{inj}$, and the corresponding completions $\langle\mbb I_q\rangle$, $\langle\mbb I_c\rangle$ and $\langle\mbb  I_\mr{e}\rangle=\langle\mbb I_\mr{inj}\rangle$ satisfy the  compact (SRP) with modulus $\de$, so it follows from   Proposition \ref{oij32irj23jr32weewe} {\it(1)} that the corresponding $R$-classes also satisfy the compact (SRP) with modulus $\varpi(\de)=2\de$ (here we are using that for classes satisfying the (SAP) with that modulus, the compact (SRP) and the (SRP) are equivalent). 
\end{proof}

From Theorem \ref{Theorem:ARP-R-spaces} and the characterization of extreme
amenability in Proposition \ref{Proposition:ARP} we obtain new extremely
amenable groups, extending the work done in \cite[Theorem 2.39]%
{bartosova_2019} for the Gurarij space, and the trivial space $R=\{0\}$ done
above in Theorem \ref{Theorem:ea-spaces}.

\begin{corollary}
\label{Corollary:ea-R-spaces}The following Polish groups are extremely
amenable:

\begin{enumerate}[(1)]

\item The stabilizer of the generic contractive $R$-functional on the
Gurarij space for any separable Lindenstrauss Banach space $R$.

\item The stabilizer of the generic completely contractive $R$-functional on
the $q$-minimal Gurarij space for any separable $q$-minimal nuclear operator
space $R$.

\item The stabilizer of the generic completely contractive $R$-functional on
the column Gurarij space for any separable nuclear operator sequence $R$.

\item The stabilizer of the generic completely contractive $R$-functional on
the noncommutative Gurarij space for any separable nuclear operator space $R$.
\end{enumerate}
\end{corollary}

Again, the same proof shows that (1) of Corollary \ref{Corollary:ea-R-spaces}
also holds when one considers the real Gurarij space and any real separable
Lindenstrauss space $R$.

\section{The Ramsey property of noncommutative Choquet simplices and
operator systems \label{Sec:systems}}

\label{iojiorjiewjwe5}

In this section we establish the approximate (dual) Ramsey property for
noncommutative Choquet simplices with a distinguished point. We will then
apply this to compute the universal minimal flows of the automorphisms group
of the noncommutative Poulsen simplex. This will be done by studying
operator systems with a distinguished ucp map to a fixed nuclear separable
operator system $R$ ($R$-operator systems).

\subsection{Choquet simplices and operator systems\label%
{Subs:operator_systems}}

The correspondence between compact convex sets and function systems admits a
natural noncommutative generalization. A \emph{compact matrix convex set }is
a sequence $\boldsymbol{K}=(K_{n})$ of sets $K_{n}\subset M_{n}(V)$ for some
topological vector space $V$ that is matrix convex \cite[Definition 1.1]%
{webster_krein-milman_1999}. This means that whenever $\alpha _{i}\in
M_{q_{i},q}$ and $v_{i}\in K_{q_{i}}$ are such that $\alpha _{1}^{\ast
}\alpha _{1}+\cdots +\alpha _{q}^{\ast }\alpha _{q}=1$, then the \emph{%
matrix convex combination} $\alpha _{1}^{\ast }v_{1}\alpha _{1}+\cdots
+\alpha _{q}^{\ast }v_{q}\alpha _{q}$ belongs to $K_{q}$. A continuous
matrix affine function $\boldsymbol{\phi }:\boldsymbol{K}\rightarrow 
\boldsymbol{T}$ between compact matrix convex sets is a sequence of
continuous functions $\phi _{n}:K_{n}\rightarrow T_{n}$ that is matrix
affine in the sense that it preserves matrix convex combinations. The group $%
\mathrm{Aut}(\boldsymbol{K})$ of matrix affine homeomorphisms of $%
\boldsymbol{K}$ is a Polish group when endowed with the compact-open
topology.

To each operator system $X$ one can canonically assign a compact matrix
convex set: the \emph{matrix state space }$\boldsymbol{S}(X)$. This is the
sequence $(S_{n}(X))$, where $S_{n}(X)\subset M_{n}(X^{\ast })$ is the space 
of all ucp maps from $X$ to $M_{n}$. Conversely, to a compact convex set $%
\boldsymbol{K}$ one can associate an operator system $A(\boldsymbol{K})$ of
matrix-affine functions from $\boldsymbol{K}$ to $\mathbb{R}$. It is proved
in \cite[Section 3]{webster_krein-milman_1999} that these constructions are
the inverse of each other, and define an equivalence between the category of
operator systems and ucp maps, and the category of compact matrix convex
sets and continuous matrix affine functions. In particular if $X$ is an
operator system, then the group $\mathrm{Aut}(X)$ of surjective unital
complete isometries on $X$ can be identified with the group $\mathrm{Aut}(%
\boldsymbol{K})$ of matrix affine homeomorphisms of the matrix state space $%
\boldsymbol{K}$ of $X$. The notions of matrix extreme point and matrix
extreme boundary can be defined in the setting of compact matrix convex sets
by using matrix convex combinations \cite{webster_krein-milman_1999}.

Recall that an operator system $X$ is called \emph{nuclear} if its identity
map is the pointwise limit of ucp maps that factor through
finite-dimensional injective operator systems. When $X=A(K)$, this is
equivalent to the assertion that the state space $K$ of $X$ is a Choquet
simplex. The matrix state spaces of nuclear operator systems can be seen as
the noncommutative generalization of Choquet simplices. The natural
noncommutative analog of the Poulsen simplex is studied in \cite%
{lupini_fraisse_2018}, where it is proved that finite-dimensional exact
operator systems form a Fra\"{\i}ss\'{e} class. The matrix state space $%
\mathbb{\mathbb{NP}}=(\mathbb{NP}_{n})$ of the corresponding Fra\"{\i}ss\'{e}
limit $A(\mathbb{NP})$ is a nontrivial noncommutative Choquet simplex with
dense matrix extreme boundary, which is called the \emph{noncommutative
Poulsen simplex }in \cite{lupini_fraisse_2018, lupini_kirchberg_2018}.

One can also define a sequence of structures $(\mathbb{P}^{(q)})$ for $q\in 
\mathbb{N}$ that interpolates between the Poulsen simplex and the
noncommutative Poulsen simplex, in the context of $q$-minimal operator
systems. An operator system is $q$-minimal if it admits a complete order
embedding into unital C*-algebra $C(K,M_{q})$ for some compact Hausdorff
space $K$ \cite{xhabli_super_2012}. Here we regard the unital selfadjoint
subspaces of $C(K,M_{q})$ as operator systems, called $q$-minimal operator
systems or $M_{q}$-systems. For $q=1$, these are precisely the function
systems. A $q$-minimal operator system $X$ can be completely recovered from
the portion of the matrix state space only consisting of $S_{k}(X)$ for $%
k=1,2,\ldots ,q$. Conversely a sequence $(K_{1},\ldots ,K_{q})$ of compact
convex sets $K_{j}\subset M_{j}(V)$ closed under matrix convex combinations $%
\alpha _{1}^{\ast }v_{1}\alpha _{1}+\cdots +\alpha _{n}^{\ast }v_{n}\alpha
_{n}$ for $\alpha _{i}\in M_{q_{i}q}$ and $v_{i}\in M_{q_{i}}$ and $%
q_{i}\leq q$ such that $\alpha _{1}^{\ast }\alpha _{1}+\cdots +\alpha
_{n}^{\ast }\alpha _{n}=1$, uniquely determines a $q$-minimal operator
system $A(K_{1},\ldots ,K_{q})$. The finite-dimensional $q$-minimal operator
systems form a Fra\"{\i}ss\'{e} class \cite[Section 6.7]{lupini_fraisse_2018}%
. The matrix state space $\mathbb{P}^{(q)}=(\mathbb{P}_{1}^{(q)},\ldots ,%
\mathbb{P}_{q}^{(q)})$ of the corresponding limit $A(\mathbb{P}^{(q)})$ is
the $q$-minimal Poulsen simplex. The model-theoretic properties of $A(%
\mathbb{P})$, $A(\mathbb{\mathbb{NP}})$, and $A(\mathbb{P}^{(q)})$ have been
studied in \cite{goldbring_model-theoretic_2015}.

As we mentioned in the Subsection \ref{jio3i4jorwe3223}, we regard operator
systems as objects of the category $\mathtt{Osy}$ which has ucp maps as
morphisms. The finite-dimensional \emph{injective} objects in this category
are precisely the finite $\infty $-sums of copies of $M_{q}$, which are also
the finite-dimensional C*-algebras. The notion of isomorphism in this
category coincides with complete order isomorphism. The Gromov-Hausdorff
pseudometric of two finite-dimensional operator systems $X,Y$ is the infimum
of $\varepsilon >0$ such that there exist ucp maps $f:X\rightarrow Y$ and $%
g:Y\rightarrow X$ such that $\left\Vert g\circ f-\mathrm{Id}_{X}\right\Vert
_{\mathrm{cb}}<\varepsilon $ and $\left\Vert f\circ g-\mathrm{Id}%
_{Y}\right\Vert _{\mathrm{cb}}<\varepsilon $. If $X$ and $Y$ are operator
systems, then $\mathrm{UCP}(X,Y)$ is the space of ucp maps from $X$ to $Y$,
and the automorphism group $\mathrm{Aut}(X)$ is the group of surjective
unital complete isometries from $X$ to itself. Both $\mathrm{Aut}(X)$ and $%
\mathrm{UCP}(X,Y)$ are Polish spaces when endowed with the topology of
pointwise convergence. There is a natural continuous action of $\mathrm{%
\mathrm{Aut}}(X)$ on $\mathrm{UCP}(X,Y)$ defined by $(\alpha ,s)\mapsto
s\circ \alpha ^{-1}$. In particular when $Y=M_{q}$ we have that $\mathrm{UCP}%
(X,Y)=S_{q}(X)$.

Recall that given a class of operator systems $\mathcal{A}$, we denote by $[%
\mathcal{A}]$ the collection of operator systems $E$ such that every
finite-dimensional operator system $X\subseteq E$ is a limit (with respect
to the Gromov-Hausdorff pseudometric) of \emph{subspaces} of operators
systems in $\mathcal{A}$, and by $\langle \mathcal{A}\rangle $ the class of
finite-dimensional operator systems in $[\mathcal{A}]$.

\begin{definition}[Injective classes]
\label{matrix_classes_osy}We say that a family of finite-dimensional
operator systems is an \emph{injective class} of operator systems if it is
one of the families $\{\ell _{\infty }^{n}\}_{n\in \mathbb{N}}$, $\{\ell
_{\infty }^{n}(M_{q})\}_{n\in \mathbb{N}}$, or $\{M_{q}\}_{q\in \mathbb{N}}$.
\end{definition}

As mentioned in Theorem \ref{injective_classes_are_fraisse}, all the classes
of operator systems considered in Definition \ref{matrix_classes_osy} are
stable Fra\"{\i}ss\'{e} classes with modulus $\varpi \left( \delta \right)
=2\delta $. The corresponding generating classes are well-known.

\begin{definition}[Spaces locally approximated by injective classes]
\label{Definition:exact_osy} \mbox{}

\begin{enumerate}[(a)]

\item $[\{\ell _{\infty }^{n}\}_{n\in \mathbb{N}}]$ is the class of \emph{%
function systems}.

\item $[\{\ell _{\infty }^{n}(M_{q})\}_{n\in \mathbb{N}}]$ is the class of $%
q $-\emph{minimal operator systems} (see\emph{\ }\cite{xhabli_super_2012}).

\item $[\{M_{q}\}_{q\in \mathbb{N}}]$ is the class of \emph{exact operator
systems} (see \cite{kavruk_quotients_2013}).
\end{enumerate}
\end{definition}

And the corresponding limits are the following.

\begin{example}
\label{Definition:P} \mbox{}

\begin{enumerate}[$\bullet$]

\item $\flim\{\ell _{\infty }^{n}\}_{n\in \mathbb{N}}$ is the function
system $A(\mathbb{P})$ associated with the Poulsen simplex $\mathbb{P}$ (see 
\cite[Section 6.3]{lupini_fraisse_2018}).

\item $\flim\{M_{q}\}_{q\in \mathbb{N}}$ is the operator system $A(\mathbb{NP%
})$ associated with the noncommutative Poulsen simplex $\mathbb{\mathbb{NP}}$
(see \cite[Section 8.2]{lupini_fraisse_2018}).

\item $\flim\{\ell _{\infty }^{n}(M_{q})\}_{n\in \mathbb{N}}$ is the
operator system $A(\mathbb{P}^{(q)})$ associated with the $q$-minimal
Poulsen simplex $\mathbb{P}^{(q)}$ (see \cite[Section 6.7]%
{lupini_fraisse_2018}).
\end{enumerate}
\end{example}

The main goal of this section is to compute the universal minimal flow of
the group $\mathrm{Aut}(\mathbb{\mathbb{NP}}) $ of matrix affine
homeomorphisms of the noncommutative Poulsen simplex $\mathbb{NP}$,
extending the work in \cite[Theorem 3.10]{bartosova_2019} for its
commutative version, the Poulsen simplex $\mathbb{P}$. Precisely, we will
prove that the universal minimal compact $\mathrm{Aut}(\mathbb{NP})$-space
is the canonical action of $\mathrm{Aut}(\mathbb{NP})$ on the space $\mathbb{%
NP}_{1}$ of (scalar) states on $A(\mathbb{\mathbb{NP}})$.

Similarly as in the case of Banach spaces and operator spaces (Section \ref%
{Sec:functional}), we need to consider operator systems with a distinguished
(matrix) state. Suppose that $X$ is an operator system. Recall that a \emph{%
state} on $X$ is a ucp map from $X$ to $\mathbb{C}$. More generally, an $%
M_{n}$-state is a ucp map from $X$ to $M_{n}$. Even more generally, if $R$
is any separable nuclear operator system, we call a ucp map from $X$ to $R$
an $R$-state on $X$. As observed above, the space $\mathrm{UCP}(X,R)$ of $R$%
-states on $X$ is a Polish space endowed with a canonical continuous action
of $\mathrm{Aut}(X)$. An \emph{$R$-operator system} is a pair $\mathbf{X}%
=(X,s_{X})$ of a operator system $X$ and an $R$-state $s_{X}$ on $X$. In the
following, we regard $\mathrm{UCP}\left( X,R\right) $ as an $\mathrm{Aut}%
\left( X\right) $ space with respect to the canonical action $\mathrm{Aut}%
\left( X\right) \curvearrowright \mathrm{UCP}\left( X,R\right) $ given by $%
\left( \alpha ,s\right) \mapsto s\circ \alpha ^{-1}$. We let $\mathrm{Aut}%
(X,s_{X})$ be the stabilizer of $s_{X}\in \mathrm{UCP}\left( X,R\right) $ in 
$\Aut(X)$. Given a family $\mathcal{A}$ of operator systems, let $\mathcal{A}%
^{R}$ be the collection of $R$-operator spaces $(X,s_{X})$ where $X\in 
\mathcal{A}$. 

\begin{proposition}
\label{Proposition:fraisse-state}
Let $\mathbb{I}$ be an injective class of
operator systems as in Definition \ref{matrix_classes_osy}, and suppose that $R$ is a separable operator system such that  $R\in [\mbb I]$. 

\begin{enumerate}[(1)]

\item$\mathbb{I}^{R}$  is a stable amalgamation class with stability modulus $\varpi (\delta )=3\delta $.

\item $\langle \mathbb{I}\rangle ^{R}$ is a stable Fra%
\"{\i}ss\'{e} class with stability modulus $\varpi (\delta )=3\delta $.

\item Suppose in addition that $R$ is nuclear. Then   the Fra\"{\i}ss\'{e} limit of $\langle\mathbb{I}\rangle^{R}$ is the $R$-operator
system $(\flim\mathbb{I},\Omega _{\flim\mathbb{I}}^{R})$.
\end{enumerate}
\end{proposition}

As in the case of operator spaces, the $R$-state $\Omega _{\flim\mathbb{I}%
}^{R}$ as in Proposition \ref{Proposition:fraisse-state} is called the \emph{%
generic} $R$-state on $\flim\mathbb{I}$. This is the unique $R$-state on $%
\Omega _{\flim\mathbb{I}}^{R}$ whose $\mathrm{Aut}(\flim\mathbb{I})$-orbit
is a dense $G_{\delta }$ subset of the space $\mathrm{UCP}(\flim\mathbb{I}%
,R) $.

\subsection{Approximate Ramsey property and extreme amenability}

For the rest of this section we fix $q,k\in \mathbb{N}$. We identify as in
Subsection \ref{Subs:ARP} the dual of $\ell _{\infty }^{d}(M_{q})$ with $%
\ell _{1}^{d}(T_{q})$. We denote by \textrm{Tr }the canonical normalized
trace of $q\times q$ matrices. The isomorphism between $T_{q}$ and the dual
of $M_{q}$ is induced by the pairing%
\begin{equation*}
\left\langle \alpha ,\beta \right\rangle \mapsto \mathrm{\mathrm{Tr}}\left(
\alpha \beta \right) \text{.}
\end{equation*}%
A linear map $\eta :M_{q}\rightarrow M_{q}$ is unital completely positive
(ucp) if and only if its dual $\eta ^{\ast }:T_{q}\rightarrow T_{q}$ is
trace-preserving and completely positive. (Such maps are called \emph{%
quantum channels }in the quantum information theory literature; see \cite[\S %
4.1]{gupta_functional_2015}.) Thus, $\eta $ is a complete order embedding if
and only if $\eta ^{\ast }$ is a trace-preserving completely positive
completely contractive quotient mapping.

Let us consider now $\ell _{\infty }^{d}\left( M_{q}\right) $. Every state $%
\lambda $ on $\ell _{\infty }^{d}$, which can be seen as a positive element $%
\left( \lambda _{1},\ldots ,\lambda _{d}\right) $ of $\ell _{1}^{d}$ of norm 
$1$, induces a normalized trace $\mathrm{\mathrm{Tr}}_{\lambda }$ on $\ell
_{1}^{d}\left( T_{q}\right) $ defined by%
\begin{equation*}
\mathrm{\mathrm{Tr}}_{\lambda }\left( a_{1},\ldots ,a_{d}\right) =\lambda
_{1}\mathrm{\mathrm{Tr}}\left( a_{1}\right) +\cdots +\lambda _{d}\mathrm{%
\mathrm{Tr}}\left( a_{d}\right) \text{.}
\end{equation*}%
We also let \textrm{Tr }be the trace on $\ell _{1}^{d}\left( T_{q}\right) $
defined by%
\begin{equation*}
\mathrm{\mathrm{Tr}}\left( a_{1},\ldots ,a_{d}\right) =\mathrm{\mathrm{%
\mathrm{T}}r}\left( a_{1}\right) +\cdots +\mathrm{\mathrm{Tr}}\left(
a_{d}\right) \text{.}
\end{equation*}

\begin{definition}
Adopting the notations above, for $n,d,q\geq 1$ we say that a linear map $%
\phi :\ell _{1}^{n}\left( M_{q}\right) \rightarrow \ell _{1}^{d}\left(
M_{q}\right) $ is \emph{trace-preserving} if, for every state $\lambda $ on $%
\ell _{\infty }^{d}$, and for every $a\in \ell _{1}^{n}\left( M_{q}\right) $%
, $\mathrm{\mathrm{Tr}}_{\lambda }\left( \phi \left( a\right) \right) =%
\mathrm{\mathrm{Tr}}\left( a\right) $.
\end{definition}

It is easy to see that a linear map $\eta :\ell _{\infty }^{d}\left(
M_{q}\right) \rightarrow \ell _{\infty }^{n}\left( M_{q}\right) $ is
completely positive and unital if and only if its dual map $\eta ^{\ast
}:\ell _{1}^{n}\left( M_{q}\right) \rightarrow \ell _{1}^{d}\left(
M_{q}\right) $ is completely positive and trace-preserving. Thus, $\eta $ is
a complete order embedding if and only if $\eta ^{\ast }$ is a
trace-preserving completely positive complete quotient mapping.

Let $\sigma _{d}$ be the $M_{q}$-state on $\ell _{\infty }^{d}(M_{q})$
mapping $(x_{1},\ldots ,x_{d})$ to $x_{d}$. A linear map $\eta :\ell
_{\infty }^{d}(M_{q})\rightarrow \ell _{\infty }^{n}(M_{q})$ has the
property that $\sigma _{n}\circ \eta =\sigma _{d}$ if and only if $\eta
^{\ast }(0,\ldots ,0,x)=(0,\ldots ,0,x)$ for every $x\in T_{q}$. We denote
by $\mathrm{TPCQ}^{M_{q}}(\ell _{1}^{n}(T_{q}),\ell _{1}^{d}(T_{q}))$ the
space of trace-preserving completely positive completely contractive
quotient mappings $\phi $ from $\ell _{1}^{n}(T_{q})$ to $\ell
_{1}^{d}(T_{q})$ such that $\phi (0,\ldots ,0,x)=(0,\ldots ,0,x)$ for every $%
x\in T_{q}$.

\begin{lemma}
\label{Lemma:perturb-ucp}Suppose that $\psi _{1},\ldots ,\psi _{d-1},\phi
_{d}:M_{q}\rightarrow M_{q}$ are completely positive linear maps such that $%
\left\Vert y-1\right\Vert <\varepsilon $, where $y=\psi _{1}(1)+\cdots +\psi
_{d-1}(1)+\phi _{d}(1)$. Then there exists a completely positive map $\psi
_{d}:M_{q}\rightarrow M_{q}$ such that%
\begin{equation*}
\psi _{1}(1)+\cdots +\psi _{d-1}(1)+\psi _{d}(1)=1
\end{equation*}%
and $\left\Vert \psi _{d}-\phi _{d}\right\Vert <\varepsilon $.
\end{lemma}

\begin{proof}
Fix any state $s$ on $M_{q}$ and define $\psi _{d}( x) =\phi _{d}( x) +s( x)
( 1-y) $.
\end{proof}

\begin{proposition}
\label{Proposition:Ramsey-systems}Fix $q\in \mathbb{N}$. For any $d,m,r\in 
\mathbb{N}$ and $\varepsilon >0$ there exists $n\in \mathbb{N}$ such that
for any $r$-coloring of $\mathrm{TPCQ}^{M_{q}}(\ell _{1}^{n}(T_{q}),\ell
_{1}^{d}(T_{q}))$ there exists $\gamma \in \mathrm{TPCQ}^{M_{q}}(\ell
_{1}^{n}(T_{q}),\ell _{1}^{m}(T_{q}))$ such that $\mathrm{TPCQ}^{M_{q}}(\ell
_{1}^{m}(T_{q}),\ell _{1}^{d}(T_{q}))\circ \gamma $ is $\varepsilon $%
-monochromatic.
\end{proposition}

\begin{proof}
The proof is analogous to the proof of Lemma \ref%
{Proposition:dual-ramsey-spaces}. Fix $d,m,r\in \mathbb{N}$ and $\varepsilon
>0$. We identify a linear map $\phi $ from $\ell _{1}^{n}(M_{q})$ to $\ell
_{1}^{d}(M_{q})$ with a $d\times n$ matrix $\left[ \phi _{ij}\right] $ where 
$\phi _{ij}:T_{q}\rightarrow T_{q}$ is a linear map. It follows from (the
dual of) Lemma \ref{Lemma:injective} that $\phi \in \mathrm{TPCQ}%
^{M_{q}}(\ell _{1}^{n}(T_{q}),\ell _{1}^{d}(T_{q}))$ if and only if

\begin{enumerate}[$\bullet$]

\item every row of $\left[ \phi _{ij}\right] $ has an entry that is an
automorphism of $T_{q}$,

\item every column is a trace-preserving completely positive completely
contractive map from $T_{q}$ to $\ell _{1}^{d}(T_{q})$,

\item the last column is $( 0,0,\ldots ,0,\mathrm{Id}_{T_{q}}) $, where $%
\mathrm{Id}_{T_{q}}$ is the identity map of $T_{q}$.
\end{enumerate}

Fix $\varepsilon _{0}\in (0,\varepsilon )$ small enough, and a finite $%
\varepsilon _{0}$-dense subset $U$ of the group of automorphisms of $T_{q}$
containing the identity map of $T_{q}$. The dual statement of Lemma \ref%
{Lemma:perturb-ucp} and the small perturbation lemma \cite[Lemma 2.13.2]%
{pisier_introduction_2003} show that one can find a finite set $\mathcal{P}$
of trace-preserving completely positive completely contractive maps from $%
T_{q}$ to $\ell _{1}^{d}(T_{q})$ with the following properties:

\begin{enumerate}[(1)]

\item for every $i\leq d$ the canonical embedding of $T_{q}$ into the $i$-th
coordinate of $\ell^{d}_{1}( T_{q}) $ belongs to $\mathcal{P}$.

\item For every trace-preserving completely positive completely contractive
map $v=(v_{1},\ldots ,v_{d}):T_{q}\rightarrow \ell _{1}^{d}(T_{q})$ such
that $(v_{1},\ldots ,v_{d-1})$ is nonzero, there is a trace-preserving
completely positive completely contractive map $w=(w_{1},\ldots ,w_{d})$ in $%
\mathcal{P}$ such that $\left\Vert w-v\right\Vert _{\mathrm{cb}}<\varepsilon
_{0}$, $(w_{1},\ldots ,w_{d-1})$ is nonzero, and $\left\Vert (w_{1},\ldots
,w_{d-1})\right\Vert _{\mathrm{cb}}<\left\Vert (v_{1},\ldots
,v_{d-1})\right\Vert _{\mathrm{cb}}$.
\end{enumerate}

Let $\mathcal{Q}$ be the (finite) set of trace-preserving completely
positive completely contractive quotient mappings from $\ell _{1}^{d}(T_{q})$
to $\ell _{1}^{m}(T_{q})$ such that the last row is $(0,0,\ldots ,\mathrm{Id}%
_{T_{q}})$, every column contains exactly one nonzero entry, every row
contains at most one nonzero entry, and every nonzero entry is an
automorphism of $T_{q}$ that belongs to $U$. Fix any linear order on $%
\mathcal{Q}$, and a linear order on $\mathcal{P}$ with the property that $%
v<w $ whenever $\left\Vert (v_{1},\ldots ,v_{d-1})\right\Vert _{\mathrm{cb}%
}<\left\Vert (w_{1},\ldots ,w_{d-1})\right\Vert _{\mathrm{cb}}$. Endow $%
\mathcal{Q}\times \mathcal{P}$ with the corresponding antilexicographic
order. Suppose now that $n\in \mathbb{N}$ is obtained from $\mathcal{P}$ and 
$\mathcal{Q}\times \mathcal{P}$ by applying the dual Ramsey Theorem \ref%
{Theorem:DLT-rigid}. We claim that $n+1$ satisfies the desired conclusions.
An element of $\mathrm{Epi}(n,\mathcal{P})$ is a tuple $\overline{v}%
=(v^{(1)},\ldots ,v^{(n)})$ of elements of $\mathcal{P}$. We associate with
such a tuple the element $\alpha _{\overline{v}}$ of $\mathrm{TPCQ}%
^{M_{q}}(\ell _{1}^{n+1}(T_{q}),\ell _{1}^{d}(T_{q}))$ whose $i$-th column
is $v^{(i)}$ for $i=1,2,\ldots ,n$ and the $(n+1)$-th column is $(0,0,\ldots
,\mathrm{Id}_{T_{q}})$. Similarly an element of $\mathrm{Epi}(n,\mathcal{Q}%
\times \mathcal{P})$ is an $n$-tuple $(\overline{B},\overline{v}%
)=(B_{1},v_{1},\ldots ,B_{n},v_{n})$. We associate with such a tuple the
completely positive completely contractive quotient mapping $\alpha _{(\bar{B%
},\bar{v})}$ from $\ell _{1}^{n+1}(T_{q})$ to $\ell _{1}^{m}(T_{q})$ whose $%
i $-th column is $B_{i}v_{i}$ for $i\leq n$, and $(0,0,\ldots ,0,\mathrm{Id}%
_{T_{q}})$ for $i=n+1$. Suppose that $c$ is an $r$-coloring of $\mathrm{TPCQ}%
^{M_{q}}(\ell _{1}^{n+1}(T_{q}),\ell _{1}^{d}(T_{q}))$. The identification
of $\mathrm{Epi}(n,\mathcal{P})$ with a subspace of $\mathrm{TPCQ}%
^{M_{q}}(\ell _{1}^{n+1}(T_{q}),\ell _{1}^{d}(T_{q}))$ described above
induces an $r$-coloring on $\mathrm{Epi}(n,\mathcal{P})$. By the choice of $%
n $ there exists an element $(\overline{B},\overline{w})$ of $\mathrm{Epi}(n,%
\mathcal{Q}\times \mathcal{P})$ such that any rigid surjection from $n$ to $%
\mathcal{P}$ that factors through $(\overline{B},\overline{w})$ has a fixed
color $i\in r$. To conclude the proof, it remains to show that the set of
elements of $\mathrm{TPCQ}^{M_{q}}(\ell _{1}^{n+1}(T_{q}),\ell
_{1}^{d}(T_{q}))$ that factor through $\alpha _{(\overline{B},\overline{w})}$
is $\varepsilon $-monochromatic. By our choice of $n$ this will follow once
we show that, given any $\rho \in \mathrm{TPCQ}^{M_{q}}(\ell
_{1}^{m}(T_{q}),\ell _{1}^{d}(T_{q}))$, there exists $\tau \in \mathrm{Epi}(%
\mathcal{Q}\times \mathcal{P},\mathcal{P})$ such that $\left\Vert \alpha
_{\tau \circ (\overline{B},\overline{w})}-\rho \circ \alpha _{(\overline{B},%
\overline{w})}\right\Vert _{\mathrm{cb}}\leq \varepsilon $. Here we denoted
by $\tau \circ (\overline{B},\overline{w})$ the rigid surjection from $n$ to 
$\mathcal{P}$ that one obtains by composing $(\overline{B},\overline{w})$%
---regarded as a rigid surjection---and $\tau $. If $\rho $ has
representative matrix $A$, this is equivalent to the assertion that for
every $i\leq n$, $\left\Vert AB_{i}w_{i}-\tau (B_{i},w_{i})\right\Vert _{%
\mathrm{cb}}\leq \varepsilon $. We proceed to define such a rigid surjection 
$\tau $ from $\mathcal{Q}\times \mathcal{P}$ to $\mathcal{P}$. By the
structure of completely positive completely contractive quotient mappings
from $\ell _{1}^{m}(T_{q})$ to $\ell _{1}^{d}(T_{q})$ recalled above, there
exists $A^{\dagger }\in \mathcal{Q}$ such that $\left\Vert AA^{\dagger }-%
\mathrm{Id}_{\ell ^{1}(T_{q})}\right\Vert _{\mathrm{cb}}\leq \varepsilon $,
provided that $\varepsilon _{0}$ is small enough. Define now $\tau :\mathcal{%
Q}\times \mathcal{P}\rightarrow \mathcal{P}$ by letting, for $B\in \mathcal{Q%
}$ and $w=\left( w_{1},\ldots ,w_{d}\right) \in \mathcal{P}$, if $%
ABw=v=\left( v_{1},\ldots ,v_{d}\right) $, $\tau \left( B,w\right) :=%
\widetilde{v}=(\widetilde{v}_{1},\ldots ,\widetilde{v}_{d})$ such that:

\begin{enumerate}[$\bullet$]

\item if $B=A^{\dagger }$ or if $v=\left( 0,0,\ldots ,0,v_{d}\right) $, then 
$\widetilde{v}=w$;

\item otherwise, $\widetilde{v}$ is an element of $\mathcal{P}$ such that $%
\left\Vert \widetilde{v}-v\right\Vert _{\mathrm{cb}}\leq \varepsilon $, $(%
\widetilde{v}_{1},\ldots ,\widetilde{v}_{d-1})$ is nonzero, and we have that $\left\Vert (%
\widetilde{v}_{1},\ldots ,\widetilde{v}_{d-1})\right\Vert _{\mathrm{cb}} $ $%
<\left\Vert (v_{1},\ldots ,v_{d-1})\right\Vert _{\mathrm{cb}}$.
\end{enumerate}

It is clear from the definition that $\left\Vert \tau (B,w)-ABw\right\Vert _{%
\mathrm{cb}}\leq \varepsilon $ for every $(B,w)\in \mathcal{Q}\times 
\mathcal{P}$. We need to verify that $\tau $ is indeed a rigid surjection
from $\mathcal{Q}\times \mathcal{P}$ to $\mathcal{P}$. Observe that $\tau $
is onto. Fix $\widetilde{v}=\left( \widetilde{v}_{1},\ldots ,\widetilde{v}%
_{d}\right) \in \mathcal{P}$. If $\widetilde{v}=\left( 0,\ldots ,0,%
\widetilde{v}_{d}\right) $, then the least element of $\mathcal{Q}\times 
\mathcal{P}$ that is mapped by $\tau $ to $\widetilde{v}$ is $\left( B,%
\widetilde{v}\right) $, where $B$ is the least element of $\mathcal{Q}$. If $%
(\widetilde{v}_{1},\ldots ,\widetilde{v}_{d-1})$ is not the zero vector,
then the least element of $\mathcal{Q}\times \mathcal{P}$ that is mapped by $%
\tau $ to $\widetilde{v}$ is $\left( A^{\dagger },\widetilde{v}\right) $.
Indeed, suppose that $\left( B,w\right) $ is an element of the preimage of $%
\widetilde{v}$ under $\tau $ such that $B$ is different from $A^{\dagger }$.
Set $ABw=v=\left( v_{1},\ldots ,v_{d}\right) $. By definition of $\tau $ we
have that%
\begin{equation*}
\left\Vert (\widetilde{v}_{1},\ldots ,\widetilde{v}_{d-1})\right\Vert _{%
\mathrm{cb}} <\left\Vert \left( v_{1},\ldots ,v_{d-1}\right) \right\Vert _{%
\mathrm{cb}} \leq \left\Vert \left( w_{1},\ldots ,w_{d-1}\right) \right\Vert
_{\mathrm{cb}}\text{.}
\end{equation*}%
Therefore by definition of the order on $\mathcal{P}$ and on $\mathcal{Q}%
\times \mathcal{P}$ we have that $\widetilde{v}<w$ and $(A^{\dagger },%
\widetilde{w})<\left( B,w\right) $. This concludes the proof that the least
element of $\mathcal{Q}\times \mathcal{P}$ that is mapped by $\tau $ to $%
\widetilde{v}$ is $(A^{\dagger },\widetilde{v})$. These remarks clearly
imply that $\tau $ is a rigid surjection.
\end{proof}

The following result can be proved from Proposition \ref%
{Proposition:Ramsey-systems} similarly as Theorem \ref{Theorem:ARP-R-spaces}.

\begin{theorem}
\label{Theorem:ARP-R-systems} The following classes of finite-dimensional $R$%
-operator systems satisfy the stable Ramsey property with modulus $\varpi
(\delta )=3\delta $:

\begin{enumerate}[(1)]

\item for every $q\in \mathbb{N}$ the class $\{(\ell_\infty^d(M_q),s_d)\}_{d%
\in \mathbb{N}}$ of $M_{q}$-operator systems, where $s _{d}( x_{1},\ldots
,x_{d}) =x_{d}$.

\item For every $q\in \mathbb{N}$ and $q$-minimal separable nuclear operator
system $R$, the class of finite-dimensional $q$-minimal $R$-operator systems.

\item For every separable nuclear operator system $R$, the class of
finite-dimensional exact $R$-operator systems.
\end{enumerate}
\end{theorem}

The limits of the Fra\"{\i}ss\'{e} classes mentioned in Theorem \ref%
{Theorem:ARP-R-systems} have extremely amenable automorphism groups in view
of the correspondence between extreme amenability and the approximate Ramsey
property given by Proposition \ref{Proposition:ARP}.

\begin{corollary}
\label{Corollary:ea-osy}The following Polish groups are extremely amenable:

\begin{enumerate}[(1)]

\item the stabilizer $\mathrm{Aut}(A(\mathbb{P}),\Omega _{A(\mathbb{P}%
)}^{A(F)})$ of the generic $A(F)$-state $\Omega _{A(\mathbb{P})}^{A(F)}$ on
the Poulsen system $A(\mathbb{P})$ for any metrizable Choquet simplex $F$ 
\cite[Theorem 3.3]{bartosova_2019}.

\item The stabilizer $\mathrm{Aut}(A(\mathbb{\mathbb{NP}}),\Omega _{A(%
\mathbb{\mathbb{NP}})}^{R})$ of the generic $R$-state $\Omega _{A(\mathbb{NP}%
)}^{R}$ on the noncommutative Poulsen system $A(\mathbb{NP})$ for any
separable nuclear operator system $R$.

\item The stabilizer $\mathrm{Aut}(A(\mathbb{P}^{(q)}),\Omega _{A(\mathbb{P}%
^{(q)})}^{R})$ of the generic $R$-state $\Omega _{A(\mathbb{P}^{(q)})}^{R}$
on the $q$-minimal Poulsen system $A(\mathbb{P}^{(q)})$ for any $q$-minimal
nuclear operator system $R$.
\end{enumerate}
\end{corollary}

\subsection{The universal minimal flows of the $\mathrm{Aut}(\mathbb{NP})$}

Using Corollary \ref{Corollary:ea-osy} we can compute the universal minimal
flows of the matrix affine homeomorphism group $\mathrm{Aut}(\mathbb{NP})$
of the noncommutative Poulsen simplex, and the matrix affine homeomorphism
group $\mathrm{Aut}(\mathbb{P}^{(q)})$ of the $q$--minimal Poulsen simplex.
The corresponding result for the Poulsen simplex was obtained in \cite[%
Theorem 3.10]{bartosova_2019}.

\begin{theorem}
\label{Theorem:flows} \mbox{}

\begin{enumerate}[(1)] 

\item The universal minimal flow of $\mathrm{Aut}(\mathbb{NP})$ is the
canonical action $\mathrm{Aut}(\mathbb{NP})\curvearrowright \mathbb{NP}_{1}$.

\item The universal minimal flow of $\mathrm{Aut}(\mathbb{P}^{(q)})$ is the
canonical action $\mathrm{Aut}(\mathbb{P}^{(q)})\curvearrowright \mathbb{P}%
_{1}^{(q)}$.
\end{enumerate}
\end{theorem}

\begin{proof}
\noindent \textit{(1)}: The minimality of the action $\mathrm{Aut}(\mathbb{NP}%
)\curvearrowright \mathbb{NP}_{1}$ is a consequence of the following fact:
for any $d\in \mathbb{N}$ and $\varepsilon >0$ there exists $m\in \mathbb{N}$
such that for any $s\in S(M_{d})$ and $t\in S(M_{m})$ there exists a
complete order embedding $\phi :M_{d}\rightarrow M_{m}$ such that $%
\left\Vert t\circ \phi -s\right\Vert _{\mathrm{cb}}<\varepsilon $; see \cite[%
Lemma 8.10]{lupini_fraisse_2018} and \cite[Proposition 5.8]%
{lupini_fraisse_2018}. Consider the generic state $\Omega _{A(\mathbb{NP})}^{%
\mathbb{C}}$ on $A(\mathbb{NP})$. It is shown in \cite[Section 8.2]%
{lupini_fraisse_2018} that $\Omega _{A\left( \mathbb{NP}\right) }^{\mathbb{C}%
}$ is a matrix extreme point of $\mathbb{NP}$ whose $\mathrm{Aut}\left( 
\mathbb{NP}\right) $-orbit is dense in $\mathbb{NP}_{1}$. The stabilizer $%
\mathrm{Aut}(\mathbb{NP},\Omega _{A(\mathbb{NP})}^{\mathbb{C}})$ of $\Omega
_{A(\mathbb{NP})}^{\mathbb{C}}$ is extremely amenable by Corollary \ref%
{Corollary:ea-osy}. The canonical $\mathrm{Aut}(\mathbb{NP})$-equivariant
map from the quotient $\mathrm{Aut}(\mathbb{NP})$-space $\mathrm{Aut}(%
\mathbb{NP})/\hspace{-0.1cm}/\mathrm{Aut}(\mathbb{NP},\Omega _{A(\mathbb{NP}%
)}^{\mathbb{C}})$ to $\mathbb{NP}_{1}$ is a uniform equivalence. This
follows from the homogeneity property of $(A(\mathbb{NP}),\Omega _{A(\mathbb{%
NP})}^{\mathbb{C}})$ as the Fra\"{\i}ss\'{e} limit of the class of
finite-dimensional operator systems with a distinguished state; see also 
\cite[Subsection 5.4]{lupini_fraisse_2018}. This allows one to conclude via
a standard argument---see \cite[Theorem1.2]{melleray_polish_2016}---that the
action $\mathrm{Aut}(\mathbb{NP})\curvearrowright \mathbb{NP}$ is the
universal minimal compact $\mathrm{Aut}(\mathbb{NP})$-space.

\noindent \textit{(2)}: Minimality of the action $\mathrm{Aut}(\mathbb{P}%
^{(q)})\curvearrowright \mathbb{P}_{1}^{(q)}$ is a consequence of a similar
assertion than in \textit{1)}, where $M_{d}$ and $M_{m}$ are replaced with $%
\ell _{\infty }^{d}(M_{q})$ and $\ell _{\infty }^{m}(M_{q})$; see \cite[%
Lemma 6.25]{lupini_fraisse_2018}. The rest of the argument is entirely
analogous.
\end{proof}

\def\cprime{$'$}


\begin{thebibliography}{10}

\bibitem{alfsen_compact_1971}
Erik~M. Alfsen.
\newblock {\em Compact convex sets and boundary integrals}.
\newblock Springer-Verlag, New York-Heidelberg, 1971.

\bibitem{bartosova_2019_1}
Dana Barto\v{s}ov\'{a}, Jordi Lopez-Abad, Martino Lupini, and Brice Mbombo.
\newblock The {R}amsey properties for {G}rassmannians over {$\mathbb R$},
  {$\mathbb C$}.
\newblock Preprint arXiv:1910.00311, 2019.

\bibitem{bartosova_2019}
Dana Barto\v{s}ov\'{a}, Jordi Lopez-Abad, Martino Lupini, and Brice Mbombo.
\newblock The {R}amsey property for {B}anach spaces and {C}hoquet simplices.
\newblock {\em Journal of the European Mathematical Society}, 2019.
\newblock to appear; arXiv:1708.01317.

\bibitem{ben_yaacov_fraisse_2015}
Ita{\"{i}} Ben~Yaacov.
\newblock Fra{\"{i}}ss{\'{e}} limits of metric structures.
\newblock {\em Journal of Symbolic Logic}, 80(1):100--115, 2015.

\bibitem{ben_yaacov_model_2008}
Ita{\"i} Ben~Yaacov, Alexander Berenstein, C.~Ward Henson, and Alexander
  Usvyatsov.
\newblock Model theory for metric structures.
\newblock In {\em Model theory with applications to algebra and analysis. Vol.
  2}, volume 350 of {\em London Mathematical Society Lecture Note Series},
  pages 315--427. Cambridge University Press, 2008.

\bibitem{blecher_operator_2004}
David~P. Blecher and Christian Le~Merdy.
\newblock {\em Operator algebras and their modules{\textemdash}an operator
  space approach}, volume~30 of {\em London Mathematical Society Monographs.
  New Series}.
\newblock Oxford University Press, Oxford, 2004.

\bibitem{blecher_metric_2011}
David~P. Blecher and Matthew Neal.
\newblock Metric characterizations of isometries and of unital operator spaces
  and systems.
\newblock {\em Proceedings of the American Mathematical Society},
  139(3):985--998, 2011.

\bibitem{choi_injectivity_1977}
Man-Duen Choi and Edward~G. Effros.
\newblock Injectivity and operator spaces.
\newblock {\em Journal of Functional Analysis}, 24(2):156--209, 1977.

\bibitem{davidson_noncommutative_2019}
Kenneth~R. Davidson and Matthew Kennedy.
\newblock Noncommutative {Choquet} theory.
\newblock {\em arXiv:1905.08436}, December 2019.
\newblock arXiv: 1905.08436.

\bibitem{eckhardt_perturbations_2010}
Caleb Eckhardt.
\newblock Perturbations of completely positive maps and strong {NF} algebras.
\newblock {\em Proceedings of the London Mathematical Society},
  101(3):795--820, 2010.

\bibitem{effros_aspects_1978}
Edward~G. Effros.
\newblock Aspects of noncommutative order.
\newblock In {\em C*-algebras and applications to physics ({Proc}. {Second}
  {Japan}-{USA} {Sem}., {Los} {Angeles}, {Calif}., 1977)}, volume 650 of {\em
  Lecture {Notes} in {Mathematics}}, pages 1--40. Springer, Berlin, 1978.

\bibitem{effros_matrix_2009}
Edward~G. Effros.
\newblock A matrix convexity approach to some celebrated quantum inequalities.
\newblock {\em Proceedings of the National Academy of Sciences of the United
  States of America}, 106(4):1006--1008, 2009.

\bibitem{effros_operator_2000}
Edward~G. Effros and Zhong-Jin Ruan.
\newblock {\em Operator spaces}, volume~23 of {\em London Mathematical Society
  Monographs. New Series}.
\newblock Oxford University Press, 2000.

\bibitem{effros_matrix_1997}
Edward~G. Effros and Soren Winkler.
\newblock Matrix convexity: operator analogues of the bipolar and
  {Hahn}-{Banach} theorems.
\newblock {\em Journal of Functional Analysis}, 144(1):117--152, February 1997.

\bibitem{ellis_universal_1960}
Robert Ellis.
\newblock Universal minimal sets.
\newblock {\em Proceedings of the American Mathematical Society}, 11:540--543,
  1960.

\bibitem{farenick_extremal_2000}
Douglas~R. Farenick.
\newblock Extremal matrix states on operator systems.
\newblock {\em Journal of the London Mathematical Society}, 61(3):885--892,
  2000.

\bibitem{farenick_pure_2004}
Douglas~R. Farenick.
\newblock Pure matrix states on operator systems.
\newblock {\em Linear Algebra and its Applications}, 393:149--173, 2004.

\bibitem{ferenczi_amalgamation_2017}
Valentin Ferenczi, Jordi Lopez-Abad, Brice Mbombo, and Stevo Todorcevic.
\newblock Amalgamation and {R}amsey properties of {$L_ p$} spaces.
\newblock {\em Adv. Math.}, 369:107190, 76, 2020.

\bibitem{fuller_boundary_2018}
Adam~H. Fuller, Michael Hartz, and Martino Lupini.
\newblock Boundary representations of operator spaces and compact rectangular
  matrix convex sets.
\newblock {\em Journal of Operator Theory}, 79(1):139--172, 2018.

\bibitem{goldbring_model-theoretic_2015}
Isaac Goldbring and Martino Lupini.
\newblock Model-theoretic aspects of the {G}urarij operator space.
\newblock {\em Israel Journal of Mathematics}.
\newblock in press.

\bibitem{goldbring_kirchbergs_2015}
Isaac Goldbring and Thomas Sinclair.
\newblock On {Kirchberg}'s embedding problem.
\newblock {\em Journal of Functional Analysis}, 269(1):155--198, July 2015.

\bibitem{graham_ramseys_1971}
Ronald~L. Graham and Bruce~L. Rothschild.
\newblock Ramsey's theorem for {$n$}-parameter sets.
\newblock {\em Transactions of the American Mathematical Society}, 159, 1971.

\bibitem{gromov_topological_1983}
Mikhael Gromov and Vitali~D. Milman.
\newblock A topological application of the isoperimetric inequality.
\newblock {\em American Journal of Mathematics}, 105(4):843--854, 1983.

\bibitem{gupta_functional_2015}
Ved~Prakash Gupta, Prabha Mandayam, and V.~S. Sunder.
\newblock {\em The functional analysis of quantum information theory}, volume
  902 of {\em Lecture {Notes} in {Physics}}.
\newblock Springer, Cham, 2015.

\bibitem{gurarij_spaces_1966}
Vladimir~I. Gurari{\u\i}.
\newblock Spaces of universal placement, isotropic spaces and a problem of
  {M}azur on rotations of {B}anach spaces.
\newblock {\em Siberian Mathematical Journal}, 7:1002--1013, 1966.

\bibitem{gutman_new_2013}
Yonatan Gutman and Hanfeng Li.
\newblock A new short proof for the uniqueness of the universal minimal space.
\newblock {\em Proceedings of the American Mathematical Society},
  141(1):265--267, 2013.

\bibitem{helton_tracial_2017}
J.~William Helton, Igor Klep, and Scott McCullough.
\newblock The tracial {Hahn}-{Banach} theorem, polar duals, matrix convex sets,
  and projections of free spectrahedra.
\newblock {\em Journal of the European Mathematical Society}, 19(6):1845--1897,
  2017.

\bibitem{helton_free_2013}
William Helton, Igor Klep, and Scott McCullough.
\newblock Free convex algebraic geometry.
\newblock In {\em Semidefinite optimization and convex algebraic geometry},
  volume~13 of {\em {MOS}-{SIAM} {Ser}. {Optim}.}, pages 341--405. SIAM,
  Philadelphia, PA, 2013.

\bibitem{kavruk_quotients_2013}
Ali~S. Kavruk, Vern~I. Paulsen, Ivan~G. Todorov, and Mark Tomforde.
\newblock Quotients, exactness, and nuclearity in the operator system category.
\newblock {\em Advances in Mathematics}, 235:321--360, 2013.

\bibitem{kechris_fraisse_2005}
Alexander~S. Kechris, Vladimir Pestov, and Stevo Todorcevic.
\newblock Fra{\"i}ss{\'e} limits, {R}amsey theory, and topological dynamics of
  automorphism groups.
\newblock {\em Geometric and Functional Analysis}, 15(1):106--189, 2005.

\bibitem{kirchberg_c*-algebras_1998}
Eberhard Kirchberg and Simon Wassermann.
\newblock C*-algebras generated by operator systems.
\newblock {\em Journal of Functional Analysis}, 155(2):324--351, 1998.

\bibitem{kubis_proof_2013}
Wies{\l}aw Kubi{\'s} and S{\l}awomir Solecki.
\newblock A proof of uniqueness of the {G}urari\u\i\ space.
\newblock {\em Israel Journal of Mathematics}, 195(1):449--456, 2013.

\bibitem{lambert_operatorfolgenraume_2002}
Anselm Lambert.
\newblock {\em Operatorfolgenr{\"{a}}ume}.
\newblock PhD thesis, Universit{\"{a}}t des Saarlandes, 2002.

\bibitem{lehner_mn-espaces_1997}
Franz Lehner.
\newblock {\em {$M_n$}-espaces, sommes d'unitaires et analyse harmonique sur le
  groupe libre}.
\newblock PhD thesis, Universit\'{e} de Paris 6, 1997.

\bibitem{lupini_operator_2015}
Martino Lupini.
\newblock Operator space and operator system analogs of {Kirchberg}'s nuclear
  embedding theorem.
\newblock {\em Journal of Mathematical Analysis and Applications},
  431(1):47--56, 2015.

\bibitem{lupini_uniqueness_2016}
Martino Lupini.
\newblock Uniqueness, universality, and homogeneity of the noncommutative
  {Gurarij} space.
\newblock {\em Advances in Mathematics}, 298:286--324, August 2016.

\bibitem{lupini_fraisse_2018}
Martino Lupini.
\newblock {Fra\"{i}ss\'{e}} limits in functional analysis.
\newblock {\em Advances in Mathematics}, 338:93--174, 2018.

\bibitem{lupini_kirchberg_2018}
Martino Lupini.
\newblock The {Kirchberg}–{Wassermann} operator system is unique.
\newblock {\em Journal of Mathematical Analysis and Applications},
  459(2):1251--1259, 2018.

\bibitem{lusky_gurarij_1976}
Wolfgang Lusky.
\newblock The {G}urarij spaces are unique.
\newblock {\em Archiv der Mathematik}, 27(6):627--635, 1976.

\bibitem{lusky_separable_1977}
Wolfgang Lusky.
\newblock On separable {Lindenstrauss} spaces.
\newblock {\em Journal of Functional Analysis}, 26(2):103--120, October 1977.

\bibitem{lusky_consequences_1978}
Wolfgang Lusky.
\newblock Some consequences of {W}. {R}udin's paper: {``{$L_p$}-isometries and
  equimeasurability''}.
\newblock {\em Indiana University Mathematics Journal}, 27(5):859--866, 1978.

\bibitem{lusky_construction_1979}
Wolfgang Lusky.
\newblock On a construction of {Lindenstrauss} and {Wulbert}.
\newblock {\em Journal of Functional Analysis}, 31(1):42--51, January 1979.

\bibitem{melleray_polish_2016}
Julien Melleray, Lionel Nguyen Van~Th{\'e}, and Todor Tsankov.
\newblock Polish groups with metrizable universal minimal flows.
\newblock {\em International Mathematics Research Notices.}, (5):1285--1307,
  2016.

\bibitem{melleray_extremely_2014}
Julien Melleray and Todor Tsankov.
\newblock Extremely amenable groups via continuous logic.
\newblock {\em {arXiv}:1404.4590}, 2014.

\bibitem{van_the_more_2013}
Lionel Nguyen Van~Thé.
\newblock More on the {Kechris}-{Pestov}-{Todorcevic} correspondence:
  precompact expansions.
\newblock {\em Fundamenta Mathematicae}, 222(1):19--47, 2013.

\bibitem{oikhberg_non-commutative_2006}
Timur Oikhberg.
\newblock The non-commutative {G}urarii space.
\newblock {\em Archiv der Mathematik}, 86(4):356--364, 2006.

\bibitem{palazuelos_survey_2016}
Carlos Palazuelos and Thomas Vidick.
\newblock Survey on {Nonlocal} {Games} and {Operator} {Space} {Theory}.
\newblock {\em Journal of Mathematical Physics}, 57(1):015220, January 2016.

\bibitem{paulsen_completely_2002}
Vern~I. Paulsen.
\newblock {\em Completely bounded maps and operator algebras}, volume~78 of
  {\em Cambridge Studies in Advanced Mathematics}.
\newblock Cambridge University Press, Cambridge, 2002.

\bibitem{pestov_isometry_2007}
Vladimir Pestov.
\newblock The isometry group of the {U}rysohn space as a {L}{\'{e}}vy group.
\newblock 154(10):2173--2184.

\bibitem{pestov_ramsey_2002}
Vladimir Pestov.
\newblock {Ramsey-Milman phenomenon, Urysohn metric spaces, and extremely
  amenable groups}.
\newblock {\em Israel Journal of Mathematics}, 127(1):317--357, 2002.

\bibitem{pestov_dynamics_2006}
Vladimir Pestov.
\newblock {\em Dynamics of Infinite-dimensional Groups}, volume~40 of {\em
  University Lecture Series}.
\newblock American Mathematical Society, Providence, {RI}, 2006.

\bibitem{pisier_exact_1995}
Gilles Pisier.
\newblock Exact operator spaces.
\newblock {\em Ast{\'e}risque}, (232):159--186, 1995.
\newblock Recent advances in operator algebras (Orl{\'e}ans, 1992).

\bibitem{pisier_operator_1996}
Gilles Pisier.
\newblock The operator {H}ilbert space {${\rm OH}$}, complex interpolation and
  tensor norms.
\newblock {\em Memoirs of the American Mathematical Society},
  122(585):viii+103, 1996.

\bibitem{pisier_introduction_2003}
Gilles Pisier.
\newblock {\em Introduction to operator space theory}, volume 294 of {\em
  London Mathematical Society Lecture Note Series}.
\newblock Cambridge University Press, Cambridge, 2003.

\bibitem{poulsen_simplex_1961}
Ebbe~T. Poulsen.
\newblock A simplex with dense extreme points.
\newblock {\em Annales de l'Institut Fourier}, 11:83--87, 1961.

\bibitem{Schechtman_1979}
Gideon Schechtman.
\newblock Almost isometric {$L_{p}$}\ subspaces of {$L_{p}(0,\,1)$}.
\newblock {\em The Journal of the London Mathematical Society}, 20(3):516--528,
  1979.

\bibitem{smith_finite_2000}
Roger~R. Smith.
\newblock Finite dimensional injective operator spaces.
\newblock {\em Proceedings of the American Mathematical Society},
  128(11):3461--3462, 2000.

\bibitem{webster_krein-milman_1999}
Corran Webster and Soren Winkler.
\newblock The {Krein}-{Milman} theorem in operator convexity.
\newblock {\em Transactions of the American Mathematical Society},
  351(1):307--322, 1999.

\bibitem{wojtaszczyk_remarks_1972}
Przemys{\l}aw Wojtaszczyk.
\newblock Some remarks on the {G}urarij space.
\newblock {\em Studia Mathematica}, 41:207--210, 1972.

\bibitem{xhabli_super_2012}
Blerina Xhabli.
\newblock The super operator system structures and their applications in
  quantum entanglement theory.
\newblock {\em Journal of Functional Analysis}, 262(4):1466--1497, 2012.

\end{thebibliography}
\end{document}